\documentclass[11pt]{amsart}
\usepackage{enumerate}
\usepackage{verbatim}
\usepackage{hyperref}
\usepackage{amsthm,amsfonts,amsmath,amscd,amssymb}
\usepackage{xcolor}
\usepackage{graphics,graphicx}
\usepackage{caption}
\usepackage{subcaption}
\let\oldmarginpar\marginpar
\renewcommand\marginpar[1]{\-\oldmarginpar[\raggedleft\footnotesize #1]%
{\raggedright\footnotesize #1}}
\theoremstyle{plain}
\newtheorem{thm}{Theorem}[section]
\newtheorem{lemma}[thm]{Lemma}
\newtheorem{prop}[thm]{Proposition}
\newtheorem{cor}[thm]{Corollary}

\theoremstyle{definition}
\newtheorem{recipe}[thm]{Recipe}
\newtheorem{app}[thm]{Applications}

\newtheorem{definition}[thm]{Definition}
\newtheorem{example}[thm]{Example}
\newtheorem{remark}[thm]{Remark}

\newtheorem{ex}[thm]{Example}
\theoremstyle{remark}

\numberwithin{equation}{section}

\renewcommand{\P}{\mathbb{P}}
\renewcommand{\L}{\mathbb{L}}
\newcommand{\D}{\mathbb{D}}
\newcommand{\N}{\mathbb{N}}
\newcommand{\Z}{\mathbb{Z}}

\newcommand{\R}{\mathbb{R}}
\newcommand{\C}{\mathbb{C}}

\newcommand{\SC}{\mathcal{C}}

\newcommand{\A}{\mathcal{A}}

\renewcommand{\a}{\alpha}
\newcommand{\La}{\Lambda}
\newcommand{\la}{\lambda}
\newcommand{\p}{\varphi}
\newcommand{\e}{\varepsilon}
\newcommand{\dd}{\partial}

\newcommand{\op}{\operatorname}

\newcommand{\sse}{\subseteq}
\newcommand{\lr}{\longrightarrow}

\newcommand{\x}{\times}
\newcommand{\sm}{\setminus}

\newcommand{\im}{\operatorname{im}}
\newcommand{\id}{\operatorname{id}}
\newcommand{\ob}{\operatorname{ob}}
\newcommand{\tb}{\operatorname{tb}}

\newcommand{\Symp}{\operatorname{Symp}}

\newcommand{\sh}{\operatorname{SH}}

\newcommand{\coh}{\operatorname{Coh}}
\newcommand{\wfuk}{\operatorname{WFuk}}
\newcommand{\wt}{\widetilde}

\newcommand{\ol}{\overline}

\newcommand{\lf}{\text{lf}}

\newcommand{\st}{\text{st}}
\newcommand{\std}{\text{st}}

\def\Op{{\mathcal O}{\it p}\,}
\begin{document}
\begin{abstract}
In this article we study Weinstein structures endowed with a Lefschetz fibration in terms of the Legendrian front projection. First we provide a systematic recipe for translating from a Weinstein Lefschetz bifibration to a Legendrian handlebody. Then we present several applications of this technique to symplectic topology. This includes the detection of flexibility and rigidity for several families of Weinstein manifolds and the existence of closed exact Lagrangian submanifolds. In addition, we prove that the Koras--Russell cubic is Stein deformation equivalent to $\C^3$ and verify the affine parts of the algebraic mirrors of two Weinstein $4$--manifolds.
\end{abstract}

\title{Legendrian Fronts for Affine Varieties}

\author{Roger Casals}
\address{Massachusetts Institute of Technology, Department of Mathematics, 77 Massachusetts Avenue Cambridge, MA 02139, USA}
\email{casals@mit.edu}

\author{Emmy Murphy}
\address{Northwestern University, Department of Mathematics, 2033 Sheridan Road Evanston, IL 60208, USA}
\email{e\_murphy@math.northwestern.edu}

\maketitle

\section{Introduction}\label{sec:intro}

In this article we study Weinstein structures in terms of Legendrian front projections.

First, we present a recipe translating from a Weinstein manifold endowed with a Weinstein Lefschetz fibration to a Legendrian handlebody decomposition; these Legendrian handlebodies are the symplectic topology counterpart to the 4--dimensional Kirby diagrams in smooth topology, with the added value that we work with Legendrian fronts in arbitrary dimensions. This recipe is explained in Sections \ref{sec:kirby} and \ref{sec:lef}, and its summary Recipe \ref{dictionary} is the main theorem of the paper, which we then use for a number of applications in Section \ref{sec:app}.\\

In short, Recipe \ref{dictionary} provides a systematic procedure for the understanding of the symplectic topology of affine varieties, which constitute a vast class of symplectic manifolds, and translates the computation of symplectic invariants to the Legendrian context. We present a selected number of applications which illustrate this, but the methods introduced in this article should open the way to the study of many more examples.\\

\begin{figure}[h!]
\centering
  \includegraphics[scale=0.43]{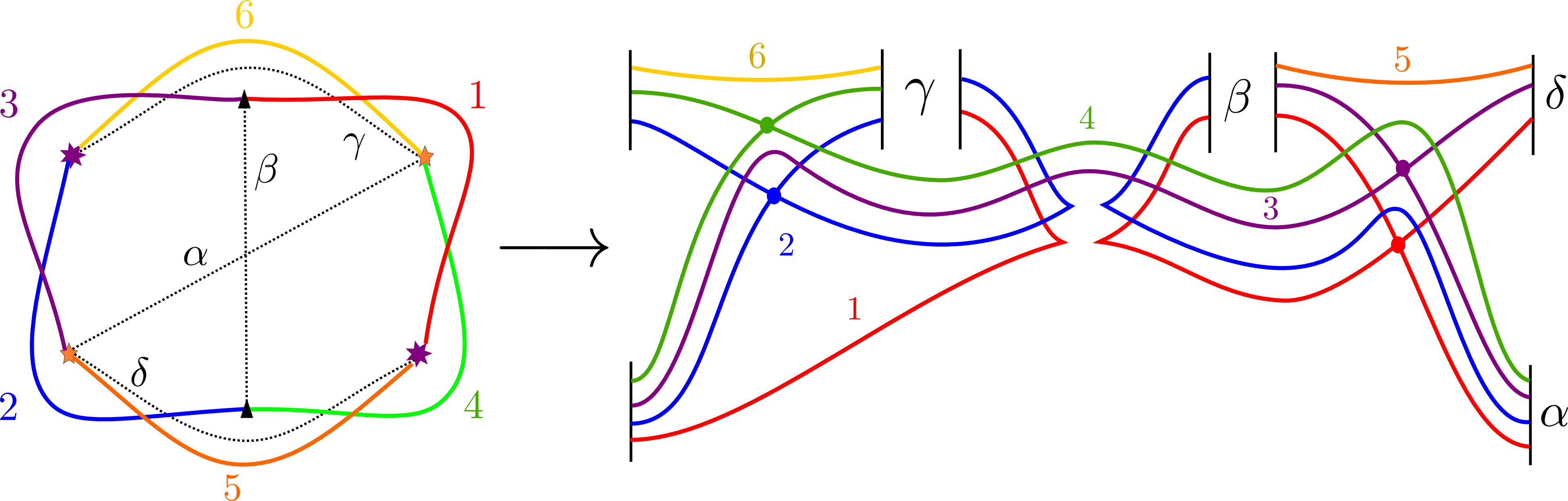}
  \caption{Legendrian Front for the Stein manifold $\{xy^2+w^3+z^2=1\}\sse\C^4$.}
  \label{fig:ex1dict}
\end{figure}

Recipe \ref{dictionary} is presented from the front--end perspective: the reader should be able to use its combinatorial description in an accessible manner, and be a useful tool in the study of specific examples. For instance, the reader might be interested in understanding a certain Stein $6$--manifold, say the affine hypersurface cut--out by the polynomial
$$X=\{(x,y,z,w):xy^2+w^3+z^2=1\}\sse\C^4.$$
Basic questions on the symplectic topology of $X$ are the existence of exact Lagrangian submanifolds and the structure of its symplectic cohomology. Recipe \ref{dictionary} answers both questions, and quite rapidly. There are two steps, in Section \ref{sec:lef} we explain how to translate from this polynomial equation to the Lefschetz bifibration diagram depicted leftmost in Figure \ref{fig:ex1dict}, and then we can directly apply the results in Section \ref{sec:kirby} to produce the Legendrian front depicted rightmost in Figure \ref{fig:ex1dict}.

From the Legendrian handlebody in Figure \ref{fig:ex1dict} we can extract many properties of the example Stein manifold $X$, say $\sh(X)=0$ in any coefficient field, and thus the non--existence of exact Lagrangians. The attentive reader should be able to do these transitions systematically after understanding the content of this article. In particular, the affine manifold in Figure \ref{fig:ex1dict} is intimately related to the Koras--Russell cubic \cite{Zai}, which we will discuss momentarily.

Second, we provide in Section \ref{sec:app} different applications of this dictionary to contact and symplectic topology; these hopefully illustrate the practicality and significant benefits of using Recipe \ref{dictionary} and working from the Legendrian viewpoint. For instance, we will prove the following result as an immediate consequence of the recipe:

\begin{thm}\label{thm:Xab1}
Consider the family of Stein manifolds
$$X^n_{a,b}=\{(x,y,\underline{z}):x^ay^b+\sum_{i=1}^{n-1}z_i^2=1\}\sse\C^{n+1},$$
where $a,b\in\N$ are two coprime integers with $1\leq a<b$.\\
Then we have that
\begin{itemize}
 \item[a.] $X^n_{1,b}$ are flexible for all $b\geq2$.
 \item[b.] $\sh^*(X^2_{a,b})\neq0$ for any $a\geq2$.
 \item[c.] $X^n_{2,b}$ contain exact Lagrangian Klein bottles $S^1\wt\times S^{n-1}$ for any odd $b \geq 3$.
\end{itemize}
\end{thm}

Even better, the statement of Theorem \ref{thm:Xab1} above is a corollary of an explicit Legendrian handlebody, which in addition yields a direct description of the wrapped Fukaya category of the Weinstein manifolds $X_{a,b}$. Note that each of the three conclusions addresses an important aspect of the symplectic topology of these manifolds, and the three of them are direct consequences of the understanding of Legendrian fronts.

\begin{remark}\label{rmk:sh}
The symplectic cohomology $\sh^*(X^n_{a,b})$ for higher $n\geq3$ can be computed with the use of Morse--Bott flow trees, nevertheless this technique is not yet available in the literature and we will only prove the $\sh^*(X^2_{a,b})\neq0$ in this article: it is be apparent from that computation how to proceed for higher $n\geq3$.\hfill$\Box$
\end{remark}

In the complex $3$--dimensional case, Theorem \ref{thm:Xab1} strengthens the results on exotic affine $3$--spheres studied by algebraic geometers \cite{DF} and implies the computation of their symplectic homology in the article \cite{MS}, showing in addition that they are symplectomorphic.

\begin{cor}
The Stein $6$--manifolds $X^3_{1,b}$ are deformation equivalent to $X^3_{1,2}$ for any $b\geq2$.\\
In particular, $X^3_{1,b}$ are exotic affine complex structures on $T^*S^3$ for any $b\geq2$.
\end{cor}
\begin{proof}
Theorem \ref{thm:Xab1} implies that the Stein manifolds $X^3_{1,b}$ are flexible, and since they are diffeomorphic to $S^3 \x \R^3$ with isomorphic complex tangent bundles, they must be Stein deformation equivalent by the $h$--principle \cite{CE}. In particular, the exact symplectic 6--folds $X^3_{1,b}$ contain no exact Lagrangian submanifold, and hence there is no symplectomorphism between $X^3_{1,2}$ and the standard cotangent bundle $X^3_{1,1}\cong(T^*S^3,\la_\st,\p_\st)$.
\end{proof}

\begin{app} Theorem \ref{thm:Xab1} illustrates three central uses of the machinery developed in this work, which we summarize as follows:

\begin{itemize}
 \item[A.] Computation of pseudoholomorphic invariants, such as symplectic homology and, more generally, a model for the wrapped Fukaya category. These invariants can be extracted from the Legendrian contact differential graded algebra \cite{Che,EES} via the generation criterion \cite{Ab} and the surgery exact sequences \cite{BEE,BEE2}.\\
 
 \item[B.] Finding exact Lagrangian submanifolds in Weinstein manifolds. Given the Legendrian handlebody decomposition of a Weinstein manifold we can use combinatorial arguments to construct Lagrangian cobordisms inside the Weinstein manifold \cite[Section 3.3]{ArnoldSing}, and in particular Lagrangian fillings of the attaching Legendrian link will provide closed exact Lagrangians in the Weinstein manifold \cite[Section 4]{BST}.\\
 
 \item[C.] Detection of flexibility \cite{CE} and subflexibility \cite{MuSi} of a Weinstein manifold. These properties are defined in terms of Weinstein handlebodies which we explicitly produce, from which the properties can be directly verified. In the case of flexibility, the $h$-principle produces novel symplectomorphisms between a priori distinct Weinstein manifolds.\\
\end{itemize}

To these three applications, we can add the more familiar arguments from Kirby calculus to the Weinstein setting such as deciding whether two given Weinstein manifolds are deformation equivalent. The front dictionary transforms this into a combinatorial problem of Legendrian isotopies and moves in the front projection, which is potentially more manageable.
\end{app}

\begin{remark}
The systematic computation of the Legendrian differential graded algebra in Application A is not a simple task. For Weinstein 4--folds the resulting Legendrian diagram consists of a Legendrian knot in the contact connected sum $\#^k(S^1\times S^2,\xi_\st)$, for some $k\in\N$, and the computation can be achieved combinatorially following the article \cite{EN}. In the case of a Weinstein $6$--manifold which can be described by a Legendrian surface link in $(S^5,\xi_\st)$, we can use the recent work on the cellular differential graded algebra \cite{RS,RS2}. In general, the computation requires T.~Ekholm's description of the pseudoholomorphic invariants in terms of Morse flow trees \cite{Ek}, combined with a high-dimensional version of \cite{EN}.
\end{remark}

\begin{remark}
Theorems \ref{thm:Xab1}, \ref{thm:torusIntro} and \ref{thm:mirror} illustrate Application B by exhibiting exact Lagrangian submanifolds of affine manifolds. In contrast, Theorem \ref{thm:kr} and Subsection \ref{ssec:subflex} obstruct the existence of exact Lagrangian submanifolds for the involved Weinstein manifolds. This article is also an open invitation to hunt for Lagrangian submanifolds in affine manifolds: for instance, we suggest hunting in the family of affine hypersurfaces presented in Subsection \ref{ssec:Tpqr}, where exact Lagrangian tori are known to exist by different methods \cite{Tpqr}. Subsection \ref{ssec:subflex} below discusses Application C, which will also be momentarily addressed.\hfill$\Box$
\end{remark}

The following theorem is another example of Application B, which is proved in Section \ref{ssec:mirror2}.

\begin{thm}\label{thm:torusIntro}
For any even integer $b\geq 2$ the Weinstein manifold
$$M^n_b = \left\lbrace (x,y,\underline{z}):x(xy^b - 1) + \sum_{i=1}^{n-1}z_i^2 = 0\right\rbrace \sse \C^{n+1}$$
contains an exact Lagrangian $S^1 \x S^{n-1}$.
\end{thm}

Theorem \ref{thm:Xab1} above is also inviting from the perspective of complex singularities since the Weinstein manifolds $X_{a,b}$ arise as Milnor fibers of non--isolated singularities, whose symplectic topology has not been much studied. In this line, algebraic geometers have been proving a wealth of significant results on the algebraic isomorphism types of affine complex varieties \cite{DF,Ram,Zai} but the underlying Stein structures are not quite understood; see \cite{McLean,McLean2,SeSm} for progress in this direction. Being affine hypersurfaces, the results of the present article provide a systematic way of studying these Milnor fibers from the symplectic viewpoint.

In order to continue the study of the symplectic topology of affine manifolds, we focus on the most salient instance of an exotic affine manifold: the Koras--Russell cubic. It is the affine cubic hypersurface defined by
$$\SC:=\{(x,y,z,w):x+x^2y+w^3+z^2=0\}\sse\C^4,$$
and constitutes a wonderful example of an exotic affine 3--space: it is an affine algebraic variety diffeomorphic to $\C^3$ and yet not algebraically isomorphic to it; see the articles \cite{KR,Zai} for more details. As an application of the Legendrian front dictionary we will prove the following theorem:

\begin{thm}\label{thm:kr1}
The Koras--Russell cubic $\SC$ is Stein deformation equivalent to $(\C^3,\la_\st,\p_\st)$.
\end{thm}

This opens the way to the study of affine varieties up to deformation equivalence via the study of explicit Legendrian links, and we encourage the readers to study the underlying Stein structures of many other interesting affine algebraic varieties \cite{Zai}. In particular, the symplectic topology of acyclic surfaces, which started with the article \cite{SeSm}, and other instances of exotic affine $\C^3$, can be the subject of exciting future work.

Third, Theorems \ref{thm:Xab1}, \ref{thm:torusIntro} and \ref{thm:kr1} are applications to symplectic topology, we believe that the methods developed in this article can also be useful in the study of homological mirror symmetry. In order to illustrate this, we present the following theorem :

\begin{thm}\label{thm:mirror}
Consider the symplectic manifolds
$$X=\{(x,y,z):x(xy^2-1)+z^2=0\}\sse\C^3,\qquad Y=\{(x,y,z):xyz+x+z+1=0\}\sse\C^3$$
endowed with their Weinstein structures as submanifolds of $(\C^3,\la_\st,\p_\st)$.\\
Then the algebraic mirrors $\check{X}$ and $\check{Y}$ are algebraically isomorphic to the affine varieties $X$ and $Y$ respectively, i.e. both $X$ and $Y$ are self--mirror manifolds. In addition, the Weinstein $4$--manifolds $X$ and $Y$ both contain exact Lagrangian tori.
\end{thm}

Both symplectic manifolds $X$ and $Y$ have been studied prominently in the literature, and the statement of Theorem \ref{thm:mirror} is not new, only the method of proof. Indeed, the former variety $X$ is algebraically equivalent to the complement of the standard smooth affine conic $\{xy=1\}$ in the complex plane $\C^2[x,y]$, and has been regularly studied in mirror symmetry, see for instance \cite{Au,Pa} and \cite[Section 3]{Au2}. The latter symplectic $4$--manifold $Y$ features also in the study of constructible sheaves on the $(2,5)$--braid and the study of isomonodromic deformations of the Painlev\'e I differential equation \cite[Section 2.2]{PS}.

But as promised, the method of proof we present in this article is genuinely different from these techniques, and it provides an understanding of the mirror symmetry correspondence from the Legendrian viewpoint. In particular, we extract the maximal Thurston--Bennequin right--handed trefoil directly from the defining equation of $Y$ which, to the authors' knowledge, has not been done directly. The explicit nature of the Legendrian handlebody diagrams for $(X,\la,\p)$ and $(Y,\la,\p)$ directly yields Theorem \ref{thm:mirror}.

Finally, the article also contains material discussing higher--dimensional Reidemeister moves in Subsection \ref{ssec:highD}, the detection of looseness in front diagrams in Subsection \ref{ssec:loose}, and examples of non--flexible subflexible Legendrian fronts in Subsection \ref{ssec:subflex}. These constitute foundational material in the study of higher--dimensional Weinstein structures and the contact topology of their boundaries.

The arc of the work is organized as follows: Section \ref{sec:kirby} contains the material related to Legendrian front projections, Dehn twists, Legendrian isotopies, and loose charts. Then Section \ref{sec:lef} presents the basic material on Lefschetz bifibrations broadening the range of applications of Recipe \ref{dictionary}, and finally Section \ref{sec:app} includes the applications presented in this introduction and their proofs.

{\bf Note}: A properly embedded submanifold $X$ of a Stein manifold inherits a Stein structure, and this endows $X$ with a Weinstein structure $(X,\la,\p)$ unique up to Weinstein deformation. Furthermore, every Weinstein structure is realized in this way, unique up to Stein deformation \cite{CE}. In particular, affine submanifolds $X\sse\C^N$ carry canonical Weinstein structures coming from the induced standard Stein structure of complex affine space $\C^n$. Throughout this article we will freely pass from one perspective to the other.\hfill$\Box$

{\bf Notation}: Given a $(2n-2)$--dimensional Liouville manifold $(F,\la)$ and a compactly supported symplectomorphism $\phi\in\Symp^c(F,\la)$, we denote by
$$(Y,\xi)= \ob(F,\lambda;\phi)$$
the $(2n-1)$--dimensional contact manifold constructed as the contact open book with Liouville page $(F,\la)$ and symplectic monodromy $\phi$ \cite{Co2,Gi2}.

Consider an exact Lagrangian $L\sse(F,\la)$ whose potential functions for $\la|_L$ is $C^0$--bounded by a small amount $\e\in\R^+$. The argument projection of the open book $\ob(F,\lambda;\phi)$ assigns an angle $\theta\in S^1$ to each page $(F_\theta,\la_\theta)$, and we can consider the exact Lagrangians $L_\theta\sse(F_\theta,\la_\theta)$. Their Legendrian lifts $\La_{L_\theta}$ to the $\e$--contactization
$$(F_\theta\times[-\e,\e],\la_\theta+ds)$$
of $(F_\theta,\la_\theta)$ are uniquely defined, up to translation. It is however important to notice that the symplectic monodromy dictates the gluing of such contactizations and thus the Legendrian lifts $\La_{L_\theta}$ can a priori depend on the chosen page $(F_\theta,\la_\theta)$ where we consider our Lagrangian $L_\theta$. We use the notation $\La_L^{\theta}\sse\ob(F,\lambda;\phi)$ to indicate the Legendrian lift of $L_\theta\sse(F_{\delta},\la_{\delta})$.

We also use the shorthand notation $\underline{z}=(z_1,\ldots,z_{n-1})\in\C^{n-1}$.\hfill$\Box$

{\bf Acknowledgements}: We are grateful to Y.~Eliashberg, A.~Keating, M.~McLean, O.~Plamenevskaya and K.~Siegel for valuable discussions and their interest in this work. Special thanks go to L.~Starkston and U.~Varolg\"unes whose many good comments have improved the quality of this article. R.~Casals is supported by the NSF grant DMS-1608018 and E.~Murphy is partially supported by NSF grant DMS-1510305 and a Sloan Research Fellowship.

\section{Front Gallery}\label{sec:kirby}

In this section we introduce an algorithm for drawing fronts of Lagrangians acted on by Dehn twists, and explain the basic rules for the diagramatic calculus in the Legendrian front. These combinatorial rules constitute a major portion of Recipe \ref{dictionary}, and are required for the applications presented in Section \ref{sec:app}.

Each of the first six subsections contributes with an ingredient leading up to the Legendrian stacking, presented in Subsection \ref{ssec:stack}. Then the second part of Subsection \ref{ssec:stack} and Subsection \ref{ssec:Tpqr} serve as our first examples of these methods.

These first six subsections can be shortly described as follows: Subsection \ref{ssec:subcrit} describes the contact manifolds where the Legendrian submanifolds belong, Subsection \ref{ssec:loose} provides the tools to detect which Legendrians are loose, Subsection \ref{ssec:legsk} builds the front projections for these Legendrians, Subsection \ref{ssec:highD} discusses a set of Reidemeister moves in the front which give Legendrian isotopies, Subsection \ref{ssec:Wmoves} introduces the front representation of a Legendrian handleslide and handle cancellation, and finally Subsection \ref{ssec:lagcob} is concerned with the front representation of ambient Legendrian surgeries and their relation to Lagrangian cobordisms.

Let us start with the first building block of a Weinstein manifold, its subcritical skeleton.

\subsection{Subcritical Topology}\label{ssec:subcrit}
The subcritical smooth topology of a Weinstein manifold does not contain meaningful symplectic topology information; this is illustrated by M.~Gromov's contact h--principle on isotropic embeddings below the Legendrian dimension \cite{PDR} and K.~Cieliebak splitting principle \cite{CE} for subcritical Weinstein manifolds.

The subcritical topology of a Weinstein manifold can be quite arbitrary, and thus we will restrict ourselves to simple subcritical skeleta for pictorial purposes; the dictionary works with arbitrary Weinstein manifolds as long as there is an efficient manner to depict their subcritical topology. Focusing on $(2n-3)$--connected Weinstein $2n$--dimensional manifolds described by a tree plumbing will be enough for the presented applications, and we shall do so from this moment onwards.

Consider a tree graph $T=T^{(0)}\cup T^{(1)}$, with vertex 0--skeleton $T^{(0)}$ and edge 1--skeleton $T^{(1)}$, and the $T$--plumbing $(F_T,\lambda_T,\p_T)$ of spheres; this is the $(2n-2)$--dimensional Weinstein manifold $(F_T,\lambda_T,\p_T)$ obtained as the quotient of the disjoint union of $|T^{(0)}|$ copies of the standard disk cotangent bundle $(\mathbb{D}(T^*S^{n-1}),\la_0,\p_0)$, labeled by the vertices $T^{(0)}$, by the equivalence relation which identifies the cotangent fiber of the $i$-th copy of $\mathbb{D}(T^*S^{n-1})$ with an open Lagrangian disk in the zero section of the $j$-th copy according to whether the vertices $i,j\in T^{(0)}$ are adjacent in $T$; see \cite[Section 2]{AbSm} for more details.

The Weinstein manifolds $(F_T\times D^2,\lambda_T+\la_0,\varphi_T+\p_0)$ capture the subcritical topology for our $2n$--dimensional Weinstein manifolds. Due to the handlebody description of a Weinstein manifold, we are in fact interested in the contact boundaries $(Y_0,\xi_0) = \dd (F_T \x D^2,\la_T+\la_0)$, which can in turn be described by the adapted contact open book $(Y_0,\xi_0) = \op{ob}(F_T,\op{id})$; this is the open book associated to the Lefschetz fibration $F_T\x D^2\longrightarrow D^2$ given by the projection onto the second factor. Nevertheless, note that the contact boundary $(Y_0,\xi_0)$ does not depend on the graph structure of the tree $T$ or the symplectic topology of the Liouville page $F_T$, except for the number of vertices $|T^{(0)}|$. Indeed, the Weinstein manifold $W_0 := F_T \x D^2$ is subcritical, and therefore it can be uniquely described as
$$W_0 = D^{2n} \cup_k h^{2n}_{n-1}$$
up to symplectomorphism. The data determining the subcritical Weinstein manifold $W_0$ is
\begin{itemize}
\item[-] The smooth isotopy type of the embedded attaching link
$$\coprod_k S^{n-2} \sse S^{2n-1}=\dd D^{2n},$$
which is necessarily trivial being in codimension $n+1$, and thus it suffices to understand the local case $W_0 = T^*S^{n-1} \x D^2$.
\item[-] The diffeomorphism type, which is $T^*S^{n-1} \x D^2 \cong S^{n-1} \x D^{n+1}$ since spheres are stably parallelizable by the normal bundle of the round embedding.
\item[-] The homotopy type of its almost complex structure, which is given by the symplectomorphism $(T^*S^{n-1} \x D^2,\la_0,\p_0) \cong (T^*(\R^n \sm \{0\}),\la_0,\p_0)$.
\end{itemize}

In particular, the contact structure $(Y_0,\xi_0)$ can be obtained by starting with the standard contact space $\R^{2n-1}_\std\sse (S^{2n-1},\xi_0)$ and performing $|T^{(0)}|$ disjoint $0$--framed subcritical contact surgeries along isotropic spheres $S^{n-2} \sse \R^{2n-1}_\std$, which are boundaries of Legendrian disks, and then consider the smooth point--compactification of the resulting contact manifold. Therefore, the front projection of a Legendrian manifold in the contact boundary $(Y_0,\xi_0)$ is given as a standard Legendrian front in $\R^n$, a nowhere vertical hypersurface with Legendrian singularities \cite{ArnoldSing}, which is additionally allowed to pass through $|T^{(0)}|$ different $S^{n-2}$--wormholes, representing the boundary attaching spheres of the $|T^{(0)}|$ subcritical handles $h^{2n}_{n-1}$. This Legendrian picture has been studied in the knot case $n=2$ \cite{Gompf}, and this discussion allows us to generalize it from there by rotational symmetry.

Because our starting point will be a Lefschetz fibration with fiber $F_T$, it is important for us to keep track of the structure of the tree $T$ since we work directly with open books; even though, as noted, the contactomorphism type of $\dd (F_T \x D^2)$ does not depend on the edges $T^{(1)}$ of the tree $T$. This discussion will be expanded in Sections \ref{ssec:legsk} and \ref{ssec:Wmoves} below.

Before this, we discuss the main theorems about loose Legendrian embeddings which are needed in order to build Recipe \ref{dictionary} and prove Theorems \ref{thm:Xab1} and \ref{thm:kr1} above.

\subsection{Loose Legendrians and flexible Weinstein manifolds}\label{ssec:loose}

In this section we discuss loose Legendrians embeddings \cite{Mu}. The core idea that the reader should get out of this section is that we are able to detect looseness of a Legendrian submanifold in the front projections provided by Proposition \ref{prop:LF}, which in turn allows us to apply $h$--principles and construct symplectomorphisms between a priori distinct symplectic manifolds.

First, the definition of a loose Legendrian submanifold.

\begin{definition}\label{def:loose}
Consider the Legendrian arc $\Lambda_0 \sse (D^3(2), \ker\a_\st)$ depicted in Figure \ref{fig:stabLeg} and the open subset $V_\rho=\{|p| < \rho, |q| < \rho \}\sse (T^*\R^{n-2}(q,p),\la_\st)$, for any $n\geq 3$.

Note that the product submanifold $\La_0\times(V_\rho\cap\{p=0\})$ is a Legendrian submanifold of corresponding product contact manifold $(D^3\times V_\rho,\ker(\a_\st+\la_\st))$. We can now introduce the two following definitions:

\begin{itemize}
\item[-] The relative pair $(D^3\times V_\rho, \Lambda_0 \x (V_\rho\cap\{p=0\}))$ is said to be a {\bf loose chart} if the radius $\rho$ satisfies $\rho>1$.\\

\item[-] Let $\Lambda^{n-1} \sse (Y, \xi)$ be a connected Legendrian submanifold, then $\Lambda$ is said to be a {\bf loose Legendrian} if there exists an open set $U \sse Y$ such that the relative pair $(U, U \cap \Lambda)$ is contactomorphic to a loose chart.\hfill$\Box$
\end{itemize}
\end{definition}

\begin{figure}[h!]
\centering
  \includegraphics[scale=0.45]{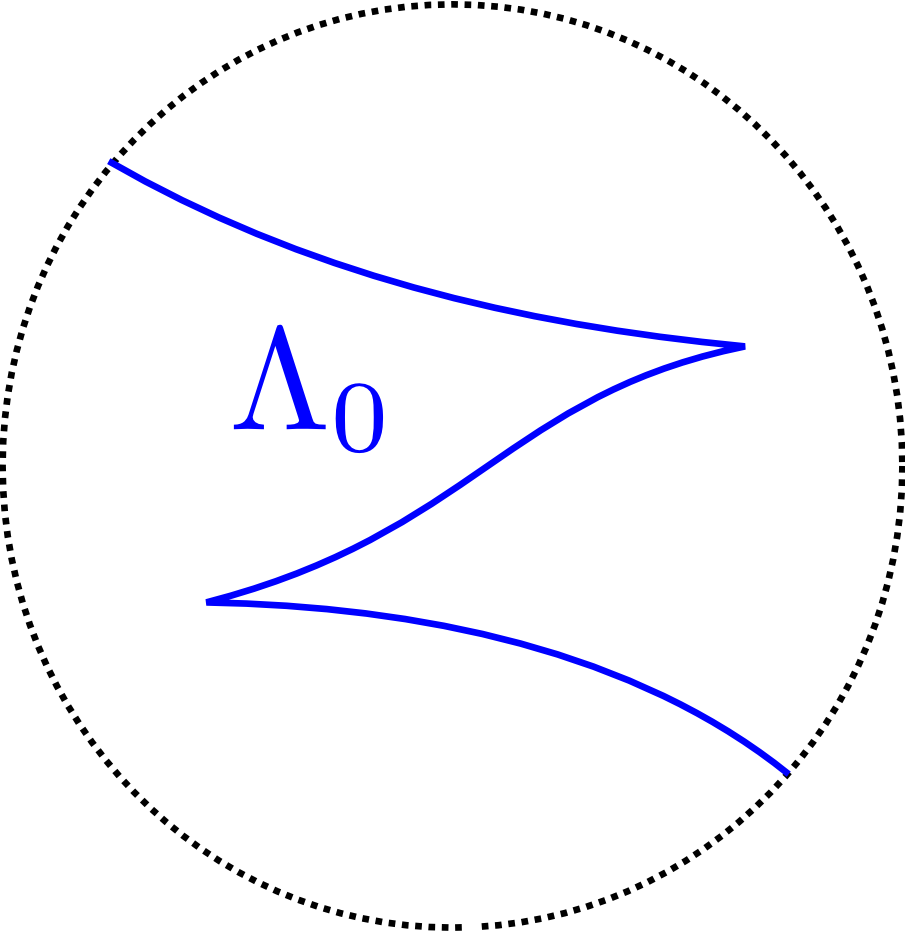}
  \caption{The Legendrian stabilized arc in Definition \ref{def:loose}.}
\label{fig:stabLeg}
\end{figure}

Observe that the open set $U\sse Y$ can be arbitrary, and the notion of a Legendrian submanifold being loose in $(Y,\xi)$ depends globally on the contact topology of $(Y,\xi)$. In particular, a given Legendrian $\Lambda$ might not be loose in a contact manifold $(Y,\xi)$ and become loose if we modify the contact structure $(Y,\xi)$ in a Darboux chart disjoint from an open neighborhood of $\Lambda$, see \cite{CM,CMP,Mu}.

Loose Legendrian submanifolds have been classified up to Legendrian isotopy in the article \cite{Mu}, where they were shown to satisfy an $h$--principle. The crucial property we shall use of loose Legendrian submanifolds is that their behaviour is (almost) only constrained to smooth topology; this is illustrated by Theorem \ref{thm: c0 loose}, for which we introduce the following preliminary definition.

\begin{definition}\label{def:formal}
Let $\Lambda \sse (Y, \xi)$ be a Legendrian manifold, then its normal bundle $\nu\Lambda$ in $(Y,\xi)$ admits a trivializing framing
$$F:\nu\Lambda \xrightarrow{\,\,\cong\,\,} T\Lambda \oplus \R.$$
This framing is canonical up to homotopy: for instance, obtained by the contractible choice of an almost complex structure on $\xi$ compatible with the canonical linear conformal symplectic structure. This framing is called the {\bf Legendrian normal framing}.

Now consider $f_0, f_1:\Lambda\lr (Y, \xi)$ two Legendrian submanifolds which are smoothly isotopic under a smooth isotopy $f_t:\Lambda\lr Y$. This isotopy defines an identification of the normal bundles $\nu f_0(\Lambda) \cong \nu f_1(\Lambda)$. If the Legendrian normal framings of both embedded Legendrian $f_0(\Lambda)$ and $f_1(\Lambda)$ agree under this identification, up to homotopy, the smooth isotopy $f_t$ is said to be a {\bf formal Legendrian isotopy}.\hfill$\Box$
\end{definition}

The following result makes precise the intuition that loose Legendrian submanifolds behave according to algebraic topological constraints:

\begin{thm}[\cite{Mu}]\label{thm: c0 loose}
Let $\Lambda\sse(Y,\xi)$ be a connected Legendrian submanifold and
$$f_0, f_1:\Lambda \longrightarrow (Y, \xi)$$
two Legendrian embeddings such that $f_0$ and $f_1$ are formally Legendrian isotopic.\\
In the case $\dim(Y)\geq 5$, suppose that both $f_0(\Lambda)$ and $f_1(\Lambda)$ are loose Legendrians, and for $\dim(Y) = 3$, suppose that each of $f_0(\Lambda)$ and $f_1(\Lambda)$ is Legendrian isotopic to full stabilization of itself. Then, the Legendrians $f_0(\Lambda)$ and $f_1(\Lambda)$ are Legendrian isotopic.
\end{thm}

\begin{remark}
The 3--dimensional statement does not explicitly appear in the literature as far as the authors are aware, but the proof follows from the theorem of Fuchs--Tabachnikov \cite{FT}: any formal Legendrian isotopy between Legendrian knots can be realized by a Legendrian isotopy after sufficiently many stabilizations, and since $\Lambda \cong s_-(s_+(\Lambda))$, we can realize any number of stabilizations by a Legendrian isotopy.\hfill$\Box$
\end{remark}

The absolute symplectic counterparts of loose Legendrian submanifolds were subsequently introduced in \cite{CE}. These are symplectic manifolds whose symplectomorphism type is also constrained by strictly algebraic topological invariants. The definition reads as follows:

\begin{definition}\label{def:flexible}
Let $(W,\la,\p)$ be a Weinstein manifold of dimension $2n\geq6$.

Then $(W,\la,\p)$ is said to be {\bf explicitly flexible} if for each index $n$ critical point of the Morse function $\p$, the associated Legendrian attaching sphere is a loose Legendrian submanifold in its corresponding contact level set.

Since we are studying Weinstein manifolds up to deformation, we also consider the deformation invariant notion and say that a Weinstein manifold is {\bf flexible} if it is Weinstein homotopic to an explicitly flexible Weinstein structure.\hfill$\Box$
\end{definition}

The corresponding $h$--principle for flexible Weinstein manifolds states that their almost symplectic type determines their Weinstein deformation type. The precise result, analogous to Theorem \ref{thm: c0 loose} above, can be stated as follows:

\begin{thm}[\cite{CE}]\label{thm:CE flex}
Let $(W_1,\la_1,\p_1)$ and $(W_2,\la_2,\p_2)$ be two flexible Weinstein structures, and suppose that there exists a diffeomorphism $f: W_1 \lr W_2$ such that the symplectic vector bundle $f^*TW_2$ is isomorphic to the symplectic vector bundle $TW_1$. Then the diffeomorpshim $f$ is isotopic to a symplectomorphism.
\end{thm}

Note that the main difficulty about Definition \ref{def:flexible} is that it requires a front description of the Legendrian attaching link which, in addition, exhibits a loose chart for each component lying in the complement of the other components. For this reason, it is difficult to tell in general if a Weinstein manifold presented as a Lefschetz fibration is flexible. Detecting flexibility of Weinstein manifolds which are not presented as surgery diagrams, such as affine varieties and explicit Lefschetz fibrations, is a major motivation of this work.

Let us now continue the discussion on loose charts, since both Definition \ref{def:flexible} and Theorem \ref{thm:CE flex} are ultimately based on their understanding. Suppose that we are given a Legendrian $\Lambda$ in a contact manifold $(Y,\xi)$ and we are aiming to prove that $\Lambda$ is a loose Legendrian: according to Definition \ref{def:loose}, we first need to find a $3$--dimensional slice of $\Lambda$ with the stabilized arc from Figure \ref{fig:stabLeg}. It is important to remark that this is not sufficient: it is simple to exhibit such 3--dimensional slices for any higher dimensional Legendrian unknot, and thus for any Legendrian, even 	in a Darboux chart \cite{CM,CMP}. Thus the main difficulty is finding a sufficiently thick product neighborhood of this slice in order to satisfy the radius condition $\rho > 1$. One of the extremely useful properties of the front projections from Proposition \ref{prop:LF} is that we are always able to ensure this condition:

\begin{prop}\label{prop:loose slice}
Consider a Legendrian $\Lambda \sse (\R^{2n-1},\xi_\std)$ and the front projection
$$\pi:(\R^{2n-1},\xi_\st)\lr\R^n.$$
Suppose that there exists a smooth $2$--disk $D^2 \sse \R^n$ parallel to the vertical direction which intersects the Legendrian front $\pi(\Lambda)$ transversely and such that the intersection $D^2 \cap \pi(\Lambda)$ is diffeomorphic to the arc in Figure \ref{fig:stabLeg} as a curve on the disk. Then the Legendrian $\Lambda$ is loose.
\end{prop}

\begin{proof}
The statement has no assumptions on the specific size and shape of the contact neighborhood of the intersection $D^2 \cap \pi(\Lambda)$, and we must then exhibit a loose chart according to the requirements of Definition \ref{def:loose}. First, since the 2--disk $D^2$ is transverse to the front $\pi(\Lambda)$, there exists a neighborhood $U\sse \R^n$ of $D^2$ such that the intersection $U \cap \pi(\Lambda)$ is diffeomorphic to $(D^2 \cap \Lambda) \x D^{n-1}(\e) \sse D^2 \x D^{n-2}(\e) \cong U$ for some $\e\in\R^+$. Let us first construct an isotopy which will allow us now to find a loose chart for $\Lambda$.

Choose a compactly supported isotopy of the disk $D^2$ sending the intersection arc $D^2 \cap \Lambda$ to a rescaling $\delta\cdot\pi(\Lambda_0)$ of the front of $\Lambda_0$, where $\delta\in\R^+$ is chosen such that the rescaling $\delta\cdot\pi(\Lambda_0)$ sits inside the given 2--disk $D^2$. Extend this smooth isotopy to an isotopy of the neighborhood $U$, such that the piece of the front $\pi(\Lambda) \cap D^2 \x D^{n-2}(\e/2)$ is sent to $\delta\cdot\pi(\Lambda_0) \x D^{n-2}(\e/2)$, and it is cutting the isotopy off in the radial direction. This isotopy then extends to a Legendrian isotopy $(\p_t)_{t\in[0,1]}$ of the initial embedded Legendrian $\Lambda \sse (\R^{2n-1},\xi_\std)$.

Second, let us use the isotopy $\p_1$ to exhibit the loose chart. Consider the set
$$V = \{(x,y,z) : (x, z) \in D^2 \x D^{n-2}(\e/2) \}\sse(\R^{2n-1},\xi_\std),$$
such that the intersection $\p_1(\Lambda) \cap V$ is isomorphic to $\delta\cdot\Lambda_0 \x Z$, where we have identified $V$ with $(\R \x D^2) \x T^*D^{n-2}$ and denoted by $Z \sse T^*D^{n-2}$ the zero section. Now, the subset $V$ is equipped with the contact form $\alpha_\std + \lambda_\std$, where $\alpha_\std = dz - \wt y d\wt x$ is the standard contact form on $\R \x B^2 \sse \R^3$, and $\lambda_\std$ is the tautological $1$-form on $T^*D^{n-2}(\e/2)$. Then taking the canonical cotangent coordinates $(q,p)\in T^*D^{n-2}(\e/2)$, we can consider the map
$$(\wt x, \wt y, q, p, z) \longmapsto f(\wt x, \wt y, q, p, z)=\left(\wt x, \frac{\wt y}\delta, \frac{4q}{\e}, \frac{\e p}{4\delta}, \frac z \delta\right).$$
By construction, the map $f$ is a contactomorphism which sends the relative pair $(V, V \cap \p_1(\Lambda))$ to a set containing a loose chart for $\La$, thus proving that $\Lambda$ is a loose Legendrian.
\end{proof}

Proposition \ref{prop:loose slice} allows us to detect loose charts for Legendrian submanifolds in the front projection. There are two particular situations in which looseness can be deduced which will be useful to us in Section \ref{sec:app}, this is the content of the following two propositions.

\begin{prop}\label{prop:cc loose}
Let $(Y,\xi)$ be a contact $(2n-1)$--dimensional manifold and $(\wt Y,\wt\xi)$ the contact manifold resulting from a Weinstein $(n-2)$--surgery with a belt $n$--sphere $C \sse \wt Y$.

In the case $n\geq 3$, any Legendrian submanifold $\Lambda \sse (\wt Y,\wt\xi)$ which intersects the belt sphere $C$ transversely at one point is a loose Legendrian submanifold of $(\wt Y,\wt\xi)$.

For $n=2$, any Legendrian knot $\Lambda \sse (\wt Y,\wt\xi)$ intersecting $C$ transversely in one point is a full stabilization of itself.
\end{prop}

\begin{proof}
This follows from the Legendrian isotopy depicted in Figure \ref{fig:legtrick} for the case $n=2$, and the argument extends to higher dimensions by a symmetric rotation. Figure \ref{fig:legtrick} depicts one of the two components of the attaching $0$--sphere, indicated with $h$, and the creation of a loose chart for any Legendrian submanifold which intersects transversely the belt sphere of the $1$--handle once.
\begin{figure}[h!]
\centering
  \includegraphics[scale=0.7]{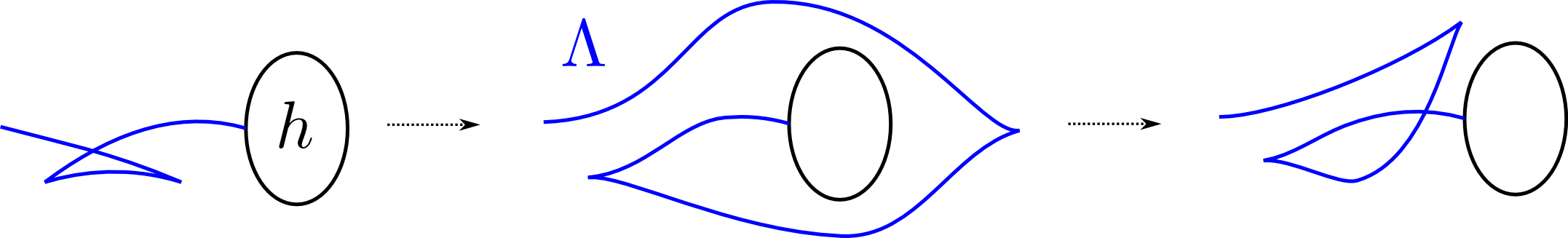}
  \caption{The Legendrian $\Lambda$ and its loose chart.}
\label{fig:legtrick}
\end{figure}
\end{proof}

Proposition \ref{prop:cc loose} is useful to detect looseness in a subcritical scenario, as described by the following:

\begin{prop}\label{prop:ob loose}
Let $(\Sigma,\la)$ be a Liouville structure, $L\sse(\Sigma,\la)$ an exact Lagrangian, and $(Y,\xi)= \op{ob}(\Sigma, \op{id})$ the contact manifold associated to the open book with trivial monodromy. Suppose that there exists a Lagrangian disk $D^{n-1} \sse (\Sigma,\la)$ with Legendrian boundary which intersects $L$ in one point. Then the Legendrian lift $\Lambda \sse (Y,\xi)$ of the exact Lagrangian $L\sse(\Sigma,\la)$ is a loose Legendrian or, in the $3$--dimensional case, a full stabilization of itself.
\end{prop}

\begin{proof}
By construction, the contact structure of the subcritically fillable manifold $Y$ is the one induced as the contact boundary $(Y,\xi)=\dd(\Sigma \x D^2,\la+\la_\std)$. By carving out the disk $D^n$, we see that the Liouville manifold $(\Sigma,\la)$ can be obtained from $\Sigma \sm \Op(D^n)$ by a one index $(n-1)$ Weinstein attachment whose co-core is precisely $D^{n-1}$. In consequence, we can obtain the product $\Sigma \x D^2$ from $(\Sigma \sm \Op(D^n))\times D^2$ by an index $(n-1)$ handle attachment, whose co-core is the thickened disk $D^{n-1} \x D^2 \sse \Sigma \x D^2$. Then the statement follows from Proposition \ref{prop:cc loose} since the belt sphere $C = \dd D^{n-1} \x D^2 \sse (Y,\xi)$ intersects the Legendrian $\Lambda$ transversely in one point.
\end{proof}

These propositions detecting loose charts for Legendrian submanifolds will be both used in this Section and subsequently in the proofs of Theorems \ref{thm:Xab} and \ref{thm:kr}. This also concludes the required ingredients for Recipe \ref{dictionary} concerning loose Legendrians, and we now move to describe the front projections that we will be using, where in particular the above propositions can be applied.

\begin{remark}
It might be interesting to note that the existence of loose charts for the Legendrian unknot implies overtwistedness of the ambient contact manifold \cite{CMP}: this is not further explored in this article but since the contents of this section apply to Legendrians in contact surgery diagrams they can be used to efficiently detect overtwisted disks.
\end{remark}

\subsection{Legendrian skeleton}\label{ssec:legsk}

In order to draw Legendrian handlebodies for Lefschetz fibrations, a crucial ingredient is a front projection for the open book $\op{ob}(F, \lambda, \op{id})$ of the contact boundary. There are several possible front projections in a Darboux chart $(\R^{2n-1},\xi_0)$, and similarly we have different choices for a front projection in $\op{ob}(F, \lambda, \op{id})$; the relevant fact is to ensure that the diagrams we are using are indeed front projections: this is part the content of the following proposition, which we also use to describe the Weinstein handlebody induced by a Lefschetz fibration.

Before stating the proposition, we introduce the following definition which captures the basic building block of our front projections.

\begin{definition}\label{def:legsk}
Let $T$ be a tree and consider the Weinstein $T$--plumbing $(F_T,\la_T,\p_T)$.

A {\bf Lagrangian $T$--skeleton} of the Weinstein manifold $(F_T,\la_T,\p_T)$ is the union of the zero sections of cotangent bundles $\mathbb{D}T^*S^{2n-1}$ which constitute the plumbing $F_T$, i.e.~ a Lagrangian $T$--skeleton is a Lagrangian consisting of $T^{(0)}$ Lagrangian spheres, intersecting transversely according to their adjacencies in the tree $T$.

A {\bf Legendrian $T$--skeleton} is a \emph{connected} Legendrian lift of a Lagrangian $T$--skeleton to the contact manifold $\op{ob}(F_T, \op{id}) \cong \#^k(S^{n-1} \x  S^n, \xi_\std)$, i.e.~ a Legendrian $T$--skeleton consists of $T^{(0)}$ Legendrian spheres which intersect according to their adjacencies in $T$ and such that at the intersection points their tangent spaces together span the contact plane.
\end{definition}

\begin{figure}[h!]
\centering
  \includegraphics[scale=0.75]{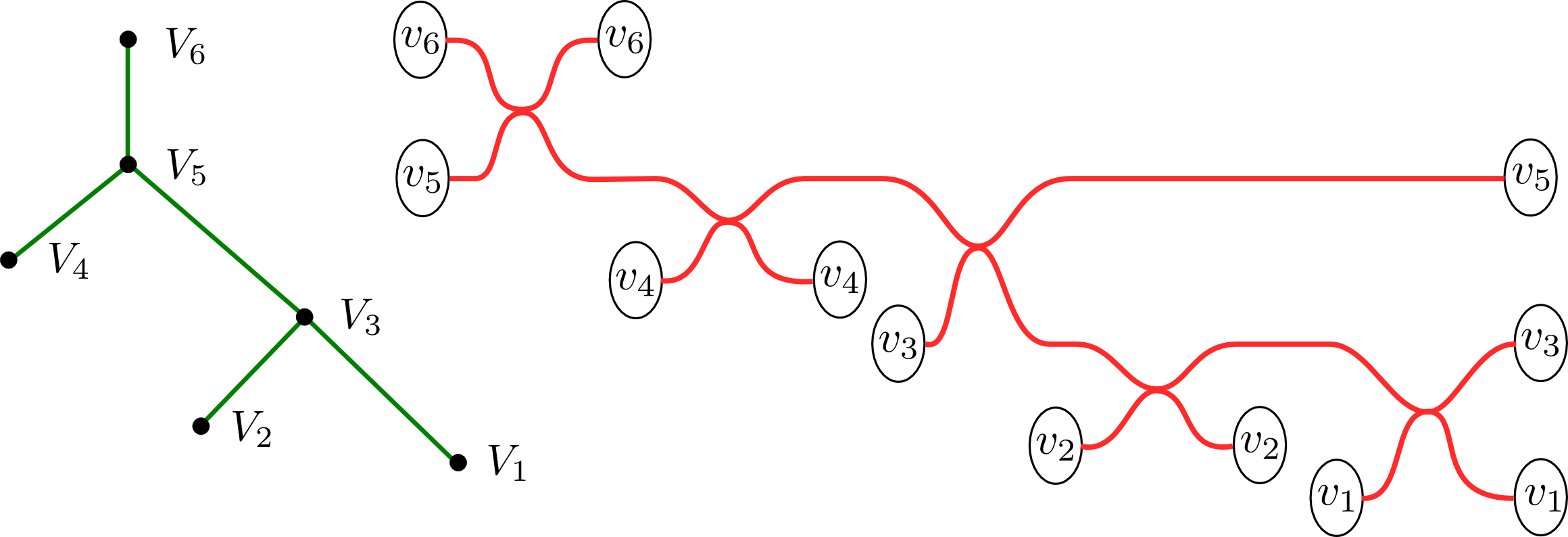}
  \caption{An example of a tree, and its corresponding Legendrian $T$--skeleton.}
\label{fig:tree}
\end{figure}

Figure \ref{fig:tree} depicts the Legendrian $T$--skeleton for a tree. These Legendrian $T$--skeleta will appear constantly from this point onwards, for they provide the geometric information required in order to meaningfully apply the calculus of Legendrian fronts. The following proposition constructs a standard Legendrian skeleton for the contact manifolds appearing as the boundary of the subcritical Weinstein manifolds we are considering:

\begin{prop}\label{prop:LF}
Consider the contact manifold $\#^k(S^{n-1} \x  S^n, \xi_\std)$, a tree $T$ and the $T^{(0)}$ Legendrian spheres $\La_1,\ldots,\La_{T^{(0)}}$, where $\La_i$ is the attaching sphere of a critical handle cancelling the $i^{\text{th}}$ subcritical handle attachment of $\#^k(S^{n-1} \x  S^n, \xi_\std)$, i.e.~ $\Lambda_i = S^{n-1} \x \{\mbox{pt}\}$ in the $i^{\text{th}}$ term of the connect sum.

Legendrian homotope the Legendrians $\Lambda_i$ such that the resulting Legendrians intersect according to the adjacencies in the given tree $T$, and intersect transversely inside the contact planes at the intersection points. Then the union $\Lambda$ of the images of this homotopies is a Legendrian $T$--skeleton.
\end{prop}

\begin{proof}
Let $\wt\Lambda \sse Y := \#^k(S^{n-1} \x  S^n, \xi_\std)$ be any given Legendrian $T$--skeleton. Both sets $\Lambda$ and $\wt\Lambda$ consist of Legendrian spheres which pass through the subcritical handles and intersect according to adjacency in the same tree $T$. In order to construct a Legendrian isotopy between them we use the $h$--principle from Theorem \ref{thm: c0 loose}, for which Propositions \ref{prop:cc loose} and \ref{prop:ob loose} will provide the loose charts.

First, we can find a contact isotopy which sends neighborhoods of the intersection points of $\Lambda$ to the corresponding points of $\wt\Lambda$: this isotopy exists because intersections between Legendrians which are transverse in the contact plane are all locally equivalent, and contact isotopies can be chosen to take any finite collection of small Darboux balls to any other such collection. Let $U \sse Y$ be the union of these small neighborhoods at the intersection points.

In the complement of $U$, both Legendrians $\Lambda$ and $\wt\Lambda$ consist of a collection of $T^{(0)}$ disjoint Legendrian punctured spheres: let us compare the Legendrian type of the complements $\La\sm U$ and $\wt\La\sm U$. On the one hand, according to Proposition \ref{prop:cc loose} each component of the complement $\Lambda \sm U$ is a loose Legendrian, or a full stabilization of itself if $\dim(Y) = 3$, in the complement of all the other components of $\Lambda\sm U$. On the other, each component of $\wt\Lambda \sm U$ is also loose, or a full stabilization of itself, in the complement of the other components: in this case this is because each component of the complement $\wt\Lambda \sm U$ is the Legendrian lift of a zero section in a plumbing of spheres, and the Lagrangian cofiber satisfies the hypothesis of Proposition \ref{prop:ob loose}. Therefore, it remains to show that $\wt\Lambda$ is formally Legendrian isotopic to $\Lambda$, by an isotopy fixed on $U$, and Theorem \ref{thm: c0 loose} will apply.

Up to formal Legendrian isotopy, each component is contained in an independent component of the connect sum, so it suffices to show the result for the case of the one vertex tree $T^{(0)}=1$: i.e.~ that the Legendrian lift $\wt\Lambda$ of the zero section in $\dd(T^*S^{n-1} \x \C)$, is formally Legendrian isotopic to the Legendrian $\Lambda = S^{n-1} \x \{\mbox{pt}\} \sse (S^{n-1} \x S^n, \xi_\st)$. In fact they are equal: the contact structure on the boundary $\dd(T^*S^{n-1} \x \C)$ is given by the hyperplane $\ker(r^2d\theta + pdq)$, where $(r, \theta)$ are polar coordinates on $\C$ and $(p,q)$ are standard cotangent bundle coordinates, and $\dd(T^*S^{n-1} \x \C) = \{r^2 + |p|^2 = 1\}$. Now, the coordinates on the $(n-1)$--sphere $S^{n-1}$ correspond to the $q$--coordinates, and the coordinates on the $n$--sphere $S^n$ are the coordinates $(p, r, \theta)$, and consequently the Legendrian lift is the Legendrian
$$\wt\Lambda = \{(r,\theta,p,q):p=0,\,\, r = 1,\,\, \theta = \theta_0\} = \Lambda,$$
for some constant value $\theta_0$. Thus the candidate Legendrian $T$--skeleton $\La$ is Legendrian isotopic to the Legendrian $T$--skeleton $\wt\La$.
\end{proof}

The contact neighborhood of the standard Legendrian $T$--skeleton provided by Proposition \ref{prop:LF} is the ambient space in which we will be constructing Legendrian submanifolds and performing Legendrian calculus. Let us now discuss the Legendrian fronts that represent these Legendrians.


\subsection{High--dimensional Legendrian fronts}\label{ssec:highD}

The complete description of generic fronts of arbitrary Legendrian submanifolds in high dimensions is not combinatorial, in contrast to the 3--dimensional case of Legendrian knots. The reason is that the space of Legendrian singularities, and therefore the space of distinct Reidemeister moves, becomes uncountable and even infinite dimensional for large dimensions \cite{ArnoldSing}. However the Weinstein manifolds we often study are constructed from standard pieces with spherical symmetry, and thus we can still describe many interesting manifolds in arbitrary dimensions with a reduced number of high--dimensional moves.

In our case, a tree $T$ is the combinatorial data that yields the front projection that we use in order to depict our Legendrian submanifolds; this is the content of Proposition \ref{prop:LF} above, which provides the ambient space and the front projection where the Legendrian calculus occurs. Let us now discuss the Legendrian singularities and corresponding Reidemeister moves that feature in these front projections and are used in this first edition of the dictionary.

\begin{figure}[h!]
\centering
  \includegraphics[scale=0.8]{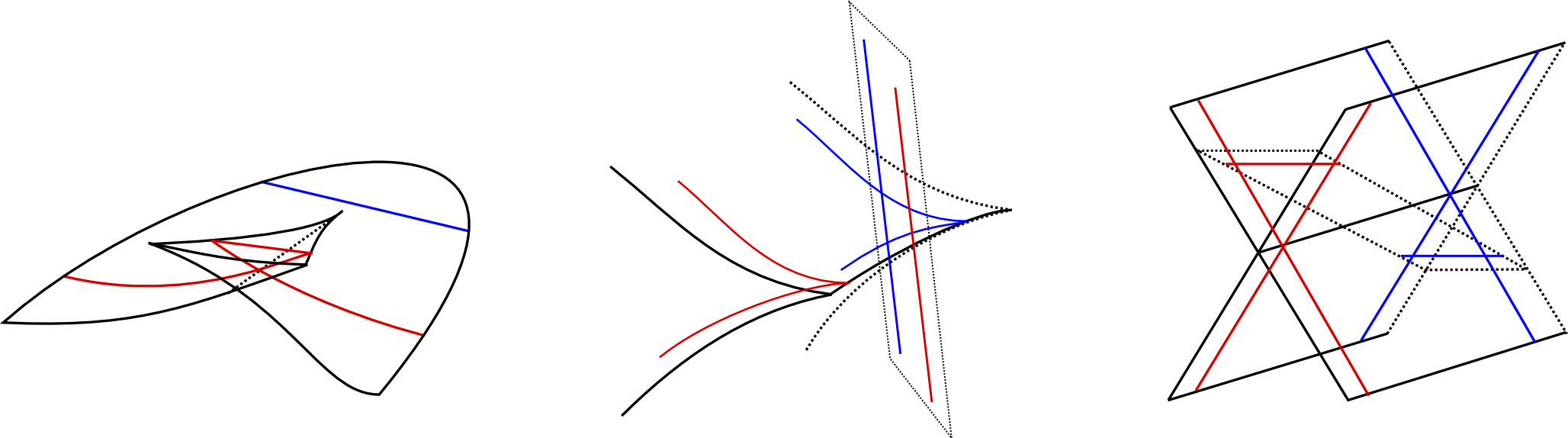}
  \caption{The 1--dimensional Reidemeister moves for Legendrian knots; we have depicted them as Legendrian surface singularities. The red front lifts to a Legendrian curve which is Legendrian isotopic to the blue front.}
  \label{fig:1DReid}
\end{figure}

There are two generic types of singularities of Legendrian fronts in $\R^2$, cusps and transverse double points, corresponding to $A_2$ and $A_1^2$ wavefronts, and the three sufficient Legendrian Reidemeister moves shown in Figure \ref{fig:1DReid}: these correspond to the surface Legendrian singularities $A_3$, $A_2A_1$ and $A_1^3$. In this article we encounter only three types of high--dimensional generic Legendrian singularities:
\begin{itemize}
\item[-] The product of a 1--dimensional cusp singularity with a smooth manifold; locally this is the product with $\R^{n-2}$.
\item[-] The transverse intersection of two smooth $\R^{n-1}$.
\item[-] The $S^{n-2}$--symmetric rotation of a transverse intersection of curves along an axis through the intersection point, which we refer to as the {\bf cone singularity}.
\end{itemize}
See Figure \ref{fig:sing} for pictures of these singularities in the case of Legendrian fronts in $\R^3$, i.e.~ Legendrian submanifolds of $\R^5_\st$. In the list above, the cone singularity is unusual in that it is not a generic Legendrian singularity. However, we restrain ourselves from generically perturbing it in order to preserve the spherical symmetry, which allows us to make arguments independent of dimension.

\begin{figure}[h!]
\centering
  \includegraphics[scale=0.8]{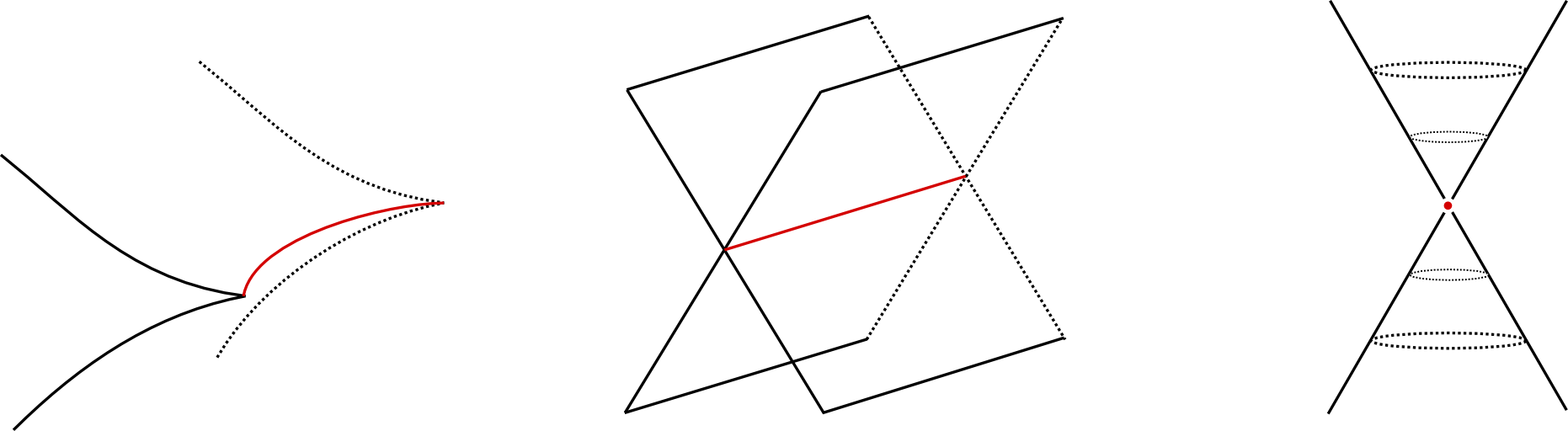}
  \caption{The three types of high--dimensional Legendrian singularities that appear in the applications of Section \ref{sec:app}.}
  \label{fig:sing}
\end{figure}

In order to use Recipe \ref{dictionary} in practice, we need an effective method for drawing Legendrian fronts in arbitrary dimensions. This is done by describing high--dimensional fronts by drawing Legendrian curves, which are then be extended to an actual high--dimensional front by having a number of local $S^{n-2}$--spherical symmetries. Let us start with an example.

\begin{figure}[h!]
\centering
  \includegraphics[scale=0.4]{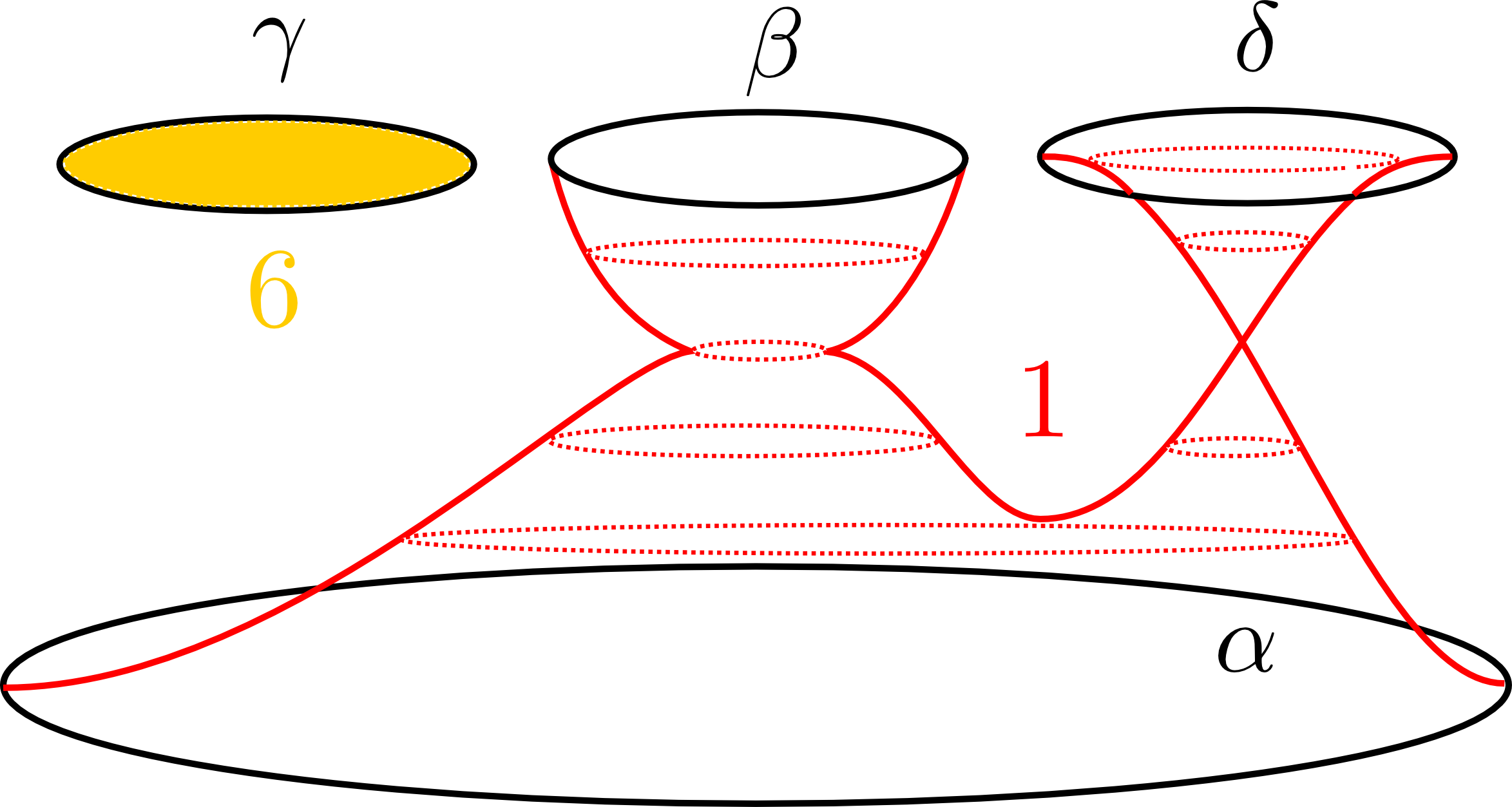}
  \caption{The actual Legendrian surface front of the Legendrians 1, in red, and 6, in yellow, from the central slice depicted in Figure \ref{fig:ex1dict}. The unknotted circles $\{\alpha,\beta,\delta,\gamma\}$ are the attaching spheres of four different subcritical $2$--handles; the Legendrian 2--spheres 1 and 6 are attaching spheres for two different critical Weinstein handles.}
  \label{fig:ex1dictreal}
\end{figure}

\begin{example}
Consider the Legendrian surface front in the right hand side of Figure \ref{fig:ex1dict}: it actually depicts a six--component link of Legendrian $2$-spheres $S^2$ in the 5--dimensional contact manifold $\#^4(S^2 \x S^3,\xi_0)$. The subcritical 2--handles $\alpha, \beta, \gamma,$ and $\delta$ represent subcritical circles in $\R^3$ obtained by spinning in the transverse direction, see Figure \ref{fig:ex1dictreal}; similarly the Legendrian 2--sphere labelled by $1$ in Figure \ref{fig:ex1dict} is drawn in its actual 2--dimensional form in Figure \ref{fig:ex1dictreal}. Notice that the Legendrian surface is not globally $S^1$--symmetric but it is determined by the curve we draw and a number of local $S^1$--rotations.

Note that, in fact for any $n$, the right hand side of Figure \ref{fig:ex1dict} depicts a six--component link of Legendrian $(n-1)$--spheres in the contact manifold $\#^4(S^{n-1} \x S^n, \xi_0)$. Once we have a picture of the central slice and we know the axes of the $S^{n-2}$--symmetries, we have a concretely defined front in any dimension.\hfill$\Box$
\end{example}

In the following list, we gather the conventions that are used throughout the article for drawing high--dimensional fronts as curves:

\begin{itemize}
\item[-] The Legendrian fronts $\Lambda^{n-1} \sse \R^n$ which arise in Recipe \ref{dictionary} are described by choosing a central slice $\R^2 \sse \R^n$ which is parallel to the vertical direction, and drawing the Legendrian curve $\Lambda \cap \R^2\sse\R^2$.\\
\item[-] In the transverse directions, Legendrians are extended by local $S^{n-2}$--symmetry, whose axes are always contained in the central slice.\\
\item[-] Subcritical handles, defined by isotropic $(n-2)$--spheres, are drawn as an $S^0$ in the central slice, and similarly defined by rotation around a vertical axis. As much as possible, an axis for a symmetry of the Legendrian front will coincide with the axis of symmetry for a subcritial handle.\\
\item[-] All Reidemeister moves respect the above symmetries.\\
\item[-] Cone singularities are marked with a {\bf thick} dot, indicating that the Legendrian is extended by a rotation through an axis through that point.\\
\item[-] Intersections of the curve which do not have a thick dot are genuine transverse intersections of the Legendrian front, and in this case the Legendrian is extended in transverse directions by a rotation whose axis does not pass through the intersection.\\
\item[-] Ideally, the location of the axes of the local spherical $S^{n-2}$--rotations are implicit from the local $S^0$--symmetries of the curve $\La\cap\R^2$ in the central slice. For example, in the Legendrian surface from Figure \ref{fig:ex1dict} the cusps and intersections come in pairs determined by reflection through vertical axes. Thus all cones pass through these axes, and the gluing map for the subcritical surgeries are also defined by these reflections.
\end{itemize}

This list concerns drawing a Legendrian front in the front projection provided by Proposition \ref{prop:LF}. but it is oftentimes possible to further simplify the front via Legendrian isotopies, which we now discuss.

The set of Reidemeister moves in higher dimensions is genuinely rich and their understanding is crucial for the development of higher--dimensional contact topology; nevertheless the subset of Reidemeister moves obtained from lower dimensional moves already provides new insights. In particular, any Reidemeister move of Legendrian curves yields a higher--dimensional Reidemeister move by spinning around any sphere $S^{n-2}$; these are particularly well--suited to us due to above drawing convention.

There are two interesting cases to consider depending on the location of the axis along which we spin. If the axis of rotation is generic, i.e.~disjoint from the Darboux chart in which the Reidemeister move occurs, the spherical spinning gives analogous Legendrian isotopies which can be immediately pictured in higher fronts. If instead the axis of rotation does intersect the Legendrian submanifold we obtain the two moves depicted in Figure \ref{fig:highDReid}, which are indeed Legendrian isotopies. The upper move states that two cusps passing through each other $S^{n-2}$--symmetrically are Legendrian isotopic to their disjoint union. The lower move in Figure \ref{fig:highDReid} tells us that a smooth plane $\R^{n-1}$ is allowed to pass through a cone singularity if it transversely intersects a neighborhood of the cone in a sphere $S^{n-2}$.

\begin{figure}[h!]
\centering
  \includegraphics[scale=0.9]{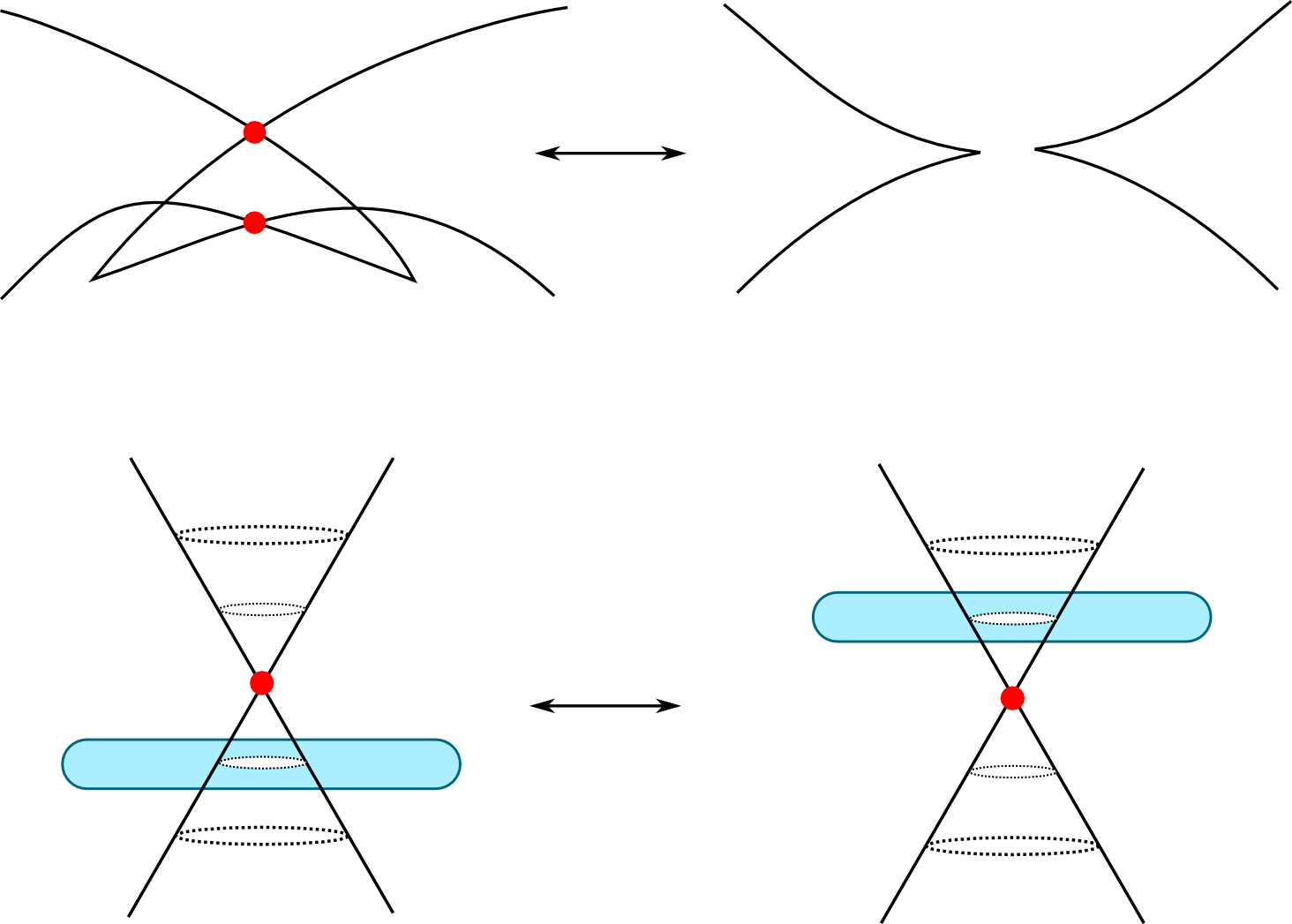}
  \caption{Two higher--dimensional Reidemeister moves featuring the non--generic Legendrian cone singularity, indicated by a thick dot as indicated in the above list of conventions.}
  \label{fig:highDReid}
\end{figure}

These rotational moves shall suffice for the applications in this article, and since they are induced from 1--dimensional moves it is tempting to assume that it suffices to understand 1--dimensional Legendrians in order to manipulate higher--dimensional fronts with these. This is not the case, and although we draw in this dimension we would like to make the reader aware that this is only a convention, and one needs to have a higher--dimensional understanding of the moves and keep track of the symmetries in order to be sure that a certain high--dimensional Reidemeister move can be applied in a given Legendrian front.

\subsection{Legendrian Kirby calculus}\label{ssec:Wmoves}

Given a contact ($\pm1$)-surgery presentation of a contact manifold, we now discuss the Legendrian front moves representing Legendrian handle slides and handle cancellations. This is the fifth ingredient in our understanding of Legendrian fronts and it has a fundamental role in Legendrian calculus.

Reidemeister moves are local diagramatic exchanges performed in the front of a Legendrian $\Lambda$ such that the Legendrian submanifold represented by the resulting (different) Legendrian front is Legendrian isotopic to $\Lambda$. In contrast, a Legendrian handle slide is an isotopy that occurs in a surgered manifold, which does not come from a Legendrian isotopy in the non--surgered manifold. Similarly, a handle cancellation is not an isotopy, but instead a move that passes between different handle decompositions of a single Weinstein manifold.

The following Proposition describes how to perform a Legendrian handle slide of a Legendrian submanifold along a Legendrian ($\pm1$)--surgery sphere. We remark that in the $3$--dimensional case this result was obtained in \cite{DG} using strictly 3--dimensional different methods.

It is important to emphasize that $(\pm 1)$--surgeries are not on equal footing from the perspective of Weinstein manifolds: a Weinstein handle attachment induces a $(-1)$--contact surgery on the boundary contact manifold, whereas a $(+1)$--surgery is the operation defined by a \emph{downward} Weinstein cobordism, or alternatively by \emph{carving} a Weinstein handle out from an existing Weinstein manifold. Both are needed throughout the paper.

\begin{prop}[Legendrian Handle Slides]\label{prop:slide}
Let $(Y,\xi)$ be a contact manifold, $\Lambda,\Sigma\sse (Y,\xi)$ two disjoint Legendrian submanifolds such that $\Lambda$ is a sphere. Then
\begin{itemize}
 \item[A.] The Legendrians $\Sigma$ and $h_\Lambda(\Sigma)$ presented in Figure \ref{fig:HSPos} are Legendrian isotopic in the surgered contact manifold $Y_\Lambda(-1)$.
 
\begin{figure}[h!]
\centering
  \includegraphics[scale=0.6]{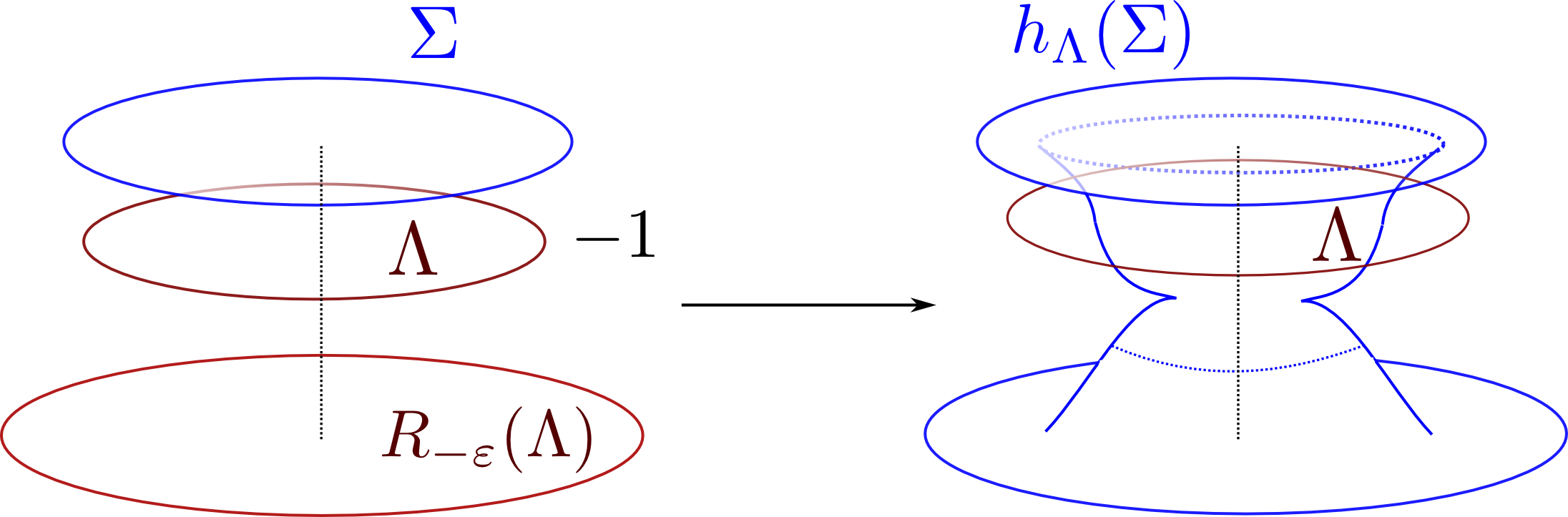}
  \caption{Handleslide of $\Sigma$ along the ($-1$)--Legendrian $\Lambda$.}
\label{fig:HSPos}
\end{figure}
 
 \item[B.] The Legendrians $\Sigma$ and $h_\Lambda(\Sigma)$ presented in Figure \ref{fig:HSNeg} are Legendrian isotopic in the surgered contact manifold $Y_\Lambda(+1)$.

 \begin{figure}[h!]
\centering
  \includegraphics[scale=0.6]{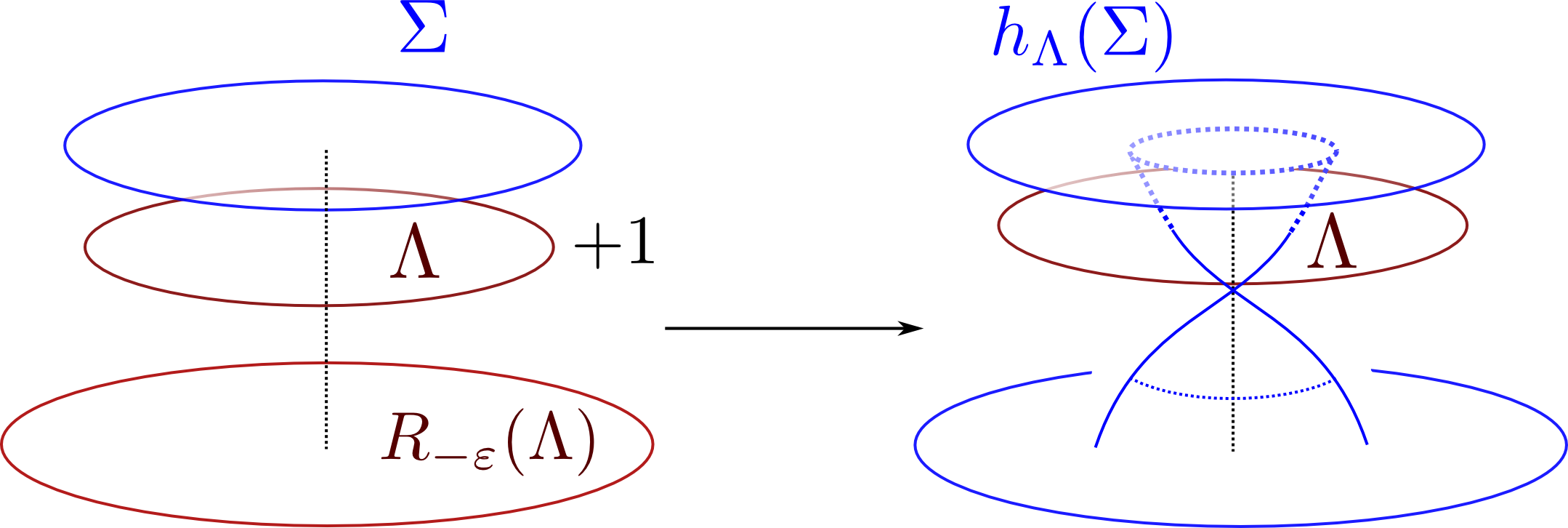}
  \caption{Handleslide of $\Sigma$ along the ($-1$)--Legendrian $\Lambda$.}
\label{fig:HSNeg}
\end{figure}
\end{itemize}
The Legendrians $h_\Lambda(\Sigma)$ are referred to as Legendrian handleslides of $\Sigma$ along $\Lambda$.
\end{prop}

\begin{remark}
Both Figures \ref{fig:HSPos} and \ref{fig:HSNeg} occur in the front projection of a contact Darboux chart $U \sse (Y,\xi)$. In the complement $Y\setminus U$ of this Darboux chart the resulting Legendrian $h_\Lambda(\Sigma)$ is equal to either $\Sigma$ or a small Reeb push--off of $\Lambda$; this Reeb pushoff $R_{\pm \e}(\Lambda)$ of the surgery sphere $\Lambda$ must be chosen such that $R_{\pm\e}(\Lambda)$ belongs to the opposite side of $\Sigma$ with respect to $\Lambda$.

Observe that the notation $h_\Lambda(\Sigma)$ is mildly inaccurate because for a fixed pair of Legendrians $\Sigma$ and $\Lambda$ there are many possible Legendrian handle slides: we will be indicating the Legendrian handleslides when performed and thus this will not be an issue.\hfill$\Box$
\end{remark}

\begin{proof}
First, we study the geometry of a critical Weinstein handle in $\R^{2n}(q_1,\ldots,q_n,p_1,\ldots,p_n)$; in this case of a critical index, the local model is described by the following Liouville form and Liouville vector field:
$$\la=\sum_{i=1}^n p_idq_i+2q_idp_i,\quad X_\la=\sum_{i=1}^n 2q_i\dd_{q_i}-p_i\dd_{p_i}.$$
The region we work with is the bidisk $H=(\{(q,p)\in\R^{2n}:|q|\leq1,|p|\leq1\},\la)$, whose boundary consists of the two contact pieces
$$(\dd^+H,\la)=(\{(q,p)\in H:|q|=1\},\la),\quad (\dd^-H,\la)=(\{(q,p)\in H:|p|=1\},\la).$$
The attaching sphere is the Legendrian sphere
$$\La_a=\{(q,p)\in \dd^-H:p=0\}\sse(\dd^-H,\la),$$
which corresponds to $\La$ in the statement of the Proposition, and the belt sphere is
$$\La_b=\{(q,p)\in \dd^+H:q=0\}\sse(\dd^+H,\la).$$
The time--$t$ contact flow $\p_t$ of the Liouville vector field $X_\la$ is the decoupled diffeomorphism
$$(q,p)\longmapsto (e^{2t}q,e^{-t}p)$$
and thus maps the region $(\dd^-H,\la)\setminus\La_a$ onto $(\dd^+H,\la)\setminus\La_b$ at time $t=-\ln|q|^{1/2}$. As it should, it collapses the Legendrian sphere $\La_a$ to the fixed point at the origin $0\in H$ in infinite time, from which the belt sphere $\La_b$ is also born, also in infinite time. We can compute the map $\p_{\ln|p|}: \dd^+H \sm \La_b \longrightarrow \dd^-H \sm \La_a$ in coordinates and see that it is given by
$$\p_{\ln|p|}(q, p) = (|p|^2q, |p|^{-1}p).$$ 

The two contact pieces of the boundaries are neighborhoods of the Legendrian sphere $\La_a$ and $\La_b$, thus contactomorphic to the 1--jet spaces of these spheres. In order to describe the Legendrian handleslide explicitly, we parametrize each of these boundaries as follows:
$$\phi:(J^1S^{n-1},\a^+)\lr(\dd^+H,\la),\quad q_i(x,y,z)=x_i,\quad p_i(x,y,z)=zx_i/2+y_i,\quad i=1,\ldots n,$$

where we endow $J^1S^{n-1}$ with the 1--jet coordinates of $J^1\R^{n}$ induced by the round unit inclusion $S^{n-1}\sse\R^n$ and its contact form. Hence the 1--jet space of the sphere is the contact manifold
$$(J^1S^{n-1},\alpha^+)=(\{(x,y,z)\in J^1\R^n:|x|=1,x\cdot y=0\},dz-y\cdot dx)$$
It is readily verified that the diffeomorphism $\phi$ is a contactomorphism.

The remaining piece requires a minor modification of the above parametrization:
$$\psi:(J^1S^{n-1},\a^-)\lr(\dd^-H,\la),\quad q_i(u,v,s)=su_i-v_i,\quad p_i(u,v,s)=u_i,\quad i=1,\ldots n,$$
where as above, $(u,v,s)$ are coordinates in $(J^1\R^n,ds-v\cdot du)$. The inverse of the contactomorphism $\psi$ is also needed in the upcoming computation; it is provided by
$$u_i(q,p)=p_i,\quad v_i=(q\cdot p)p_i-q_i,\quad s=q\cdot p.$$

This setup above now allows to describe a relative Legendrian isotopy of a Legendrian disk crossing the belt sphere $\La_b\sse (\dd^+H,\la)$ once, from the perspective of a neighborhood of the attaching sphere $\La_a\sse (\dd^-H,\la)$, i.e.~a Legendrian handleslide. We consider the 1--parametric family of Legendrian disks
$$i_c:D^n\lr(J^1 S^{n-1},\a^+),\quad x_n(t)=\nu(t),\quad y_n(t)=-2\|t\|^2\nu(t),\quad z(t)=\|t\|^2+c,$$
$$\quad x_i(t)=t_i,\quad y_i(t)=2t_i\nu(t)^2,\quad 1\leq i\leq n,\quad c\in[-1,1],$$
where the disk has coordinates $D^n(t_1,\ldots,t_{n-1})$ and we have denoted $\displaystyle\nu(t)=\left(1-\|t\|^2\right)^{1/2}$. This takes place in a neighborhood of the belt sphere $\La_b\sse(\dd^+H,\la)$, and the Legendrian disks $i_1(D^n)$ and $i_{-1}(D^n)$ are Legendrian isotopic; note that the Legendrian embeddings $i_c$ can readily be modified to have compact support. Therefore comparing to the notation in the statement of the Proposition, we have
$$\Sigma = \psi^{-1}\circ \p_{\ln|p|} \circ\phi\circ i_{-1}D^n,\quad h_\Lambda(\Sigma) = \psi^{-1}\circ \p_{\ln|p|} \circ \phi \circ i_1D^n.$$
It therefore remains to show that $h_\Lambda(\Sigma)$ has the front projection in Figure \ref{fig:HSNeg} as claimed. The contactomorphism
$$\psi^{-1}\circ \p_{\ln|p|}\circ\phi:(J^1S^{n-1},\a_+)\lr(J^1S^{n-1},\a_-)$$
is given in these coordinates by
$$s(x,y,z)=z\eta(x,y,z)/2,\quad u_i(x,y,z)=(zx_i/2+y_i)/\eta(x,y,z),\quad 1\leq i\leq n,$$
where
$$\eta:J^1S^{n-1}\lr\R,\quad \eta(x,y,z)=\sqrt{z^2/4+\|y\|^2}.$$
The advantage of parametrizing $(J^1S^{n-1},\a_+)$ is just being able to describe explicitly the family of disks $i_c(D^n)$, but that is not strictly necessary; nevertheless, the parametrization of $(J^1S^{n-1},\a_-)$ is crucial, since it provides the front projection $\pi(u,v,s)=(u,v)$. For this reason, we only need the coordinates $u(x,y,z)$ and $s(x,y,z)$. Precomposing with the inclusion $i_c$ and projection $\pi$, we obtain the $c$--family of Legendrian fronts
$$\pi\circ\psi^{-1}\circ \p_{\ln|p|}\circ\phi\circ i_c:D^n\lr(J^1S^{n-1},\a^-),\quad 1\leq i\leq n-1,$$
$$s(t)=\frac{(\|t\|^2+c)\eta(t)}{2},\quad u_n(t)=\frac{(\|t\|^2+c)\nu(t)-4\|t\|^2\nu^2(t)}{2\eta(t)},\quad u_i(t)=\frac{(\|t\|^2+c)t_i+4t_i\nu(t)^2}{2\eta(t)}.$$
These functions parametrize the spherical cuspidal edge depicted in Figure \ref{fig:HSPos} in the statement of Proposition \ref{prop:slide}, thus concluding Part A. The proof for the front description of a handleslide along a (+1)--surgery Legendrian is identical except for the fact that the Liouville vector field must be considered in the reversed direction, so instead consider the family of fronts 
$$\pi \circ \left(\psi^{-1}\circ \p_{\ln|p|}\circ\phi\right)^{-1}\circ i_c: D^n \lr (J^1S^{n-1}, \alpha^+).$$
The corresponding computations yield the front projection shown in Figure \ref{fig:HSNeg}.
\end{proof}

\begin{ex}
It is a good exercise to verify that a Legendrian ($\pm1$)--handleslide can be undone with another Legendrian ($\pm1$)--handleslide in another location. This is illustrated in Figure \ref{fig:CancelPosHS} in the case of a $(-1)$--handleslide; an analogous sequence shows that a $(+1)$--handleslide can be undone by performing an appropriate $(+1)$--handleslide.
\begin{figure}[h!]
\centering
  \includegraphics[scale=0.8]{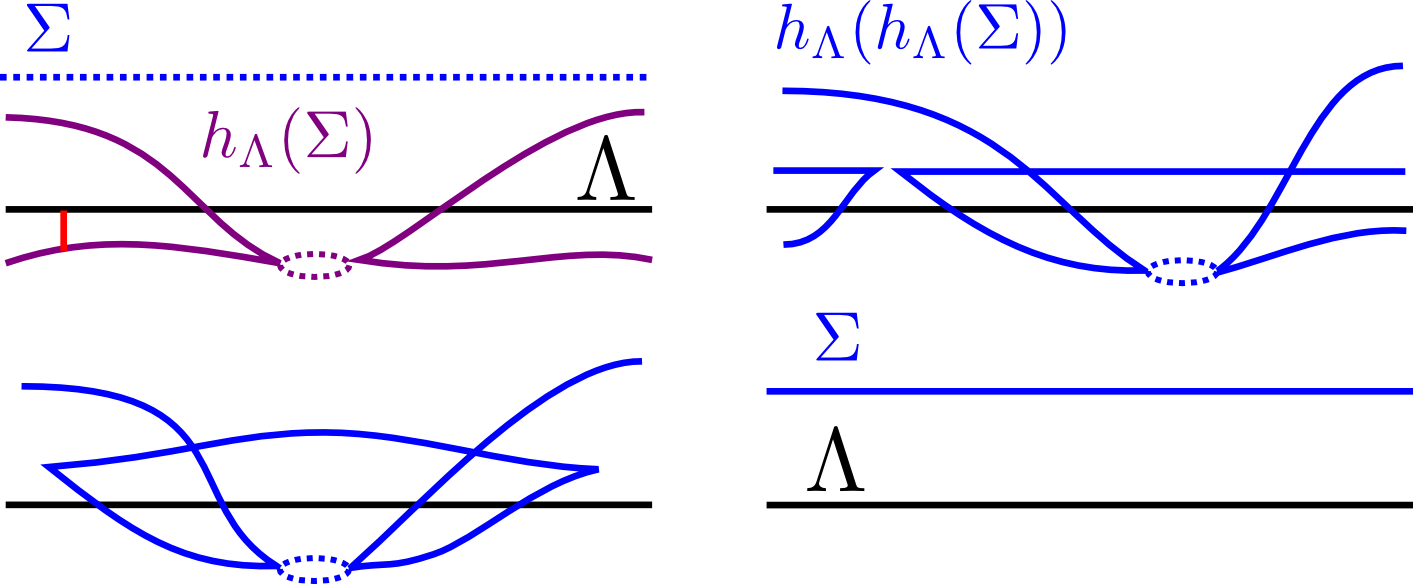}
  \caption{In this picture $\Lambda$ is a $(-1)$--Legendrian sphere, the sequence depicts a $(-1)$--handleslide along $\Lambda$ undoing a previous $(-1)$--handleslide along $\Lambda$: the Legendrian $\Sigma$ is isotopic to $h_\Lambda(h_\Lambda(\Sigma))$.}
\label{fig:CancelPosHS}
\end{figure}
\end{ex}

The second operation after Legendrian handle slides are handle cancellations, which we now address in the following two propositions. The first of them is the Weinstein equivalent to the cancellation of a cancelling pair of Morse critical points:

\begin{prop}[{\cite[Theorem 10.12]{CE}}]\label{prop:basic cancel}
Let $(W, \lambda, \p)$ be a $2n$--dimensional Weinstein cobordism with exactly two critical points $p$ and $q$ such that $\mbox{ind}(p)=n-1$, $\mbox{ind}(q)=n$, and $\p(p) < \p(q)$. Consider a value $c\in\R$ with $\p(p) < c < \p(q)$, and let $(Y,\xi) = (\p^{-1}(c),\la)$ be its contact level set.

Suppose that the attaching Legendrian sphere in $(Y,\xi)$ of the critical point $q$ intersects the belt sphere of $p$, a coisotropic sphere in $(Y,\xi)$, transversely in a single point. Then the Weinstein cobordism $(W, \lambda, \p)$ is Weinstein homotopic to the symplectization
$$([0,1] \x  Y_0, e^r\alpha, (r, y) \mapsto r)$$
of the contact boundary $(Y_0, \alpha) \cong \dd_-W \cong \dd_+W$.
\end{prop}

\begin{figure}[h!]
\centering
  \includegraphics[scale=1]{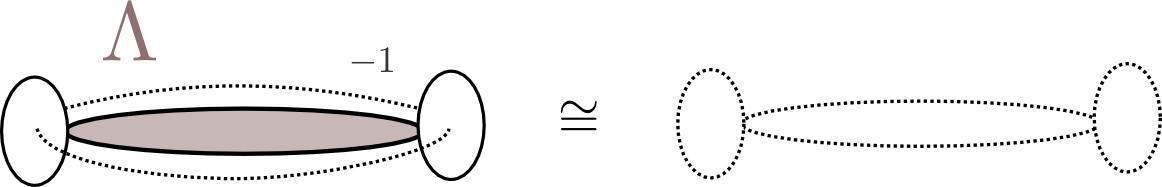}
  \caption{Cancellation Move I: Proposition \ref{prop:basic cancel} states that the Weinstein cobordism obtained by attaching a handle along the isotropic sphere, represented by a 1--sphere, and a second handle along the grey Legendrian 2--sphere is Weinstein homotopic to the symplectization, where no Weinstein handles are attached.}
\label{fig:CancelMove2}
\end{figure}

We can use Proposition \ref{prop:basic cancel} in order to describe handle cancellations in the front projection. The simplest instance is the equivalence depicted in Figure \ref{fig:CancelMove2}, where the cancelling pair of isotropic and Legendrian spheres disappear. This holds at the level of Weinstein cobordisms, and thus can be considered as a statement about contact surgeries.

In a cancellation pair, there might be relevant Legendrian submanifolds that interact with the subcritical handle and the direct cancellation depicted in Figure \ref{fig:CancelMove2} cannot be applied; the following proposition explains how to modify these Legendrians in the front and still achieve cancellation.

\begin{prop}\label{prop:handle cancel}
Let $(W, \lambda, \p)$ be a Weinstein domain represented by a handlebody diagram: i.e.~ we assume that there is a $c \in [0, \infty)$ so that all critical values of subcritical index are below $c$, and all of critical index are above than $c$, and we represent $(W, \lambda, \p)$ as a Legendrian link $\Lambda \sse (Y_0,\la)$ with a front projection. Here we have denoted $(Y_0,\xi) := (\p^{-1}(c),\la)$ and $\Lambda$ is the collection of attaching spheres of all critical points of critical index.

Let $S^n \sse (Y_0,\xi)$ be the belt sphere of a \emph{zero-framed} $(n-1)$--handle, and let $\Lambda_0$ be a component of $\Lambda$ which intersects this belt sphere $S^n$ transversely in a single point. Then the following cancellation move is allowed to the Legendrian handlebody diagram:
\begin{itemize}
\item[-] First, erase the subcritical handle $S^n$.\\

\item[-] Then, near any of the points $p \in \Lambda \cap S^n$ of the intersection which are not in the cancelling component $\Lambda_0$, we cap off $\Lambda \sm \{p\}$ with a parallel copy of $\Lambda_0$.\\
This capping depicted in Figure \ref{fig:cancel}.\\

\item[-] Finally, erase the Legendrian $\Lambda_0$.
\end{itemize}
The resulting diagram describes a Weinstein structure $(W,\wt\lambda,\wt\p)$ on $W$ which is Weinstein equivalent to the original $(W,\lambda,\p)$.
\end{prop}

\begin{figure}[h!]
\centering
  \includegraphics[scale=0.6]{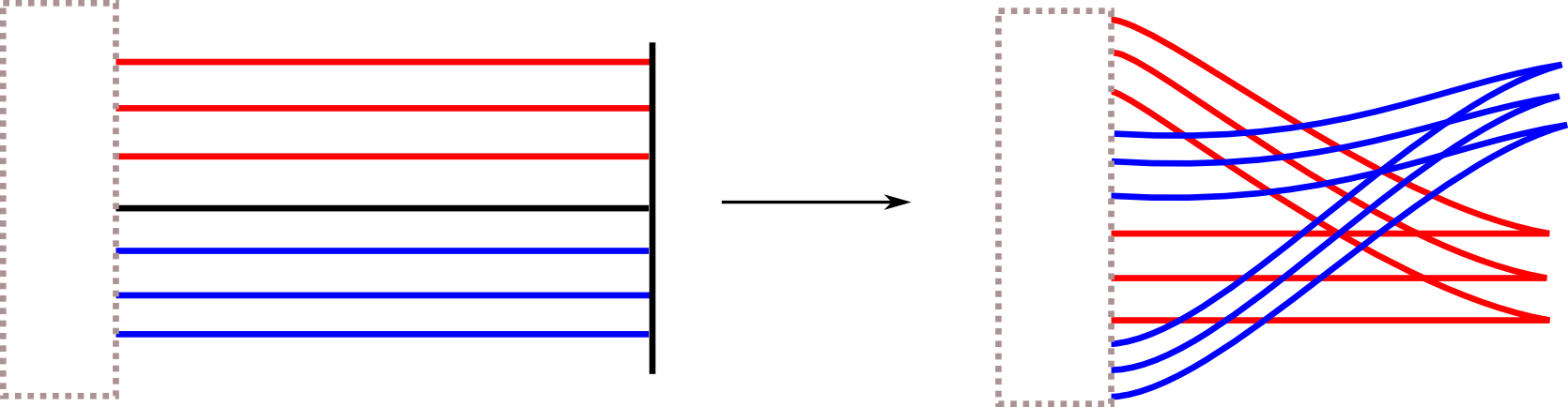}
  \caption{A handle cancellation between the black Legendrian and the subcritical handle.}
\label{fig:cancel}
\end{figure}

\begin{figure}[h!]
\centering
  \includegraphics[scale=0.6]{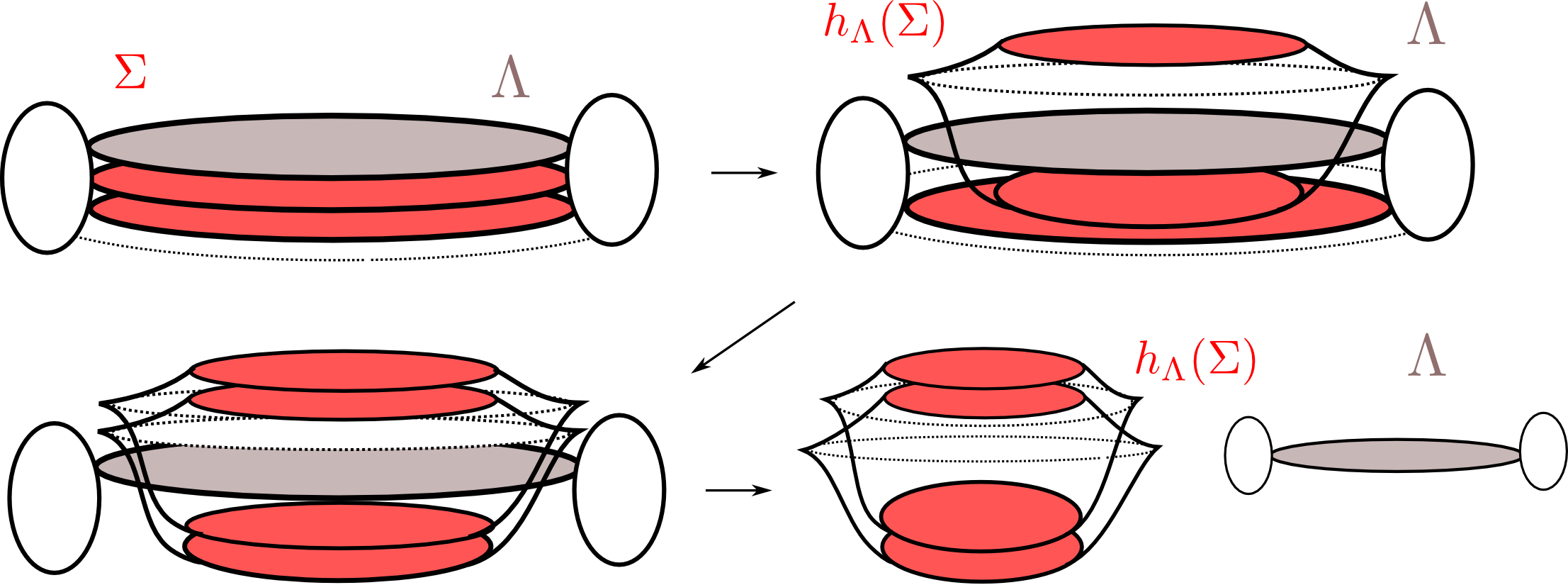}
  \caption{Cancellation Move II: the 2--handle represented by its attaching isotropic 1--sphere is cancelled with the grey Legendrian 2--sphere, which corresponds to a Weinstein 3--handle attachment. In order to proceed the two red Legendrian 2--spheres must be slid off the 2--handle, this is the sequence depicted in this figure. In the fourth step the remaining Legendrian grey 2--sphere and the isotropic 1--sphere can be erased as in Figure \ref{fig:CancelMove2}.}
\label{fig:CancelMove}
\end{figure}

\begin{proof}
Except for $\Lambda_0$, for every piece of the Legendrian $\Lambda$ which passes through $S^n$ we perform a handle slide over $\Lambda_0$. This displaces these pieces from a neighborhood of the sphere $S^n$, and it is depicted in Figure \ref{fig:CancelMove}. Once $\Lambda\sm\Lambda_0$ has been Legendrian slid disjoint from the sphere $S^n$, the Legendrian $\Lambda$ intersects $S^n$ at a unique intersection point in $\Lambda_0$ and thus we can erase the two handles by using Proposition \ref{prop:basic cancel}. These are the three steps listed in the statement of the Proposition, and thus the sequence realizes an equivalence of Legendrian handlebodies.\end{proof}

Legendrian handle slides and Weinstein handle cancellation constitute the basic set of moves in the front. In Section \ref{ssec:stack} we will be using them to obtain Legendrian lifts of exact Lagrangians described as words in Dehn twists, however there is an additional operation that we can perform to a Legendrian front which we now address.

\subsection{Lagrangian cobordisms in the front}\label{ssec:lagcob}
Consider a Weinstein manifold $(W, \lambda, \p)$ presented as Legendrian handlebody, in this subsection we discuss methods of constructing exact Lagrangian submanifolds $\ol L\sse (W,\la)$ combinatorially based on the front projection of the Legendrian handle attaching maps. Let us fix notation and suppose that the Legendrian handle decomposition for $(W,\la,\p)$ is given by a subcritical Weinstein domain $(W_0,\la_0,\p_0)$ together with a Legendrian link of spheres $\Lambda \sse (\dd W_0,\la_0)$ such that by attaching critical handles to $(W_0,\la_0,\p_0)$ along $\Lambda$ we obtain $(W,\la,\p)$.

In order to construct closed exact Lagrangians $\ol L\sse (W,\la)$, the first observation is that an exact Lagrangian $L \sse (W_0,\la_0)$ such that $\dd L = \Lambda$, or some subset of the components of the Legendrian link $\Lambda$, defines a closed exact Lagrangian $\ol L \sse W$. This Lagrangian $\ol L$ is constructed by taking the union of the exact Lagrangian $L$ with the Lagrangian cores of the handles attached to their boundary $\dd L$. Thus, constructing exact Lagrangian fillings of $\La$ yield closed exact Lagrangians in $(W,\la)$. Fortunately, there is a set of combinatorial moves on Legendrian fronts which are induced by exact Lagrangian cobordisms.

\begin{definition}[Legendrian Surgery]\label{def:legsurgery}
Let $\Lambda \sse (Y^{2n-1}, \xi)$ be a Legendrian in a contact manifold, and suppose that there is a Darboux chart $U \sse (Y,\xi)$ such that the front projection of the Legendrian piece $\Lambda \cap U$ contains a subset of cusp singularities whose topology is that of a sphere $S^k \sse \Lambda \cap U$, for some index $1\leq k\leq n-2$. Suppose also that there exists an isotropic disk $D^{k+1} \sse U$ with a smooth front projection which is disjoint from the complement $\Lambda\sm S^k$, bounds the sphere of cusps $S^k = \dd D^{k+1}$ and intersects it approaching from the outside of the cusp. Then define the Legendrian surgery $\wt\Lambda\sse (Y^{2n-1},\xi)$ of $\La$ along $D^{k+1}$ as follows.

Consider a small neighborhood of the isotropic disk $D^{k+1}$ and extend it to a Legendrian ribbon $R = D^{k+1} \x D^{n-k-2}$ such that the Legendrian ribbon $R$ also has a smooth front projection, the piece of the boundary $S^k \x D^{n-k-2} \sse \dd R$ is completely contained in the cusp singular set of $\Lambda$, and anywhere else the ribbon $R$ is disjoint from $\Lambda$. This Legendrian ribbon is part of a Legendrian tube $\wt R$, diffeomorphic to $D^{k+1} \x S^{n-k-2}$, which is formed by gluing together $R$ and a small vertical pushoff of itself along the piece $D^k \x \dd D^{n-k-2}$ forming cusp--edge singularities. Then the Legendrian $\wt\Lambda$ obtained by deleting a small neighborhood of $S^k \x D^{n-k-2}$ from $\Lambda$, and gluing the resulting boundary to $\wt R$.\hfill$\Box$
\end{definition}

\begin{figure}[h!]
\centering
  \includegraphics[scale=0.8]{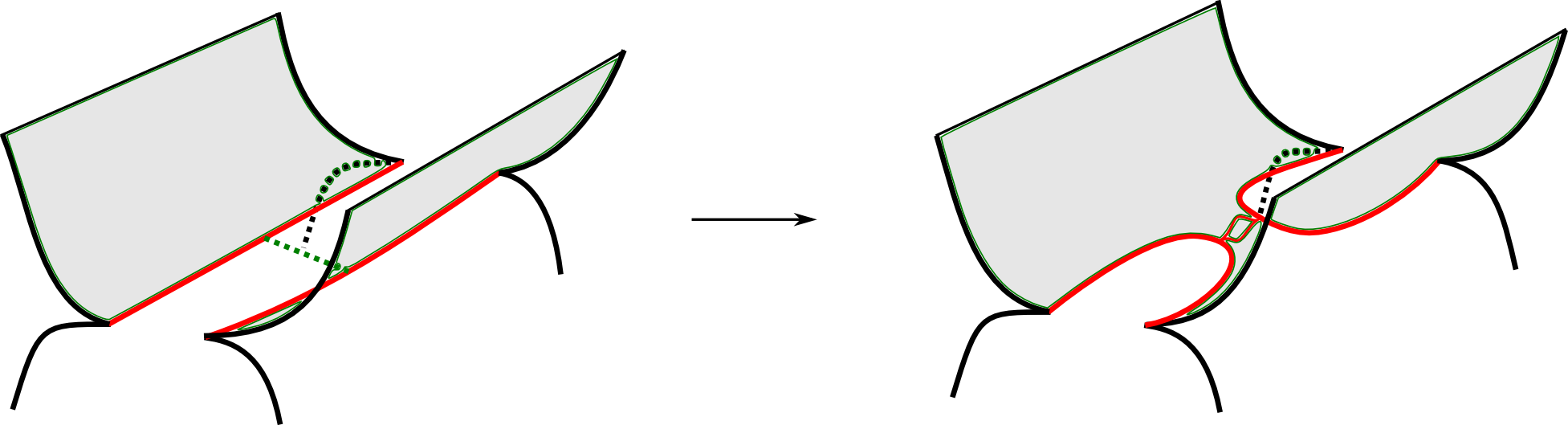}
  \caption{Legendrian surgery depicted in the case of $\Lambda$ being a Legendrian surface in a Darboux ball and $k=0$. The Legendrian $S^0\times D^2$ is modified by surgery along the green path to a Legendrian cylinder $S^1\times D^1$.}
\label{fig:legsurgery}
\end{figure}

The above definition might be elaborate to read, but the geometric idea is simple: it is the Legendrian analogue of a smooth surgery where vertical tangencies are substituted by cusp edges; see Figure \ref{fig:legsurgery} for an example and the discussions in \cite{BST} for more details. Since the ambient manifold $(Y,\xi)$ remains unchanged, this operation is also referred to as ambient Legendrian surgery. The following result can be proven with elementary means in contact topology, generating functions being one of them: it tells us that Legendrian ambient surgeries can be realized by Lagrangian cobordisms:

\begin{thm}\label{thm:lagcob}
The following three statements hold:
\begin{itemize}
 \item[1.] Legendrian isotopies are induced by Lagrangian cobordisms, whose smooth topology is that of a cylinder $[0,1] \x \Lambda$ \cite{EG}.
 \item[2.] Legendrian ambient surgeries are induced by upward Lagrangian cobordisms, whose smooth type is that of a single handle attachment of the corresponding index \cite{BST,DR}.
 \item[3.] The standard Legendrian unknot is the boundary of a Lagrangian disk.\hfill$\Box$
\end{itemize}
\end{thm}

\begin{remark}
We emphasize that Lagrangian cobordisms and Legendrian ambient surgeries carry an important directionality: the inverse of a Legendrian ambient surgery is not a Legendrian ambient surgery, and inverse surgeries are not generally induced by {\it upwards} Lagrangian cobordisms. Instead, inverse surgeries are induced by \emph{downward} cobordisms.\hfill$\Box$
\end{remark}

Note that, while a Lagrangian filling of a Legendrian $\Lambda$ builds a closed Lagrangian $\ol L$ inside the manifold $(W,\la)$ obtained by attaching handle(s) to $\Lambda$, a Lagrangian filling of $\Lambda$ together with a Legendrian pushoff does not. Therefore, every closed Lagrangian $\ol L$ built using this method will have the property that $[\ol L] \in H_n(W)$ is a primitive class.

The Lagrangian cobordisms described in Theorem \ref{thm:lagcob} suffice for our purposes, and they can be used when trying to construct exact Lagrangian submanifolds as, for instance, in Theorems \ref{thm:Xab1} and \ref{thm:torusIntro}.

\begin{remark}
There are two main open questions concerning whether the above constructions are sufficient to construct all exact Lagrangians.

First, there is a question of whether every closed exact Lagrangian $\ol L$ in a Weinstein manifold $(W,\la,\p)$ can be built out of a Lagrangian filling of the Legendrian $\Lambda \sse (W_0,\la_0)$. This can be phrased either in terms of a fixed handle presentation of $(W,\la,\p)$ given by $(W_0,\la_0,\p_0)$ and $\Lambda$, or varying over all possible presentations; both questions are open. As explained above, every Lagrangian constructed in this way has a primitive fundamental class in $H_n(W)$ and it is also an open question whether this holds for all closed exact Lagrangians in Weinstein manifolds.

Second, a separate open question is whether every exact Lagrangian filling of a Legendrian $\Lambda \sse (W_0,\la_0)$ can be built from Legendrian isotopies and ambient surgeries. Lagrangians which can be built in this way are exactly those Lagrangians which are \emph{regular} in the sense of \cite{EGL}. It is an open question whether every Lagrangian in a Weinstein manifold is regular, and this is (if true) stronger than all previous questions: regular Lagrangians are exactly those which can be built from isotopies, ambient surgeries, and cores of ambient handles.

Finally we remark that all of these questions are known to have counterexamples if we allow either Lagrangians with concave boundary inside the boundary of a Darboux ball, or closed Lagrangians inside Weinstein cobordisms with nonempty concave contact boundary \cite{EM}. Therefore, if it is the case that all exact Lagrangians are of the various forms described above, it must be for an essentially global reason.\hfill$\Box$
\end{remark}

This concludes our discussion on Lagrangian cobordisms, and completes the basic set of ingredients required to perform Legendrian calculus once the Legendrian handlebody of a Weinstein manifold is obtained. It is thus time to explain how to obtain the fronts of the Legendrian attaching spheres of Weinstein manifolds.

\subsection{Legendrian stacking}\label{ssec:stack}

It is now our aim to translate the picture of Weinstein Lefschetz fibrations, and bifibrations, into the language of Weinstein handlebodies. Both Lefschetz fibrations \cite{GP,McLean,Se08} and bifibrations \cite{May,MS,Se15} have been thoroughly studied in the literature but to the authors' knowledge the transition towards explicit Legendrian handlebodies remained unexplored. We find such connection fundamental, and we develop it here.

The Legendrian attaching spheres in the Legendrian handlebody will appear as Legendrian lifts of exact Lagrangian spheres, which belong to a regular fiber $F_\pi$ of a given Weinstein Lefschetz fibration $\pi:W\lr\C$. In the vast majority of known cases, these Lagrangian spheres appear as words in symplectic Dehn twists on a finite set of known Lagrangian spheres, and it is of central interest to understand the Legendrian lifts of a Lagrangian presented in this manner. The following proposition addresses a core instance of such question.

\begin{prop}\label{prop:stacking}
Let $(Y,\xi)= \ob(F,\lambda;\phi)$ be a contact manifold, and $S,L \sse (F,\la)$ two exact Lagrangian submanifolds such that $S$ is diffeomorphic to a sphere. Suppose that the potential functions for $\lambda|_S$ and $\lambda|_L$ are $C^0$--bounded by a small enough $\e\in\R^+$, and consider the contact manifold $(\wt Y,\wt\xi)$ obtained by performing $(+1)$--surgery along $\Lambda_S^\e$ and $(-1)$--surgery along $\Lambda_S^{5\e}$. Then
\begin{itemize}
 \item[-] There exists a canonical contact identification $(Y,\xi)\cong(\wt Y,\wt\xi)$.
 \item[-] The Legendrian $\Lambda^{3\e}_L \sse (\wt Y,\wt\xi)$ is Legendrian isotopic to $\Lambda^0_{\tau_SL} \sse (Y,\xi)$.
\end{itemize}

In an analogous manner, performing contact $(-1)$ and $(+1)$--surgeries along $\Lambda_S^\e$ and $\Lambda_S^{5\e}$ in $(Y,\xi)$ results in a contact manifold $(\overline Y,\overline\xi)$ with a contact identification $(\overline Y,\overline\xi)\cong(Y,\xi)$ under which $\Lambda_L^{3\e} \sse(\overline Y,\overline\xi)$ is Legendrian isotopic to $\Lambda_{\tau_S^{-1}L} \sse (Y,\xi)$.
\end{prop}

\begin{remark}
Note that the hypothesis on the bound of the potential functions on the Lagrangians can be arranged with a Hamiltonian isotopy.
\end{remark}

\begin{proof}
The contactomorphism $(Y,\xi)\cong(\wt Y,\wt\xi)$ is a consequence of Proposition \ref{prop:basic cancel} above, the interesting statement is the existence of the Legendrian isotopy between $\Lambda^{3\e}_L$ and $\Lambda^0_{\tau_SL}$, which we now prove. Consider the contact open book presentation
$$(\wt Y,\wt\xi) = \op{ob}(F, \lambda, (\p \circ \tau_S)\circ \tau_S^{-1}),$$
where we are building $(\wt Y,\wt\xi)$ as a quotient of the contactization $F \x ([0,1] \cup [2,3])$, where $F \x \{1\}$ is identified with $F \x \{2\}$ by the symplectomorphism $\tau_S^{-1}$, and $F \x \{3\}$ is identified with $F \x \{0\}$ by the symplectomorphism $\p \circ \tau_S$, and then gluing the standard contact neighborhood $\partial F\times D^2$ of the binding.

In this presentation, the Legendrian $\Lambda_{\tau_SL}$ is Legendrian isotopic to the Legendrian lift
$$\tau_S(L) \x \left\{\frac12 + h_{\tau_S(L)}\right\}\sse \op{ob}(F, \lambda, (\p \circ \tau_S)\circ \tau_S^{-1}),$$
where $h_K: K \lr (-\e, \e)$ denotes the potential function of an exact Lagrangian $K \sse (F,\la)$. By translating in the positive Reeb direction, we also realize the Legendrian isotopy
$$\tau_S(L) \x \left\{\frac12 + h_{\tau_SL}\right\}\simeq L \x\left\{\frac52 + h_L\right\}.$$
The Legendrian in the right hand side of the isotopy is itself Legendrian isotopic to $\Lambda_L^{3\e}$, since it lies both above the region where the gluing with the symplectomorphism $\tau_S^{-1}$ occurs and below the region where the gluing with $\tau_S$ occurs. This proves the required Legendrian isotopy $\Lambda^{3\e}_L\simeq\Lambda^0_{\tau_SL}$.
\end{proof}

\begin{remark}
Proposition \ref{prop:stacking} provides enough tools to prove new results, as illustrated by the applications in Section \ref{sec:app}, and it suffices for the purposes of the present article. However, we want to remark that the following general question is also of interest. Let $S$ and $L$ be exact Lagrangians in a Liouville manifold $(F, \lambda)$ where $S$ is a sphere, and let $(Y,\xi)$ be a contact open book with Liouville page $(F,\la)$. Suppose that we understand the Legendrian lifts of $S$ and $L$, then how do we understand the Legendrian lift of $\tau^{\pm 1}_S(L)$ ?

In the theory of Legendrian handlebodies understanding a Legendrian means being able to draw a picture of its front projection, since from there we can isotope it freely and compute its invariants. However, to begin with the front projection is only defined for $(\R^{2n-1},\xi_\st)$, or particular contact manifolds such as 1--jet spaces $(J^1(Q),\xi_\st)$, and thus constructing the front projection for a contact manifold $(Y,\xi)=\ob(F,\la;\phi)$, even if presented with an open book, takes work. In the $3$--dimensional case, a general systematic method for making open books, front projections, and Legendrian lifts is discussed in the article \cite{GL}.\hfill$\Box$
\end{remark}

Proposition \ref{prop:stacking} being proven, we can use Proposition \ref{prop:slide} in order to obtain the Legendrian lift $\La_L\sse\ob(F,\la;\id)$ of an exact Lagrangian $L\sse(F,\la)$ which is Hamiltonian isotopic to
$$\tau^{\e_1}_{S_1}\circ\tau^{\e_2}_{S_2}\circ\cdots\circ\tau^{\e_n}_{S_n}(S_{n+1})$$
where $S_i$ are Lagrangian spheres in $(F,\la)$, possibly equal, and $\e_i=\pm1$ for $1\leq i\leq n$. In the reminder of this section and the subsequent Subsection \ref{ssec:Tpqr} we provide several explicit examples of such Legendrian lifts that might serve the reader as a good gauge for his understanding of the material presented thus far.

\begin{figure}[h!]
\centering
  \includegraphics[scale=0.7]{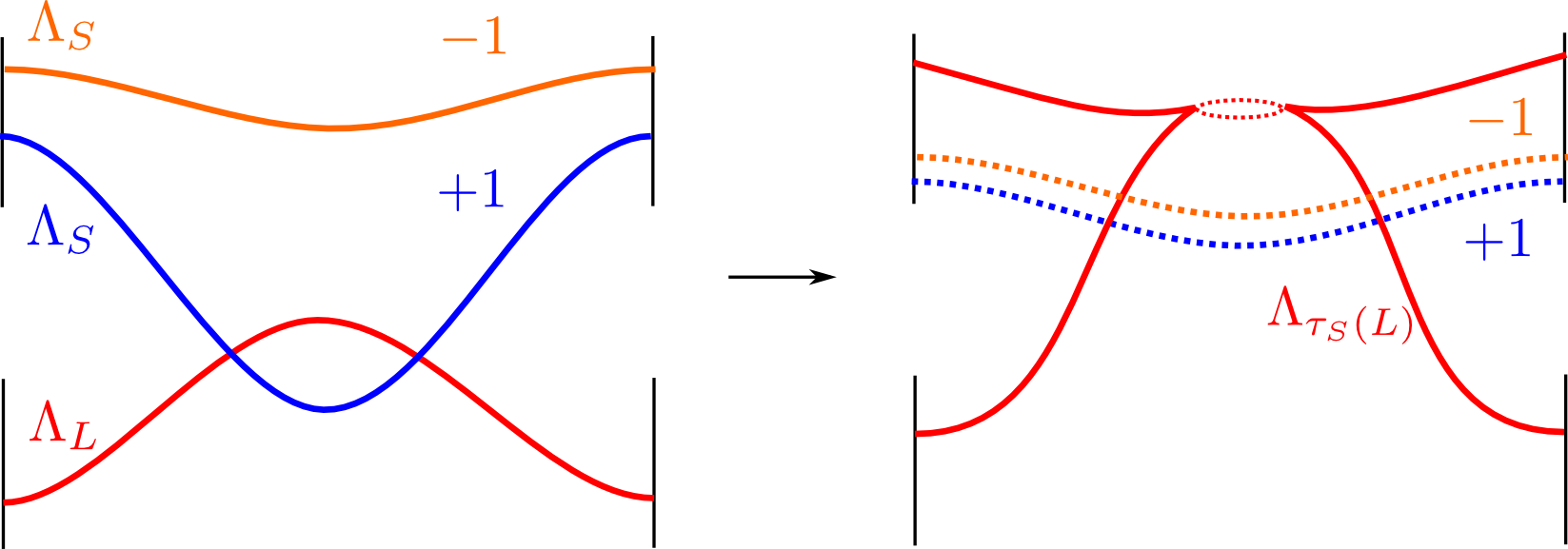}
  \caption{The Legendrian lift of $\tau_S(L)$.}
\label{fig:cusp}
\end{figure}

\begin{figure}[h!]
\centering
  \includegraphics[scale=0.7]{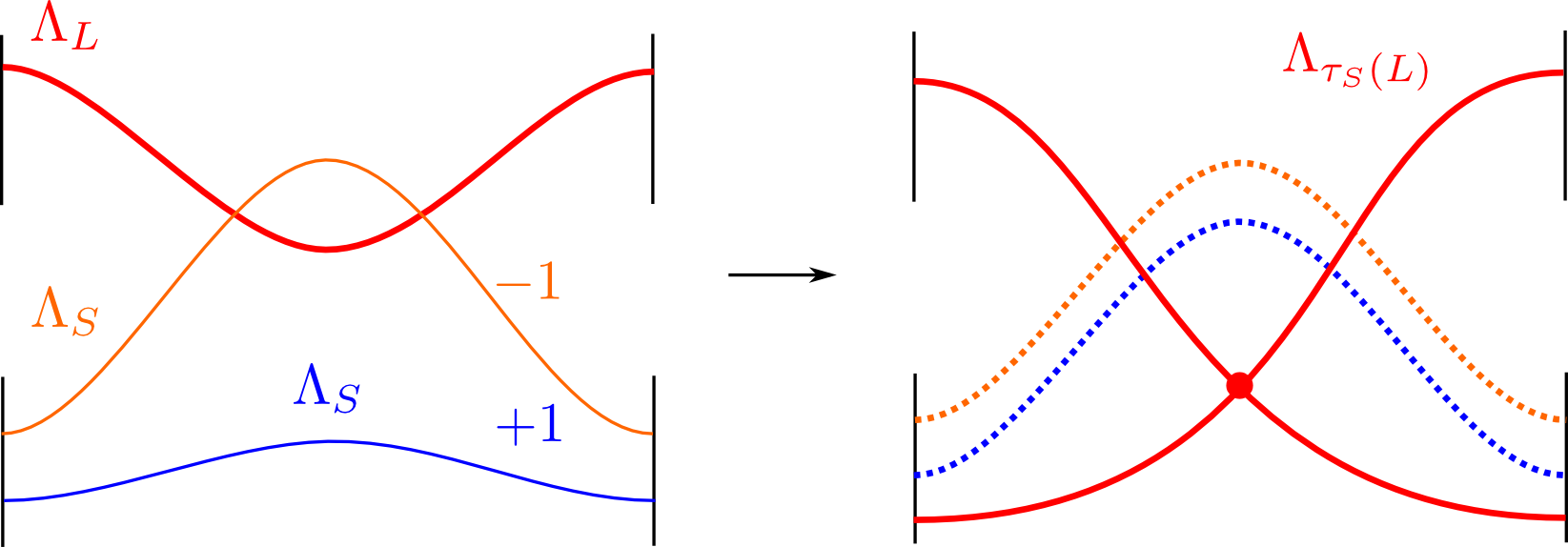}
  \caption{Another Legendrian lift of $\tau_S(L)$.}
\label{fig:cone}
\end{figure}

\begin{ex}\label{ex:most basic}
Consider two exact Lagrangians $S, L \sse (F,\la)$, with $S$ diffeomorphic to a sphere and intersecting transversely in a single point. Suppose we choose Legendrian lifts $\Lambda_S$ and $\Lambda_L$ of $S$ and $L$ respectively, so that the height of $\Lambda_S$ is everywhere larger than the height of $\Lambda_L$, except at the intersection point where they coincide. Then we apply Proposition \ref{prop:slide} and Proposition \ref{prop:stacking} to conclude that the Legendrian lift of $\tau_S(L)$ has a Legendrian front as depicted in Figure \ref{fig:cone}. This Legendrian lift is obtained by connecting the Legendrian lift $\Lambda_S$ to the Legendrian $\Lambda_L$ by the tube, diffeomorphic to the product of a $\dim(S)$--sphere and an interval $D^1$, with a cusp--edge along the central sphere $\dim(S)$--sphere.

Alternatively, suppose that the height of $\Lambda_S$ is everywhere less than the height of $\Lambda_L$, except at the intersection point. Then Proposition \ref{prop:slide} and Proposition \ref{prop:stacking} imply that the Legendrian lift of $\tau_S(L)$ is as depicted in Figure \ref{fig:cone}. This Legendrian lift consists of connecting the Legendrian sphere $\Lambda_S$ to the Legendrian $\Lambda_L$ by using a tube whose front contains a cone singularity.

Thus we have learnt how to draw the front projection of $\tau_S(L)$ depending on the initial conditions dictated the location of the Legendrians $\La_S$ and $\La_L$. Note also that in the case that the exact Lagrangian $L$ is also a sphere, we have the following Lagrangian isotopy $\tau_S^{-1}(L) \simeq \tau_L(S)$. Therefore the Legendrian lift of $\tau^{-1}_S(L)$ also corresponds to connecting the Legendrians $\Lambda_S$ and $\Lambda_L$ with tubes either through cusps--edges or a cone singularity, but in this case with the opposite convention regarding the $z$--height of the Legendrian lifts. These local models are summarized in the following table:
\end{ex}

\begin{table}[h!]
\begin{center}
  \begin{tabular}{| c | c | c |}
    \hline
       & $z(\Lambda_S) > z(\Lambda_L)$ & $z(\Lambda_S) < z(\Lambda_L)$ \\ \hline
    $\tau_S(L)$ & cusp & cone \\ \hline
    $\tau_S^{-1}(L)$ & cone & cusp \\
    \hline
  \end{tabular}
  \caption{Legendrian lift of $\tau_S^{\pm 1}(L)$.}\label{table}
\end{center}
\end{table}

\begin{remark}
Despite the subtlety of signs inherent in Table \ref{table}, we emphasize that Legendrian handle slides do not depend on relative height: as stated in Proposition \ref{prop:slide} a Legendrian handle slide over a $(-1)$--surgery Legendrian always involves a spherical cusp--edge, and a Legendrian handle slide over a $(+1)$--surgery Legendrian always yields a cone singularity.\hfill$\Box$
\end{remark}

Being comfortable with the Legendrian lifts of the Lagrangians $\tau_S^{\pm}(L)$, we now proceed to the understanding of the Legendrian lifts of $\tau_S^{2}(L)$.

\begin{ex}\label{ex:square}
Suppose again that $S,L\sse (F,\la)$ are Lagrangian spheres intersecting transversely in a single point, and assume as well that coordinates are chosen such that the Legendrian lift $\Lambda_S$ of the Lagrangian sphere $S$ is everywhere above the Legendrian lift $\Lambda_L$ of the Lagrangian $L$. Then using Proposition \ref{prop:stacking} iteratively obtain the computation for the Legendrian lift of the Lagrangian $\tau^2_S(L)$ shown in Figure \ref{fig:square}.

Analogously, by iteratively using Proposition \ref{prop:stacking} the reader should now be able to depict the Legendrian fronts of the Legendrian lifts of the exact Lagrangians $\tau^{\pm k}_S(L)$ for any $k\in\N$; this computation is not required in this article but it is an instructive exercise.\hfill$\Box$
\end{ex}

\begin{figure}[h!]
\centering
  \includegraphics[scale=0.7]{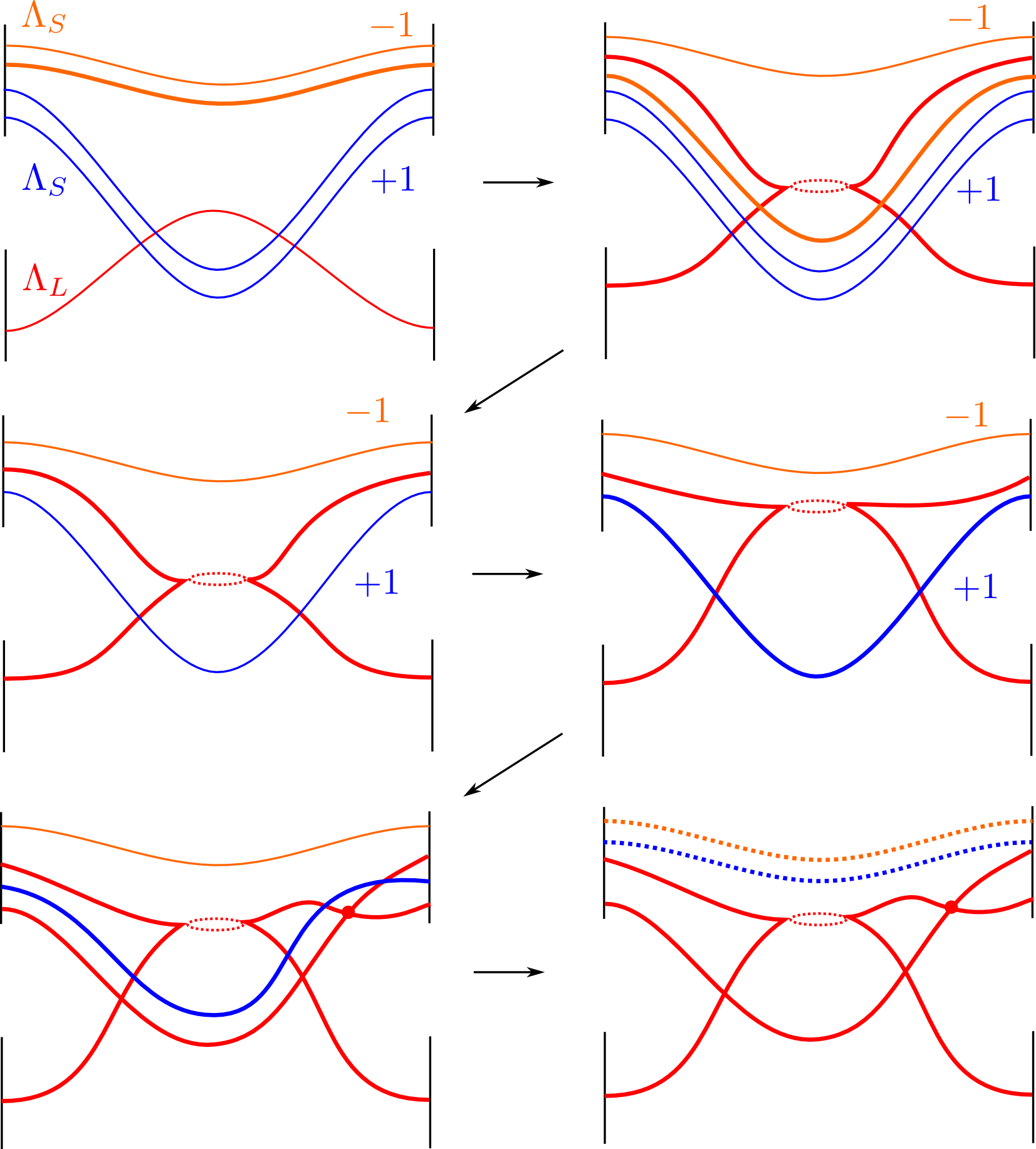}
  \caption{The Legendrian lift of $\tau^2_S(L)$.}
\label{fig:square}
\end{figure}

The previous two examples constitute local models for the Legendrian fronts, let us start using this for the Weinstein manifolds discussed in Subsection \ref{ssec:subcrit}: suppose that the fiber $(F,\la)= (F_T,\la_T)$ is a $T$--plumbing of standard cotangent bundles $(T^*S^{n-1},\la_\st)$. By iterating Proposition \ref{prop:stacking}, we are able to draw a front for the Legendrian lift of any Lagrangian sphere in the Weinstein manifold $(F_T,\la_T)$ expressed as a sequence of Dehn twists along the various zero sections of its $T$--plumbing structure. Note that as explained in section \ref{ssec:subcrit} the subcritical Weinstein manifold $(F_T\times D^2,\la_T+\la_\st)$ will be the subcritical skeleton of the Weinstein manifolds we are interested in, Section \ref{sec:app} contains many instances of this.

The next example we consider is the plumbing of only two cotangent bundles, which also serves as a local model for any pair of exact Lagrangian spheres intersecting at a point.

\begin{ex}\label{ex:exampleX12}
Consider the tree $T=A_2$ consisting of two Lagrangian spheres $S$ and $L$ intersecting transversely at a point, and let us construct the Weinstein manifold obtained by attaching three critical Weinstein handles to the subcritical piece $(F_T\times D^2,\la_T+\la_\st)$. We choose to fix two of the three critical handles to be attached along the Legendrian lifts $\La_S,\La_L\sse\dd(F_T\times D^2)$ of the Lagrangian spheres $S,L\sse (F_T,\la_T)$, and let us declare the third handle to be attached along the Legendrian lift $\La_{\tau_S(L)}$ of the Lagrangian sphere $\tau_S(L)\sse (F_T,\la_T)$. The unordered list of Legendrian attaching spheres is thus $\{\La_S,\La_L,\La_{\tau_S(L)}\}$, and we will see now how the choice of order crucially affects the resulting Weinstein structure.

First, consider the case in which the cyclically ordered set of attachments is $(\La_S,\La_L,\La_{\tau_S(L)})$. Figure \ref{fig:TS2} depicts a sequence, the first stage of which is a Legendrian front representing these three critical Weinstein attachments. The subsequent three stages are the result of applying the Legendrian front calculus described in Subsections \ref{ssec:highD} and \ref{ssec:Wmoves}. These stages consist respectively in a handle slide of the red Legendrian $\La_{\tau_S(L)}$ along the blue Legendrian $\La_L$, performed in order to be able to cancel $\La_L$ with the unique subcritical handle that it intersects, two Reidemeister moves pulling the sphere of cusps and a handle slide of $\La_{\tau_S(L)}$ along the yellow Legendrian $\La_S$. The Legendrian front in the fourth step describes the Weinstein manifold $(T^*S^n,\la_\st,\p_\st)$ obtained by attaching a Weinstein handle along the Legendrian unknot, since the two subcritical handles can be cancelled respectively with the Legendrians $\La_S$ and $\La_L$ using Proposition \ref{prop:basic cancel}.

\begin{figure}[h!]
\centering
  \includegraphics[scale=0.7]{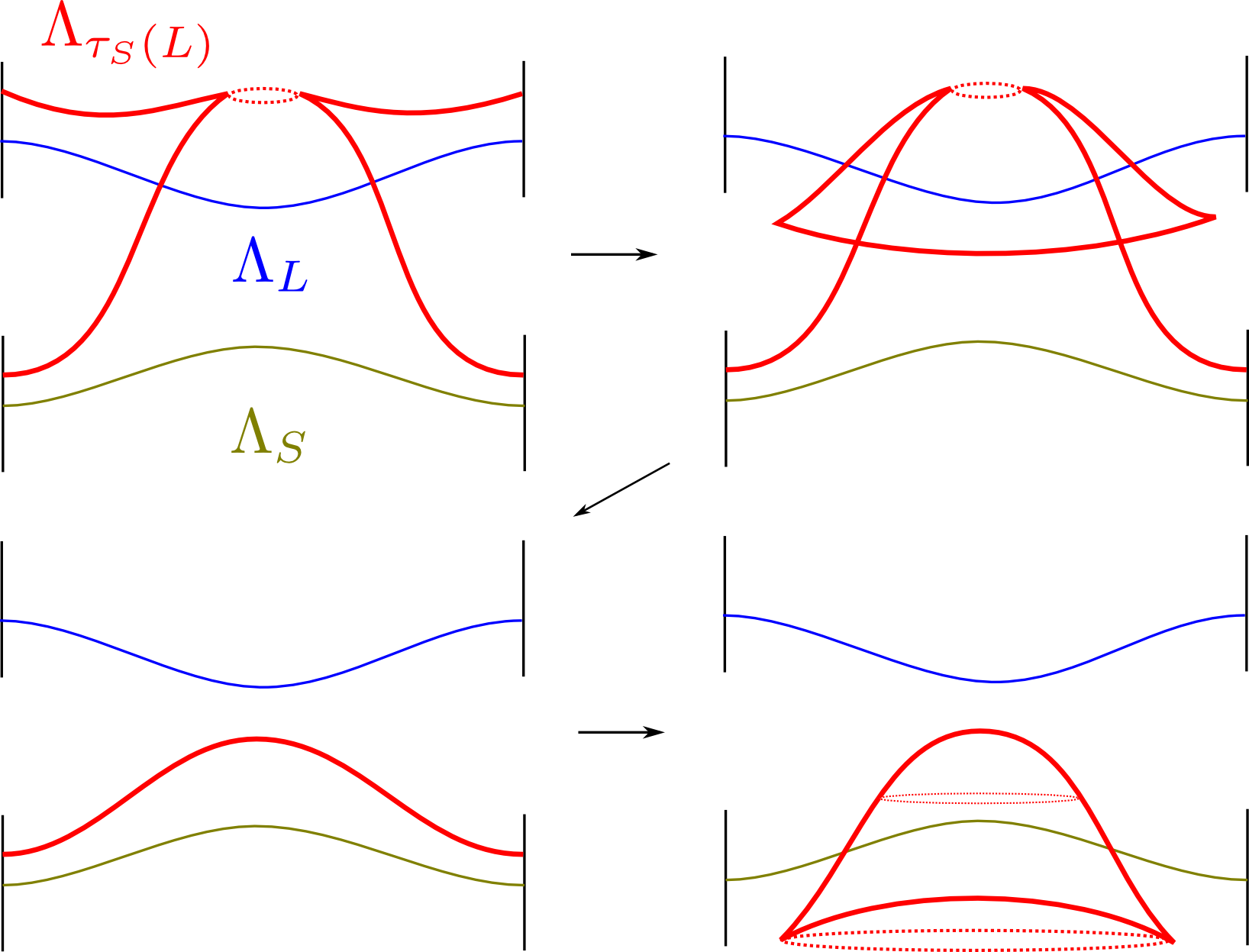}
  \caption{The Weinstein manifold obtained by critical handle attachments to $\dd(F_{A_2},\la_{A_2})$ along the ordered triple $(\La_S,\La_L,\La_{\tau_S(L)})$.}
\label{fig:TS2}
\end{figure}

Second, consider the alternative case where the order is given by the cycle $(\La_L,\La_S,\La_{\tau_S(L)})$. The handle attachment is illustrated in Figure \ref{fig:X12} at the beginning of the sequence. Performing handle slides and Reidemeister moves yields the four--stage sequence resulting in a Weinstein manifold described by a unique handle attachment along the red Legendrian sphere. Though this Legendrian is smoothly unknotted, it is not Legendrian isotopic to the Legendrian unknot thanks to our discussion in Subsection \ref{ssec:loose}; instead, the visible loose chart implies that the Weinstein structure we obtain in the smooth manifold $T^*S^n$ is a flexible Weinstein structure and in particular contains no Lagrangian spheres. In particular, it is not symplectomorphic to the standard cotangent bundle $(T^*S^n,\la_\st)$.

\begin{figure}[h!]
\centering
  \includegraphics[scale=0.7]{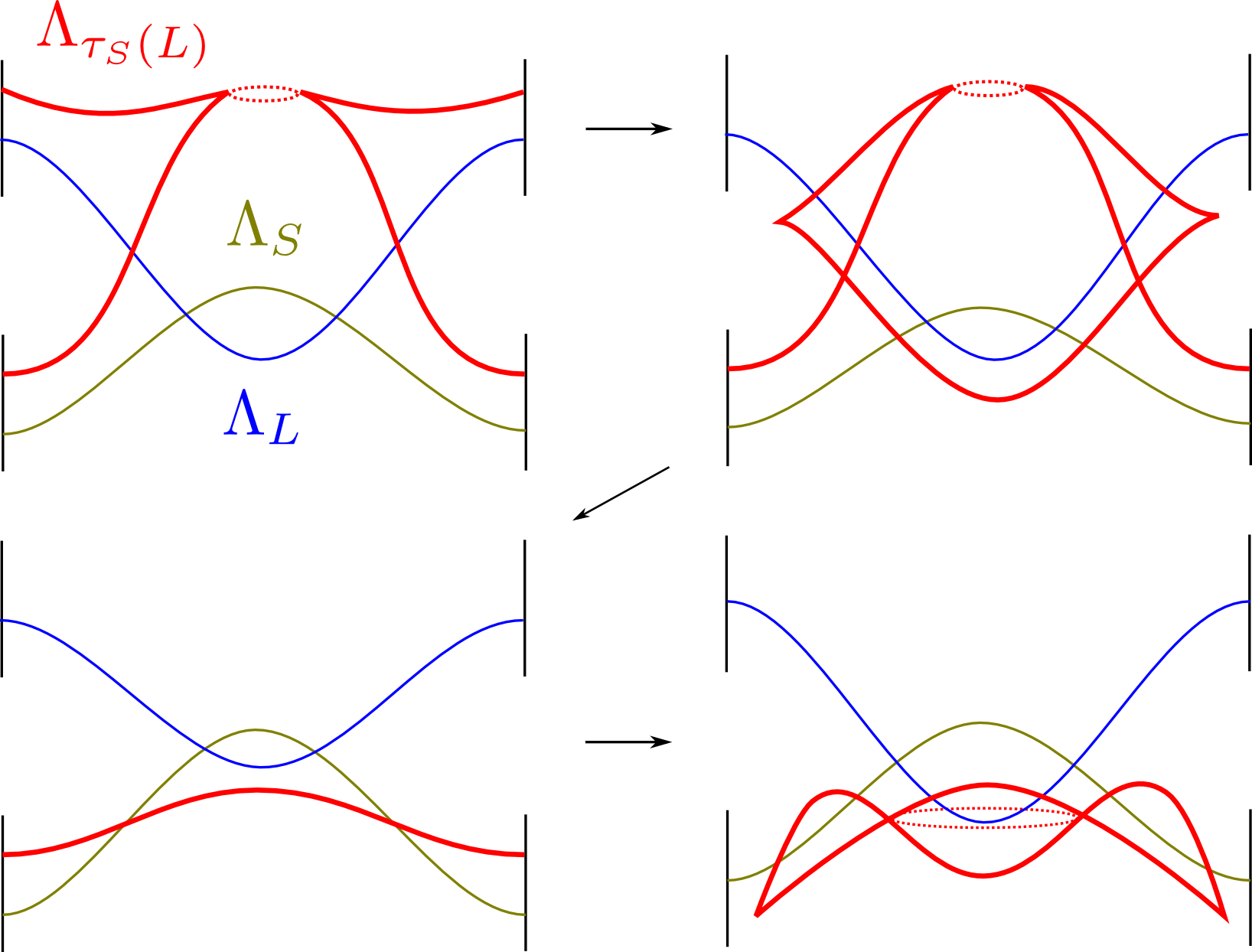}
  \caption{The Weinstein manifold obtained by critical handle attachments to $\dd(F_{A_2},\la_{A_2})$ along the ordered triple $(\La_L,\La_S,\La_{\tau_S(L)})$.}
\label{fig:X12}
\end{figure}

Following this two examples the reader is encouraged to consider the following simple generalization where $T=A_k$ is a linear tree with $|T^{(0)}|=k$ vertices, the subcritical skeleton $(F_T,\la_T)$ is the linear plumbing of $k$ Lagrangian spheres $\{L_i\}$, $1\leq i\leq k$, and the Weinstein handles are attached either as the cycle
$$(L_1,L_2,\ldots,L_{k-1},L_k,\tau_{L_k}\circ\tau_{L_{k-1}}\circ\cdots\circ\tau_{L_3}\circ\tau_{L_2}(L_1))$$
or the reversed cycle
$$(L_k,L_{k-1},\ldots,L_{2},L_1,\tau_{L_k}\circ\tau_{L_{k-1}}\circ\cdots\circ\tau_{L_3}\circ\tau_{L_2}(L_1)).$$
The Weinstein manifold associated to the first cycle is again the standard $(T^*S^n,\la_\st,\p_\st)$, whereas the second cycle yields other interesting flexible Weinstein structures. The latter case is studied later in this article under the name of $X_{1,b}$, which has already appeared in Theorem \ref{thm:Xab1}. \hfill$\Box$
\end{ex}

Proposition \ref{prop:stacking} is one of two central ingredients in Recipe \ref{dictionary}, along with the Lefschetz bifibrations featuring in Section \ref{sec:lef}, and thus it is used constantly in the applications of Section \ref{sec:app}. We now complete our discussion on Legendrian fronts with an additional family of examples, explaining how to draw front diagrams for the Milnor fibers of the $T_{p,q,r}$--singularities. The following Section \ref{sec:lef} shall then present a discussion on Lefschetz bifibrations to which the contents of this current Section \ref{sec:kirby} will be applied in Section \ref{sec:app}.

\subsection{Legendrian front for the $T_{p,q,r}$ Milnor fiber}\label{ssec:Tpqr}
The applications of Section \ref{sec:app} require understanding of Section \ref{sec:lef}, and thus we present this subsection as an application of Proposition \ref{prop:stacking} where the reader does not require the background from Section \ref{sec:lef}. This particular example owes its existence to discussions of the first author with A.~Keating, to whom the authors are very grateful.

Consider the family of Weinstein 4--folds
$$T_{p,q,r}=\{(x,y,z):x^p+y^q+z^r+xyz=1\}\sse\C^3,$$
these are known as the Milnor fibers of the $T_{p,q,r}$--singularities, which constitute one of the modality one families of isolated singularities \cite{AGV,Tpqr}. There exists a Weinstein Lefschetz fibration
$$\pi:T_{p,q,r}\lr\C$$
with $0\in\C$ as a regular value, regular fiber $(F_\pi,\la,\p)=(\pi^{-1}(0),\la_\st,\p_\st)$ symplectomorphic to the thrice punctured torus and $(p+q+r+3)$ distinct critical values.

\begin{figure}[h!]
\centering
\includegraphics[scale=0.5]{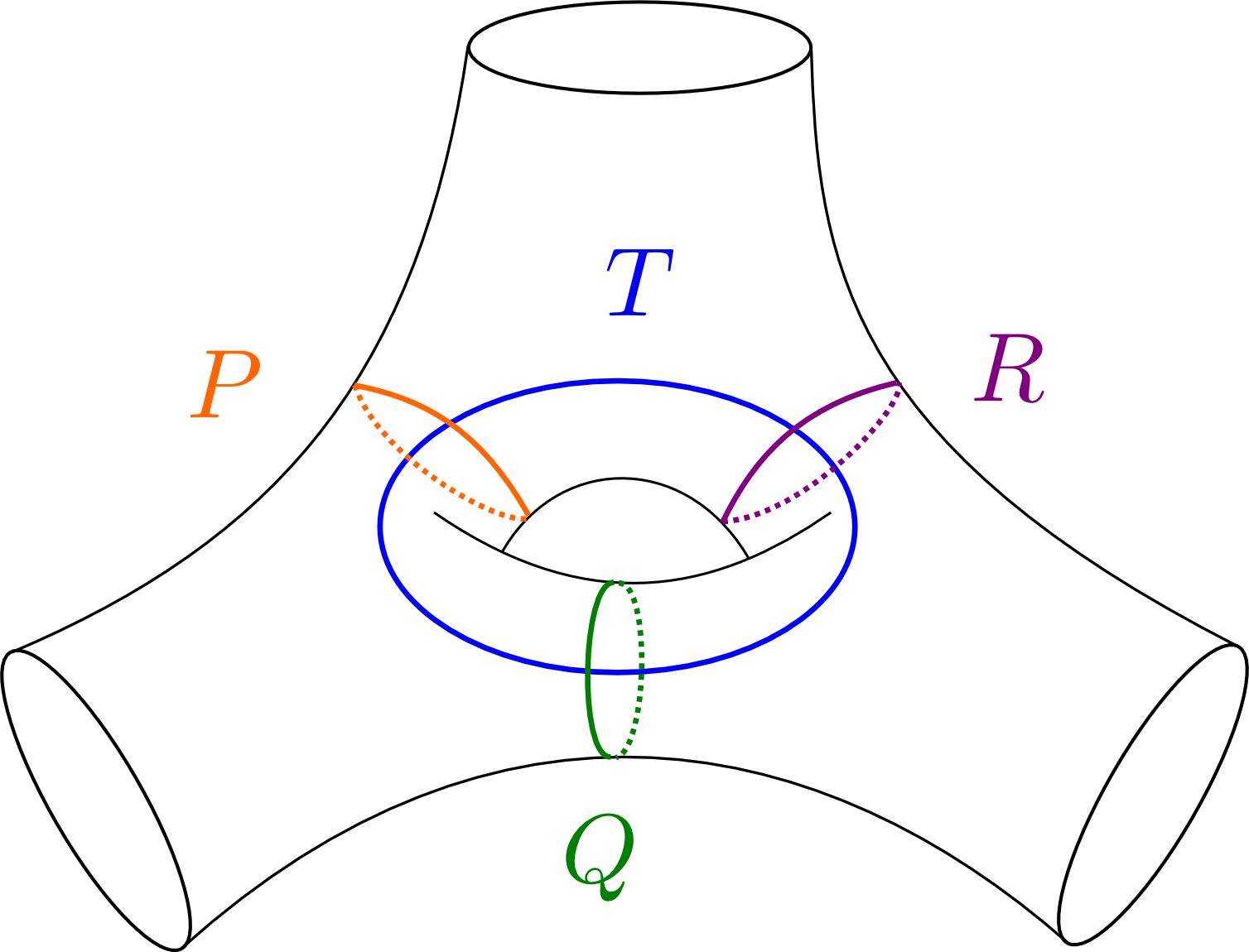}
\caption{The regular fiber and a $\Z$--basis $H_1(\Sigma,\Z)=\langle T,P,Q,R\rangle$.}
\label{fig:Tpqr}
\end{figure}

Consider the four curves $T,P,Q,R\sse F_\pi$ depicted in Figure \ref{fig:Tpqr}, which constitute a $1$--skeleton for the surface $F_\pi$ and we can use them to describe the set of $(p+q+r+3)$ vanishing cycles
$$V_\pi=\{V_{-2},V_{-1},V_{0},V_1,\ldots,V_{p+q+r}\}.$$
The vanishing cycles are the following words in Dehn twists:
$$V_{-2}=T,\quad V_{-1}=\tau^2_P\tau^2_Q\tau^2_R(T),\quad V_0=\tau_P\tau_Q\tau_R(T),\quad V_1=\ldots=V_p=P,$$
$$V_{p+1}=\ldots=V_{p+q-1}=Q,\quad V_{p+q}=\ldots=V_{p+q+r}=R.$$
Once we are given this description, we can use Proposition \ref{prop:stacking} to draw a handle decomposition of the Milnor fiber $T_{p,q,r}$. It is obtained by attaching $(p+q+r+3)$ 2--handles to the subcritical skeleton $(F_\pi\times D^2,\la+\la_\st)$, attached along the Legendrian lifts of the exact Lagrangians vanishing cycles in the set $V_\pi$ in the contact boundary $(F_\pi\times S^1,\la+\la_\st)$. Proposition \ref{prop:stacking} allows us to draw the correct Legendrian fronts for these Legendrian lifts and depict the Legendrian handlebody for $T_{p,q,r}$: we have drawn the first three vanishing cycles $\{V_{-2},V_{-1},V_0\}\sse V_\pi$ in Figure \ref{fig:TpqrFront}. The reader might appreciate here the relevance of Examples \ref{ex:most basic} and \ref{ex:square}, from which this Legendrian fronts are built.

\begin{figure}[h!]
\centering
\includegraphics[scale=0.55]{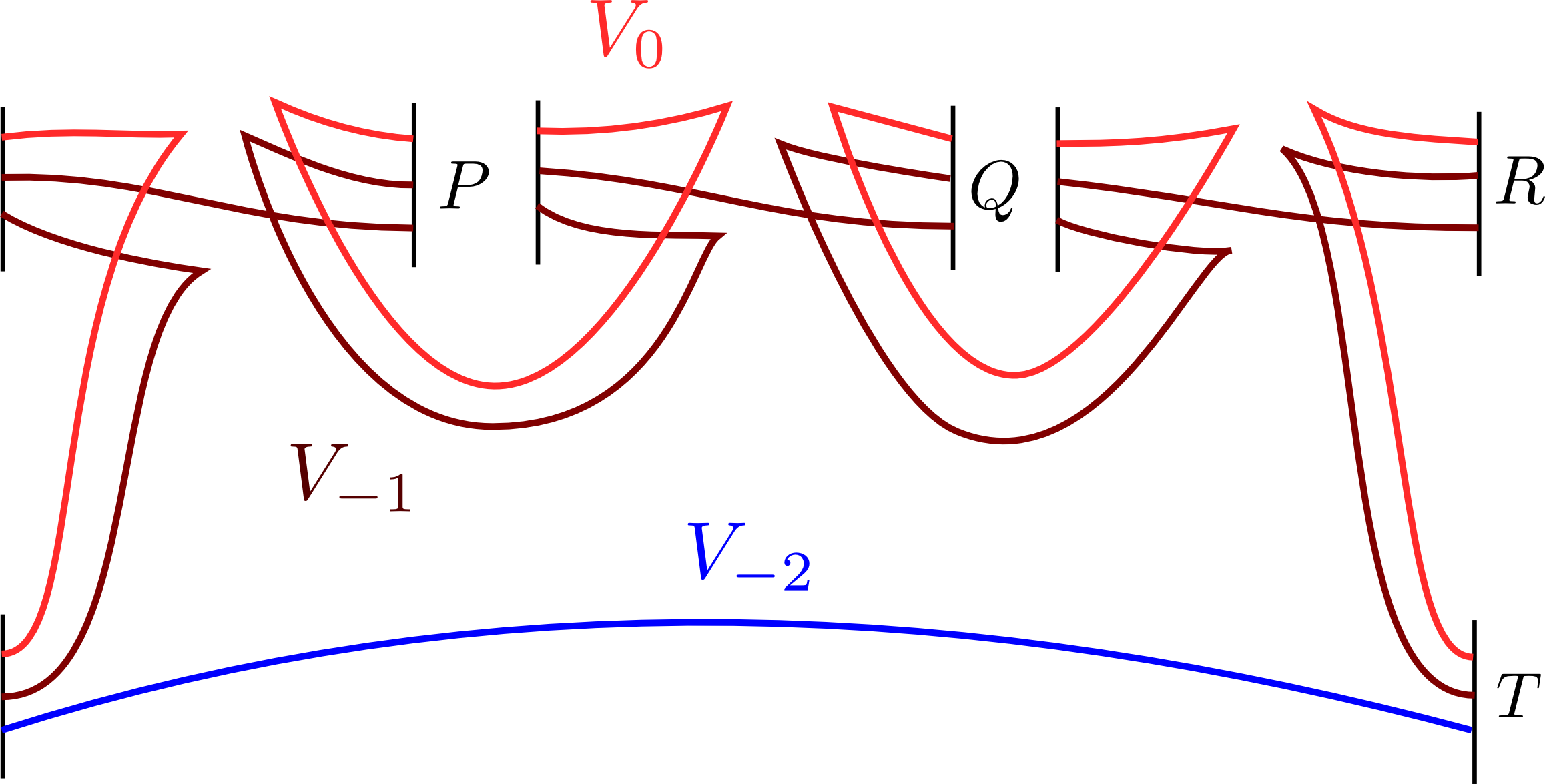}
\caption{The components corresponding to $\{V_{-2},V_{-1},V_0\}$ in the Legendrian front for the Weinstein 4--fold $T_{p,q,r}$ drawn by using Proposition \ref{prop:stacking}.}
\label{fig:TpqrFront}
\end{figure}

Figure \ref{fig:TpqrFront} uses the front projection from Proposition \ref{prop:LF} and the drawing conventions explained in this Subsection \ref{ssec:highD} except for the fact that we have depicted the neighborhood of each attaching $0$--sphere by two walls instead of two circles in order to simplify the presentation. It is direct to draw the remaining $p+q+r$ vanishing cycles in the set $V_\pi$ from the Legendrian front in Figure \ref{fig:TpqrFront}: above the existing Legendrian link, first insert $p$ parallel horizontal lines connecting the $0$--sphere corresponding to $P$, $q$ parallel horizontal lines connecting the $0$--sphere corresponding to $Q$ and then $r$ parallel horizontal lines connecting the $0$--sphere corresponding to $R$.

\begin{remark}
In the case of $p=q=r=0$, the underlying smooth Kirby diagram of the Legendrian front in Figure \ref{fig:TpqrFront} describes a $D^2$--bundle over the 2--torus $T^2$ with Euler class $e=0$. Indeed, from the defining equation it follows that $T_{0,0,0}$ is Weinstein equivalent to the cotangent bundle $(T^*T^2,\la_\st,\p_\st)$. The attentive reader is invited to simplify the front and compare it to the Legendrian diagram presented in \cite{Gompf}, which uses a unique $2$--handle in $\#^2(S^1\times S^2)$ instead of three $2$--handles in the contact boundary $\#^4(S^1\times S^2)$.
\end{remark}

This concludes the material on Legendrian fronts, which allows us to understand a Weinstein manifold $(W,\la,\p)$ presented as an abstract handlebody whose Legendrian attaching spheres are lifts of exact Lagrangian spheres in the subcritical skeleton $(W_0\times D^2,\la+\la_\st,\p+\p_\st)$ that can be expressed as words in Dehn twists. In order to complete our process and establish Recipe \ref{dictionary} we must also discuss how to obtain this data from the Weinstein manifolds $(W,\la,\p)$ we are interested in. This is the goal of the upcoming Section \ref{sec:lef}.

\section{Lefschetz Bifibrations}\label{sec:lef}
Given a Weinstein manifold $F=(F_T,\la_T,\p_T)$ which is a $T$--plumbing of Lagrangian spheres $\{L_1,\ldots,L_k\}$, for a given tree $T$ with $k=|T^{(0)}|$ vertices, Proposition \ref{prop:LF} provides a description of the front projection on the contact manifold $(\dd (F \x \C),\la_T)$. Combining this with Proposition \ref{prop:stacking}, we have an algorithm to draw the front of the Legendrian lift of any Lagrangian $L \sse \Sigma$ which can be expressed as a word in Dehn twists involving the Lagrangian spheres $\{L_1,\ldots,L_k\}$. In particular, such a procedure allows us to obtain explicit Legendrian handlebody decompositions of Weinstein manifolds, which are built from $F \x \C$ by attaching critical handles along such Legendrian spheres.

In a working environment, Weinstein manifolds $(W,\la,\p)$ often appear as complete intersections, or geometric modifications, of affine varieties: this is particularly the case in mirror symmetry and in the study of the symplectic topology of algebraic varieties \cite{JH,McLean,McLean2}. Due to their explicit nature and ubiquity, the examples we present shall be Weinstein manifolds given as affine hypersurfaces; the results also apply to affine complete intersections but the core geometric ideas are already present in the hypersurface case and we thus restrict to it.

Thanks to Section \ref{sec:kirby}, in order to draw a Legendrian handlebody we only need to extract from $(W,\la,\p)$ the data of attaching spheres in terms of the algebra of Dehn twists: the use of matching paths is well suited for this purpose, and the goal of this section is to introduce the necessary material on Weinstein Lefschetz bifibrations to make the dictionary translating from a Lefschetz fibration to a Legendrian front a practical machinery.

The two following subsections suffice for our purpose. Subsections \ref{ssec:lef} and \ref{ssec:match} focus on the abstract understanding of Weinstein Lefschetz fibrations, Subsection \ref{ssec:recipe} explains the main Recipe, and Subsection \ref{ssec:AD_bifiber} provides two explicit instances of them.

The general technique of working with Lefschetz bifibrations has been used in the literature for a number of years, see \cite{Au00, AS, May, MS, Se08, SeMem, Se15}. The reader familiar with \cite[Part III]{Se08} can move directly to Subsection \ref{ssec:recipe}.

\subsection{Lefschetz Fibrations}\label{ssec:lef}
Let us consider a generic Weinstein Lefschetz fibration
$$\pi:(W^{2n},\la,\p)\lr\C$$
with regular fiber $F_\pi=\pi^{-1}(0)$ and critical values $S_\pi=\{v_1,\ldots,v_s\}\in\D^2(1)$. We will always assume that the critical points have distinct critical values and the main result in the article of Giroux--Pardon \cite[Theorem 1.9]{GP} establishes that any Weinstein manifold $(W,\la,\p)$ admits such a Lefschetz fibration $\pi$. Nevertheless, a strong point in favour of the study of affine varieties is that explicit Lefschetz fibrations can be constructed by using generic hyperplane sections \cite{GH}. In short, any generic linear function on an affine variety has complex non--degenerate Morse singularities, and will therefore be a Lefschetz fibration up to isotopy \cite{McLeanThe}.

Given $(W,\la,\p)$ and the Lefschetz fibration $\pi$, the set of vanishing cycles $V_\pi=\{V_1,\ldots,V_s\}$ is a cyclically ordered set of exact Lagrangian spheres in the fiber $\Sigma_\pi$, obtained as follows. A vanishing cycle $V_\gamma$ is the boundary of an embedded Lagrangian disk $\Delta_\gamma\cong D^n\sse (W,\la,\p)$ whose image under the projection $\pi$ is an embedded path
$$\gamma_j:([0,1],\partial[0,1])\lr (D^2,\{0,v_j\}).$$
This embedded Lagrangian disk $\Delta_\gamma$ is unique, up to a contractible choice of Hamiltonian isotopy, if the plane path $\gamma$ is fixed, and it is called the vanishing thimble of $\gamma$. The disk $\Delta_\gamma$ is defined as being the set of all points in $W$ which are sent to the critical point corresponding to the critical value $v$ by symplectic parallel transport along $\gamma$. The vanishing cycles $V_\pi$ for the Lefschetz fibration $\pi$ are defined as $V_j = \Delta_{\gamma_j} \cap F_\pi$. We can assume that the critical points have distinct arguments, in which case we can canonically choose vanishing cycles associated to the linear paths $$\gamma_j(t)=tv_j,\quad t\in[0,1]\quad j=1,\ldots,s.$$
This set is cyclically ordered by the argument of the critical values in the counterclockwise direction. The Weinstein submanifold $(F_\pi,\la|_{F_\pi},\p|_{F_\pi})$ and the set of vanishing cycles $V_\pi$ determine the Weinstein manifold $(W,\la,\p)$ up to Weinstein deformation equivalence; indeed, attaching a critical Weinstein handle to the Weinstein domain $F_\pi\times D^2$ in the contact boundary $F_\pi\times S^1\sse\partial(F_\pi\times D^2)$ along each of the Legendrian lifts of the exact Lagrangians in the set $V_\pi$ results in a Weinstein manifold deformation equivalent to $(W,\la,\p)$ \cite{CE,GP}.

\begin{remark}
Technically, for either Lefschetz fibrations or Weinstein handle attachments, each vanishing cycles is endowed with a smooth parametrization induced by $V_j = \dd \Delta_{\gamma_j}$. Implicitly we equip all vanishing cycles and Weinstein attaching maps with these parametrizations throughout the paper.\hfill$\Box$
\end{remark}

Let us denote this relation between the total space $(W,\la,\p)$ and the data of the Weinstein fiber $(F_\pi,\la|_{F_\pi})$ plus the vanishing cycles $V_\pi$ of a Weinstein Lefschetz fibration $\pi$ by the equality
$$(W,\la,\p)=\lf(F_\pi;V_\pi).$$
The construction of $(W,\la,\p)$ from $(F_\pi;V_\pi)$ only requires $F_\pi$ to be a Weinstein manifold $(F,\la_F,\p_F)$ and $V_\pi$ be a cyclically ordered set $V$ of exact Lagrangian spheres in $F$, thence we can also use the notation $$(W,\la,\p)=\lf(F;V)$$ without specifying the Weinstein Lefschetz fibration $\pi$, whenever $V$ is a finite collection of Lagrangian spheres in $(F, \lambda_F)$, equipped with a cyclic ordering.\\

Then two Weinstein manifolds $\lf(F;V)$ and $\lf(F';V')$ are Weinstein deformation equivalent if $F'$ is Weinstein deformation equivalent to $F$ and the vanishing cycles of $V'$ are Hamiltonian isotopic to those in $V$ in an ordered manner, up to a cyclic shift \cite{CE}. In addition, there are two local modifications of the Weinstein structure of $(W,\la,\p)$ which establish the following two equalities:
\begin{itemize}
 \item[1.] (Hurwitz Moves) corresponding to a different choice of path $\gamma_i$ for the same Lefschetz fibration, we have
 $$\lf(F;\{L_1,\ldots,L_i,L_{i+1},\ldots,L_s\})=\lf(F;\{L_1,\ldots,L_{i+1},\tau_{L_{i+1}}(L_i),\ldots,L_s\}),$$
 $$\lf(F;\{L_1,\ldots,L_i,L_{i+1},\ldots,L_s\})=\lf(F;\{L_1,\ldots,\tau^{-1}_{L_i}(L_{i+1}),L_i,\ldots,L_s\}),$$
 for any index $i$ modulo $s$. Since we have the symplectic isotopy
 $$\tau_{f(L)}\simeq f^{-1}\tau_Lf$$
 for any compactly supported symplectomorphism $f\in\Symp^c(F)$, the three sets of vanishing cycles have the same global monodromy $\tau_{L_1}\circ\cdots\circ \tau_{L_i}\circ\tau_{L_{i+1}}\circ\cdots\tau_{L_s}$. In this Weinstein handlebody interpretation, these moves correspond to handle slides, as in Proposition \ref{prop:slide}.\\
 
 \item[2.] (Stabilization) Consider a Lagrangian disk $(D,\partial D)\sse(\Sigma,\partial F)$ with Legendrian boundary and attach a critical handle along $\partial D$; denote the resulting manifold $F(D)$. It contains the exact Lagrangian sphere $D\cup_\partial C$ consisting of the Lagrangian disk $D\sse F$ and the Lagrangian core $C$ of the critical handle glued along their common boundary. Then we have
$$\lf(F;V)=\lf(F(D);V\cup\{D\cup_\partial C\}).$$
As Weinstein handlebodies, this corresponds to introducing a canceling pair of handles, of consecutive indices $n-1$ and $n$, as in Proposition \ref{prop:handle cancel}. In this article we have chosen Lefschetz fibrations which do not destabilize, but in general both Lefschetz fibrations appearing in affine algebraic geometry and those obtained with asymptotically holomorphic techniques are often stabilized.
\end{itemize}

Notice that the two moves above have different character. A Hurwitz move states that for a fixed Lefschetz fibration $\pi:W\lr\C$ we can have different sets of vanishing cycles corresponding to different choices of vanishing paths $\gamma$. By contrast, a stabilization is choosing a different map $\pi:W\lr\C$ for a fixed total space $W$.

There is a third equivalence of another sort: different vanishing paths, with endpoints in different critical values, can define Lagrangian thimbles $\Delta$ which are Hamiltonian isotopic, and in particular they will have the same vanishing cycles. One such relation is based on the binary braid relation $\tau_SL =\tau_L^{-1}S$, where $S$ and $L$ are Lagrangian spheres intersecting in a single point, see \cite[Appendix A]{Se99} and \cite[Section 2]{Ke}, especially \cite[Figure 9]{Ke}. Now, consider two critical values $v$ and $v'$ with associated vanishing cycles $V$ and $V'$ such that their intersection $V\cap V'$ is transverse and consists of a unique point; these two critical points $v$ and $v'$ are depicted in Figure \ref{fig:Vmove} with a triangle and a four--pointed star respectively. Then the vanishing thimbles $\Delta_\gamma$ and $\Delta_{\gamma'}$ associated to the embedded paths $\gamma$ and $\gamma'$ from Figure \ref{fig:Vmove} are two Lagrangian disks which are Lagrangian isotopic in the total Weinstein manifold relative to their common boundary $\partial\Delta_\gamma=\partial\Delta_{\gamma'}\sse F$, which lies in the fiber over the white point in Figure \ref{fig:Vmove}.

\begin{figure}[h!]
\centering
  \includegraphics[scale=0.35]{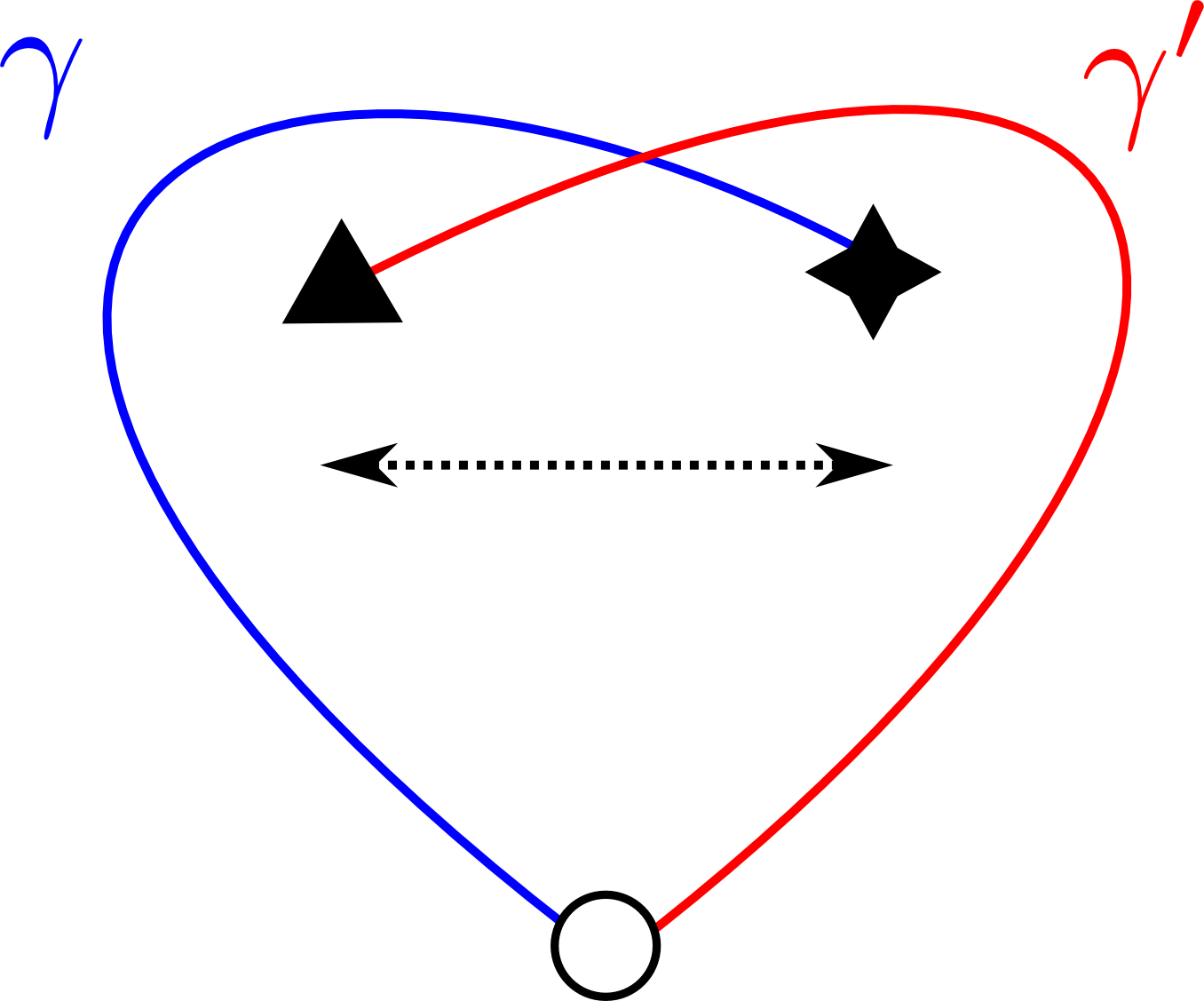}
  \caption{Exchange $\Delta_\gamma\longleftrightarrow\Delta_{\gamma'}$.}
  \label{fig:Vmove}
\end{figure}

Due to the heart--shaped nature of Figure \ref{fig:Vmove}, we call such an exchange a V--move. One may ask why such a move would ever be useful: given a Lefschetz fibration we are often only interested in the vanishing cycles and the way they depend on vanishing paths for a given critical point; if they happen to be isotopic for two different critical points then we still need to record them with multiplicity. This move will become relevant when we use the Lagrangian thimbles of a Lefschetz bifibration to describe vanishing cycles using matching paths, and in particular for the proof of Theorem \ref{thm:kr}.

This concludes the first part of the discussion on Weinstein Lefschetz fibrations, and we now proceed with the study of their vanishing cycles in terms of matching paths. In particular, we will describe Recipe \ref{dictionary}.

\subsection{Matching Paths}\label{ssec:match}
Consider a Weinstein Lefschetz fibration $\rho:(F,\la,\p)\lr\C$, and let $F_\rho := \rho^{-1}(0) \sse F$ be the generic fiber. Suppose we have two critical values $v_1,v_2\in\C$ with vanishing thimbles $\Delta_{\gamma_1},\Delta_{\gamma_2}\sse F$, which give the same vanishing cycle $\partial\Delta_{\gamma_1}=\partial\Delta_{\gamma_2}\sse F_\rho$. Then we can concatenate the two Lagrangian vanishing thimbles into an exact Lagrangian sphere lying over the concatenated embedded path $\gamma=\gamma_1\cup\overline{\gamma_2}$; we denote this sphere by
$$S_{\gamma}=\Delta_{\gamma_1}\cup_\partial\Delta_{\gamma_2}\sse (F,\la,\p).$$
The path $\gamma\sse\C$ between critical points is said to be a matching path for the exact Lagrangian sphere $S_\gamma$. The Lagrangian sphere $S_\gamma$ is determined by $\gamma$, and the only constraint in order to obtain a Lagrangian sphere with this method is that the vanishing cycles near the endpoints of $\gamma$ coincide up to Hamiltonian isotopy under symplectic parallel transport along $\gamma$. 

Given an embedded path $\gamma\sse\C$ in the plane, the \emph{half twist} is a symplectomorphism $\tau_\gamma$ of the plane which is compactly supported in an arbitrarily small neighborhood $\Op(\gamma)\sse\C$. It is determined by the action on a transverse curve $\alpha$ which is described in Figure \ref{fig:halftwist}. Up to isotopy among symplectomorphisms of the plane fixing the endpoints of $\gamma$, the symplectomorphism $\tau_\gamma$ is determined by the isotopy class of the path $\gamma$ \cite{Do}.

\begin{figure}[h!]
\centering
  \includegraphics[scale=0.45]{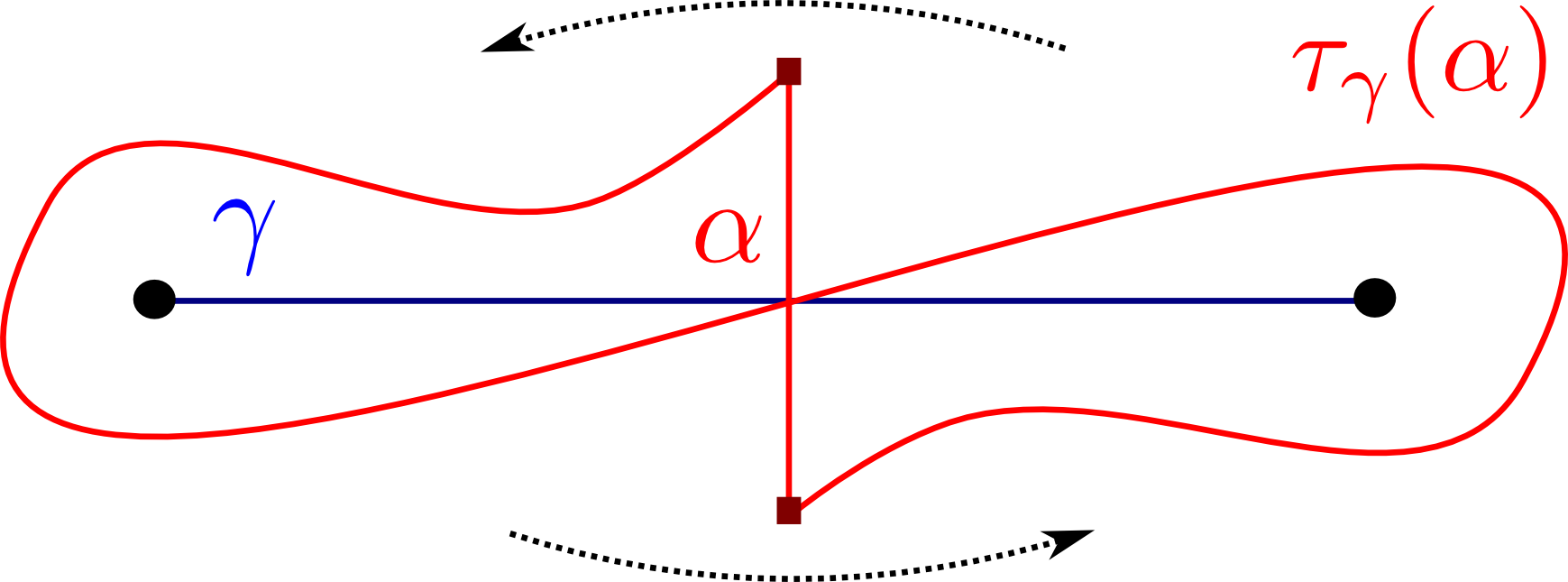}
  \caption{Half twist on the plane.}
  \label{fig:halftwist}
\end{figure}

In this moment, given an embedded path $\gamma$ between two critical values we have a diffeomorphism $\tau_\gamma$ of the complex plane $\C$, and an exact Lagrangian sphere $S_\gamma\sse (F,\la,\p)$ which gives rise to a compactly supported symplectomorphism $\tau_{S_\gamma}$ of the Weinstein fiber $F$. The essential ingredient that allows us to strictly operate with curves in the plane is the relation between the half--twist in the plane and the Dehn twist inside $F$, which we now state the following lemma; details of the proof can be found in \cite[Section 14.2]{Do} for the surface case, and \cite[Lemma 7.1]{MS} for the general case:

\begin{lemma}
Let $\gamma,\alpha$ be two matching pats for a Lefschetz fibration $\rho:(F,\la)\lr\C$.\\
Then the Lagrangian sphere $S_{\tau_{\gamma}(\alpha)}$ is Lagrangian isotopic to the Lagrangian sphere $\tau_{S_\gamma}(S_\alpha)$.
\end{lemma}

\subsection{The recipe}\label{ssec:recipe}

In Section \ref{sec:kirby}, Proposition \ref{prop:stacking} provides an algorithm for drawing the front projection of Legendrian lifts of Lagrangians, which are given by words in Dehn twists of Lagrangians we understand. Therefore, we would like to express given Lagrangians in this way. Taking the above lemma into account, if we have a Lagrangian sphere which is given as a matching path, expressing it in this way becomes the combinatorial problem of expressing the matching path in terms of half--twists on some basis of curves. This tends to be a feasible task with the appropriate use of Hurwitz moves, V-moves and destabilizations. In conjuction, the recipe translating from a given Weinstein manifold to an explicit Legendrian handlebody can be summarized as follows.

\begin{recipe}\label{dictionary}
Suppose we are given a Weinstein manifold expressed as a Lefschetz fibration $(W,\la,\p)=\lf(F;V)$ with the symplectic topology of $(F,\la,\p)$ well understood, such as the case in which the Weinstein fiber $(F,\la,\p)=(F_T,\la_T,\p_T)$ is a plumbing of spheres along a tree $T$. Expressing $(W,\la,\p)$ explicitly as a Weinstein handlebody amounts to drawing the Legendrian front projection of the Legendrian lifts of the vanishing cycles in the set $V$, inside the contact manifold $(\dd(F \x D^2),\la+\la_\st)$. For this, we proceed in the following manner:

\begin{itemize}
\item[1.] Choose a set of exact Lagrangian spheres $\L=\{L_1,\ldots,L_r\}$ in the fiber $(F,\la,\p)$.\\

\noindent This must be a set that we understand and it often, but not strictly necessarily, forms a Lagrangian skeleton of $F$. The essential property of these Lagrangian spheres is that we know how to draw their Legendrian lifts in the front projection of the contact boundary $(\dd(F \x D^2),\la+\la_\st)$; Proposition \ref{prop:LF} describes how to do this for Lagrangian skeletons of plumbings of spheres.\\

\item[2.] Choose a Weinstein Lefschetz fibration $\rho:(F,\la)\lr\C$ and express the Lagrangian spheres in the set $\L$ as matching paths $\Gamma=\{\gamma_1,\ldots,\gamma_r\}$ of $\rho$.\\

\item[3.] Given a vanishing cycle $V_i\sse (F,\la)$, draw the embedded path $\rho(V_i) = \vartheta_i \sse\C$.\\

\item[4.] Express each matching path $\vartheta_i$ as a word in half--twists along the arcs in $\Gamma$.\\

\item[5.] Once each vanishing cycle $V_i\sse (F,\la)$ is expressed as a word in Dehn twists with the Lagrangian spheres in the set $\L$, we apply Proposition \ref{prop:stacking} to draw the front projection of their Legendrian lifts $\Lambda_i \sse (\dd(F \x D^2),\la+\la_\st)$.\\

\item[6.] Then we consider the Legendrian link $\bigcup_i\Lambda_i$ determined by the cyclic ordering of the indices $i$: we push the Legendrian component $\Lambda_i$ in the Reeb direction by height equal to its index $i$, and this gives a well-defined link.\\

\item[7.] Simplify the Legendrian front projection of the link in Step 6 by applying Reidemeister moves and Legendrian handleslides from Subsections \ref{ssec:highD} and \ref{ssec:Wmoves}.\hfill$\Box$
\end{itemize}

\end{recipe}

\begin{remark}
The first four steps of Recipe \ref{dictionary} belong to the theory of Lefschetz bifibrations and have been used for a number of results, see for example \cite{Au00, AS, May, MS, Se08, SeMem, Se15}, whereas steps five to seven are original content of this paper. It might be interesting to point out that, while stopping at the fourth step gives a combinatorial description of the symplectic topology of the Weinstein manifold, it seems unlikely that many of the results in this paper such as Theorems \ref{thm:Xab}(a) and \ref{thm:kr} could be proven by considering only Lefschetz fibrations and not Legendrian fronts.\hfill$\Box$
\end{remark}

Recipe \ref{dictionary} becomes particularly productive if Steps 1 and 2 are systematized for each given Weinstein fiber $(F,\la,\p)$. That is, given a Weinstein fiber $(F,\la,\p)$ we can fix the auxiliary Lefschetz bifibration $\rho:(F,\la,\p)\lr\C$ and a set $\L$ and then become skilled with the combinatorics of its set $\Gamma$ of matching paths. Then in order to draw a Legendrian handlebody for the Weinstein manifold $(W,\la,\p)=\lf(F;V)$ we focus directly on the later steps of the recipe. The following subsection introduces two families of Weinstein manifolds, the Milnor fibers of the $A_k$ and $D_k$ singularities, and carries out Steps 1 and 2 for these manifolds for later use.

Let us explain in more detail how to perform Step 3 in the case where our initial Weinstein manifold $(W,\la,\p)$ is an affine variety
$$W = \{f_1(z) = \ldots = \ldots f_k(z) = 0 \} \sse \C^{n+k}.$$
First, we choose a generic complex linear function $\pi: \C^{n+k} \lr \C$, consider its regular fiber $(F,\la) = (F_\pi,\la) = (\pi|_W^{-1}(0),\la_\st)$ and compute the critical values $\{v_i\} \sse \C$ of the restriction $\pi|_W:W\lr\C$. Let us also denote $F_t = \pi|_W^{-1}(t)$ for $t \in \C$.

Deducing the symplectic topology of fiber $(F,\la)$ is generally difficult, but in practice it can often be chosen to be something well understood. In this paper $(F,\la)$ will always be one of the manifolds from Section \ref{ssec:AD_bifiber}, and in general its complexity is directly tied to the complexity of the defining polynomials $\{f_j\}$, but with the advantage that we have the linear constraint $\pi(z_1,\ldots,z_{n+1}) = 0$, which we can choose to make $(F,\la)$ as simple as possible.

Then we perform Step 2, in which we choose a function $\rho: \C^{n+k}\lr\C$ compatible with our understanding of the fiber $(F,\la)$ i.e.~ such that we know the vanishing cycles for each critical value of the restriction $\rho|_F:F\lr\C$. Then the restrictions $\rho|_{F_t}:F_t \lr \C$ are Lefschetz fibrations for generic values of $t\in\C$, and their critical values $\{u_i\}\sse\C$ vary continuously for those values of $t\in\C$.

Now, let us take a linear path from the origin $t = 0$ to the critical value $t = v_i$, and see how the critical values of the Lefschetz fibration $\rho|_{F_t}$ change along this path: as $t$ approaches the critical value $v_i$, two of the critical values $\{u_i\}$ of the auxiliary fibration $\rho|_{F_t}$ will collide, and let us assume that no other collisions will occur. Thus there will be two critical values, say $u_{i_1},u_{i_2}$ of the fibration $\rho|_F$ which move in the plane $\C=\im(\rho|_F)$ as $t$ changes and eventually collide at $t = v_i$; all the other critical values in $\{u_i\}$ also vary with the parameter $t$ but they remain distinct by the genericity of the initial Lefschetz fibration $\pi$. The trace of this collision therefore gives a path $\gamma_i \sse \C$ which connects $u_1$ and $u_2$ and is disjoint from all other critical values: this path $\gamma_i$ is exactly the matching path for the vanishing cycle of the critical point $v_i$ since the value where the critical values $u_1$ and $u_2$ collide is the value of the fibration $\rho$ at the critical point of $\pi|_W$ corresponding to the critical value $v_i$.

Since the calculation of the critical values of $\pi|_W$ and $\rho|_{F_t}$ relies on finding roots of polynomials, generically this does require the use of a computer for numerical approximation. Fortunately, since we are only interested in the path $\gamma_i$ as a plane curve up to isotopy, and this information is robust up to $C^0$--error, numerical approximation completeley suffices for our purposes. For the applications in Section \ref{sec:app} we will start with polynomials and write down sets of vanishing paths with little comment, but in all cases we have verified the results by using computer algorithms to do the computations. This concludes the digression on Step 3 of our Recipe \ref{dictionary}.

The remaining Steps 4, 5, 6 and 7 are combinatorial in nature. Steps 5 and 6, while visually complicated, are completely mechanical with practice; Steps 4 and 7 are the ones that require actual effort and vary substantially according to each case. There is a recorded talk \cite{CMIAS} by the authors at the Institute of Advanced Study where Recipe \ref{dictionary} is strictly followed: the reader might benefit from listening to the first seventeen minutes where the recipe is applied to the Weinstein six--fold $X_{2,3}=\{x^2y^3+z^2+w^2=1\}\sse\C^4$. This manifold is also discussed in detail in Section \ref{sec:app} below.

Let us now introduce the two models that we will be using for Steps 1 and 2; as explained above, the computations in these steps only need to be performed once for each Weinstein fiber $(F,\la)$, and thus the upcoming results in Subsection \ref{ssec:AD_bifiber} allows us in practice to start Recipe \ref{dictionary} at Step 3.

\subsection{$A_k$ and $D_k$--bifibers}\label{ssec:AD_bifiber}

We now present explicit Lefschetz fibrations for two families of $2n$--dimensional Weinstein manifolds $(F_\pi,\la_\st|_{F_\pi})$: these two families are defined by the $A_k$ and $D_k$--plumbing diagrams as depicted in Figure \ref{fig:ADdiagram}. These plumbing intersection diagrams are precisely the trees $T$ used in Section \ref{sec:kirby}.\

In context, $A_k$ and $D_k$ fibers appear prominently in the applications of presented in this article, in particular Theorem \ref{thm:Xab}, Theorem \ref{thm:subflex} and Section \ref{ssec:mirror2} are of $A_k$--type, and $D_k$--fibers are used in Theorem \ref{thm:kr} and Section \ref{ssec:mirror1}. In both cases the Lefschetz fibrations are non--generic linear sections and the set $\L$ forms a Lagrangian skeleton to which the total Weinstein manifolds retract. In consequence, an a priori knowledge of a fibration for these Weinstein fibers increases the efficiency of Recipe \ref{dictionary} being applied.

\begin{figure}[h!]
\centering
  \includegraphics[scale=0.65]{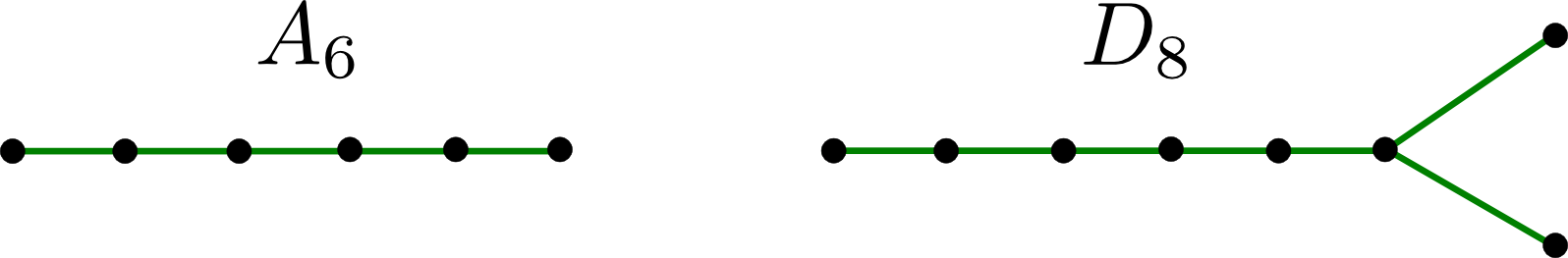}
  \caption{Two instances of plumbing trees.}
  \label{fig:ADdiagram}
\end{figure}

Let us discuss the $A_k$--case first. The Weinstein manifold
$$A^{2n-2}_{k}=\{(x,\underline{z}):x^{k+1}+\sum_{i=1}^{n-1}{z^2_i}=1\}\sse\C^{n}[x,z_1,\ldots,z_{n-1}]$$
is the affine Milnor fiber of the cyclic singularity, and it admits the non--generic Weinstein Lefschetz fibration
$$\rho:A_k\lr\C,\qquad \rho(x,z_1,\ldots,z_{n-1})=x,$$
whose fiber is the standard affine conic $(F_\rho,\la)=(A^{2n-4}_1,\la_\st|_{F_\rho})\cong (T^*S^{n-2},\la_\st)$ and $(k+1)$ critical points $\{p_1,\ldots,p_{k+1}\}$; the set of vanishing cycles consists of $(k+1)$ copies of the sphere zero section $Z\cong S^{n-2}\sse (T^*S^{n-2},\la_\st)\cong (F_\rho,\la)$ of the cotangent bundle. The critical values will be drawn in the real line, as in Figure \ref{fig:A4_diagram}, although strictly they are distributed in a circle as $(k+1)$th roots of unity.

\begin{figure}[h!]
\centering
  \includegraphics[scale=0.35]{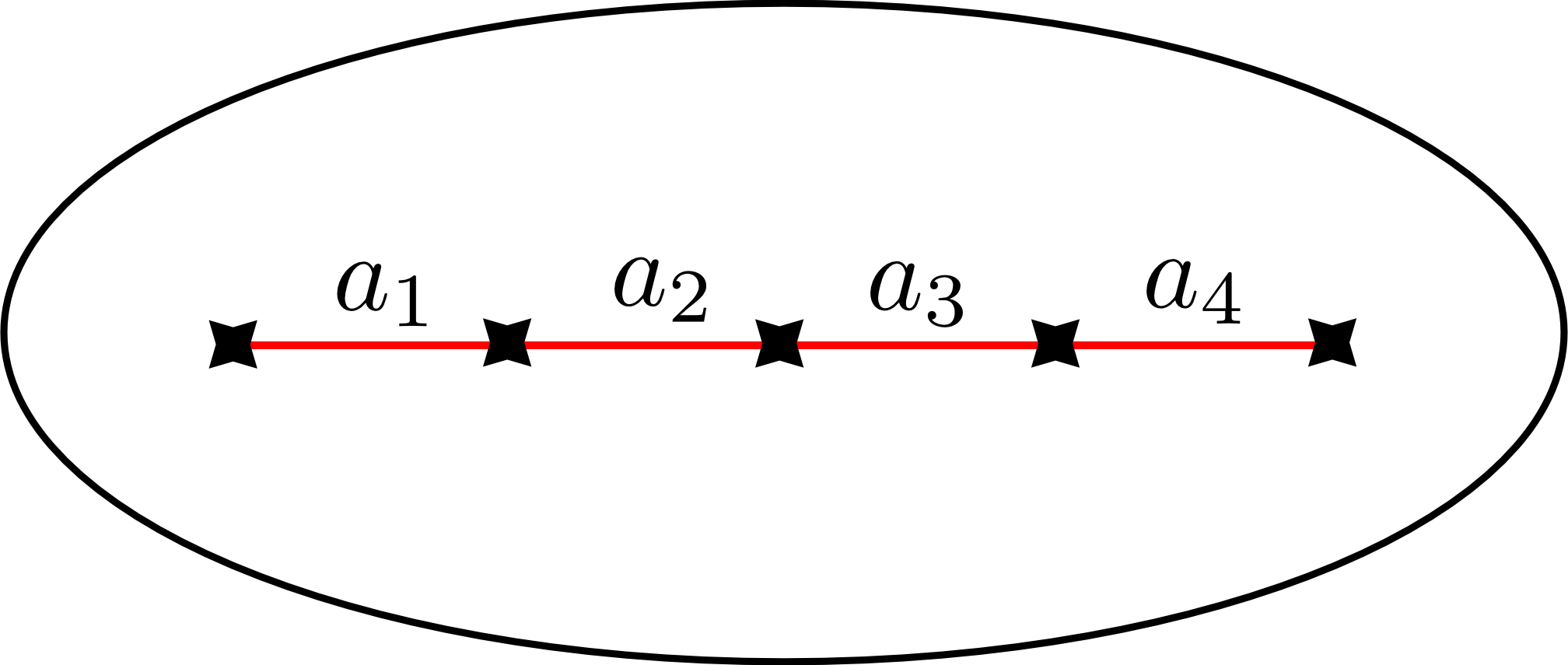}
  \caption{$A^{2n-2}_4=\lf(T^*Z,\{Z,Z,Z,Z\})$ and basis $(a_1,a_2,a_3,a_4)$.}
  \label{fig:A4_diagram}
\end{figure}

We can describe the $k$ spheres $\L=\{a_1,\ldots,a_k\}$ conforming the Lagrangian $A_k$--skeleton as $k$ homonymous matching paths $\Gamma=\{a_1,\ldots,a_k\}$ of the fibration $\rho$; the simplest choice is to consider the straight matching paths $a_i$ connecting the critical values $\rho(p_i)$ and $\rho(p_{i+1})$, for $0\leq i\leq k+1$. The intersection pattern is that of an $A_k$--diagram and the set $\L$ will be taken as a cyclically ordered Lagrangian basis for the half--dimensional homology group $H_{n-1}(A^{2n-2}_k,\Z)\cong\Z^k$.

\begin{remark}
These $A_k$--diagrams suffice in the study of Weinstein $4$--manifolds since any surface can be stabilized to an $A_k$--Milnor fiber. Nevertheless, even in the case of $4$--manifolds it is beneficial at times to use other fibers.\hfill$\Box$
\end{remark}

Let us now study the $D_k$--case. The affine Milnor fiber of the dihedral singularity
$$D^{2n-2}_k=\{(x,y,\underline{z}):x^{k-1}+xy^2+\sum_{i=1}^{n-2}{z^2_i}=1\}\sse\C^{n}[x,y,z_1,\ldots,z_{n-2}]$$
can be endowed with the Weinstein Lefschetz fibration
$$\rho:D_k\lr\C,\qquad\rho(x,y,z_1,\ldots,z_{n-2})=x+y$$
with fiber equal to $(F_\rho,\la) = (A^{2n-4}_{k-2},\la)$ and $2(k-1)$ critical points $\{p_1,\ldots,p_{2(k-1)}\}$. In this case, the vanishing cycles of the associated distinct critical values, which we can draw as $2(k-1)$th roots of unity, are
$$V_i=a_i,\quad V_{i+k-1}=a_i, \quad 1\leq i\leq k-2,$$
$$V_{k-1}=\tau_{a_2}\circ\cdots\circ\tau_{a_{k-2}}(a_1),\quad V_{2(k-1)}=\tau_{a_2}\circ\cdots\circ\tau_{a_{k-2}}(a_1),$$
where there sequence of spheres $(a_1,\ldots,a_{k-2})$ constitute the linear $A_{k-2}$--skeleton of the fiber $(F_\rho,\la)$. In order to run Step 2 in Recipe \ref{dictionary} it suffices to describe a Lagrangian $D_k$--basis $\L$ in terms of matching paths $\Gamma$; this can be done as follows. First, consider the three matching paths $(\alpha,\gamma,\delta)$ as in Figure \ref{fig:D6_diagram}:

\begin{figure}[h!]
  \centering
  \includegraphics[scale=0.35]{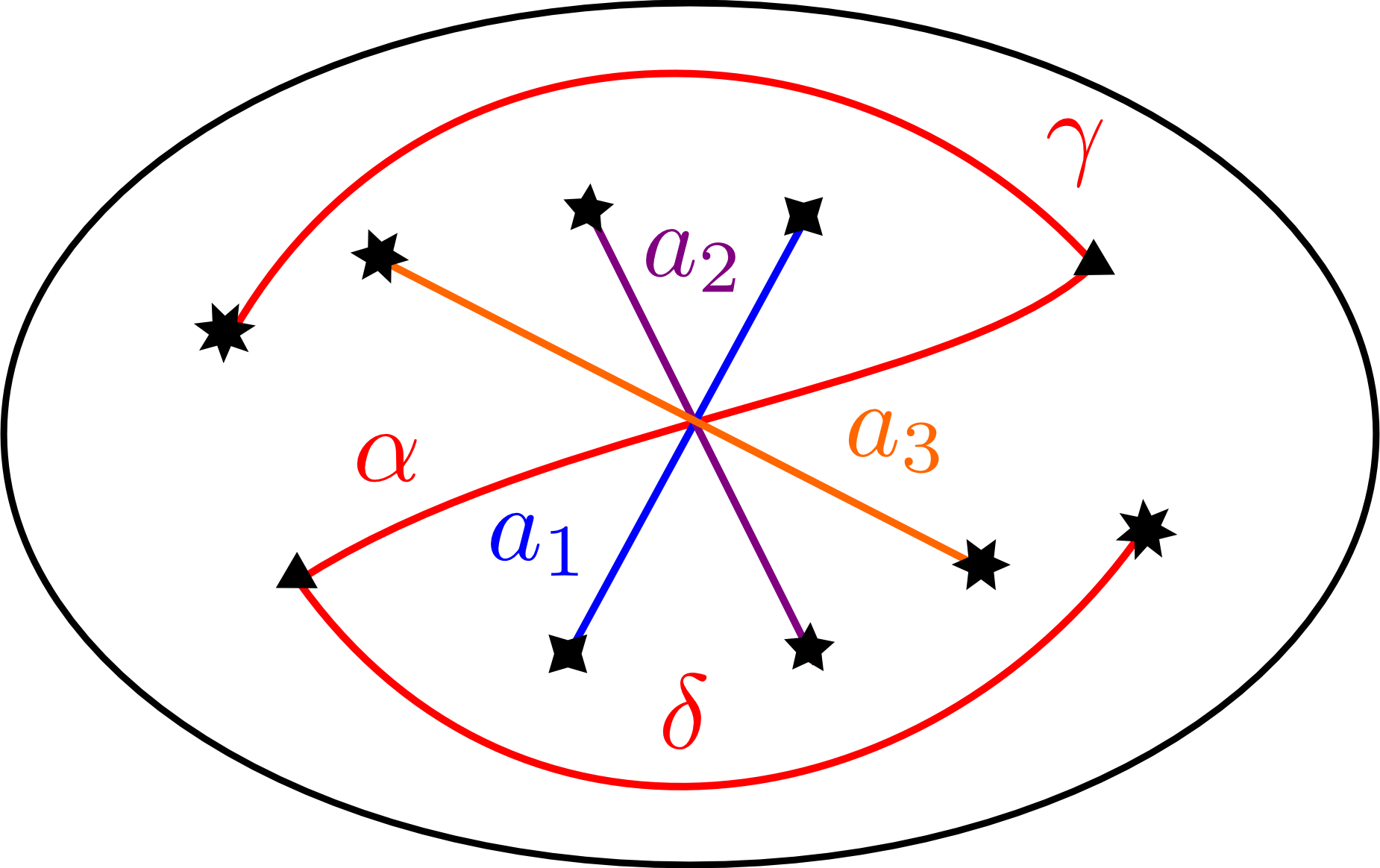}
  \caption{$D^{2n-2}_6=\lf(A_4,V)$ and basis $(\alpha,\gamma,\delta;a_1,a_2,a_3)$.}
\label{fig:D6_diagram}
\end{figure}

The description of the matching paths $(\alpha,\gamma,\delta)$ is the following:

\begin{itemize}
 \item[-] The matching path $\alpha$ is the linear path connecting the critical point $p_1$ with the critical point $p_k$. The vanishing cycles associated to $p_1$ and $p_k$ are Lagrangian isotopic to the zero section $a_1$, and thus $\alpha$ is a matching path.\\
 
 \item[-] The matching path $\delta$ is a path connecting the critical point $p_1$ with the critical point $p_{k-1}$ whose interior is contained in the unbounded component $\{z\in\C:|z|>1\}$. The vanishing cycles associated to $p_1$ and $p_{k-1}$ are the zero section $V_1=a_1$ and the Lagrangian sphere $V_{k-1}=\tau_{a_2}\circ\cdots\circ\tau_{a_{k-2}}(a_1)$ respectively. Then $\delta$ is a matching path because the symplectic parallel transport of the vanishing cycle $a_1$ along the path $\delta$ accumulates the monodromy of the critical values from $a_2$ to $a_{k-2}$ thus yielding the vanishing cycle $\tau_{a_2}\circ\cdots\circ\tau_{a_{k-2}}(a_1)$ when considered near the critical point $p_{k-1}$, which then coincides with $V_{k-1}$.\\
 
 \item[-] The matching path $\gamma$ connects the critical point $p_k$ with the critical point $p_{2(k-1)}$ with its interior being contained in the unbounded component $\{z\in\C:|z|>1\}$. The path $\gamma$ is a matching path for the Lefschetz fibration $\rho:D_k\lr\C$ for the same reasons as in the previous item showing that the path $\delta$ is a matching path.
\end{itemize}

Second, consider the additional $(k-3)$ matching paths $(a_2,\ldots,a_{k-2})$, where $a_i$ is the linear path connecting the critical points $p_i$ and $p_{i+k-1}$. These are matching paths for the fibration $\rho:D_k\lr\C$ since the corresponding vanishing cycles $V_i$ and $V_{i+k-1}$ coincide: they are both Lagrangian isotopic to the Lagrangian sphere $a_i\sse (F_\rho,\la)$, which also explains the choice of notation for these matching path. Notice that in this notation we could also add the matching path $a_1=\alpha$.

The set of Lagrangian spheres obtained from both sets of matching paths together conform to a $D_k$--configuration. Indeed, the Lagrangian spheres over the matching paths $(\alpha,a_2,\ldots,a_{k-2})$ form a linear chain of length $(k-2)$ and the Lagrangian spheres associated to the matchings paths $\delta$ and $\gamma$ are mutually disjoint and only intersect the linear chain at the Lagrangian sphere $\alpha$ once each; this exhibits the $D_k$--plumbing diagram we needed. This configuration is depicted in Figure \ref{fig:D6_diagram} in the case $k=6$.

\begin{remark}
The $D_4$--diagram has already appeared in Subsection \ref{ssec:Tpqr} and it is essential in the study of cubic polynomials since a generic elliptic curve in the projective plane intersects an affine chart at a thrice punctured torus. It also features in the ongoing study of the symplectic topology of Dieck--Petri surfaces \cite{Zai} and the Milnor fibers of non--isolated singularities \cite{Si,Si2}.\hfill$\Box$
\end{remark}

This concludes our discussion on Lefschetz bifibration diagrams: at this point we have established Recipe \ref{dictionary}, that we can systematically apply, and performed the computations that we can directly use in case the fibers in the Lefschetz fibrations are $A_k$ or $D_k$--plumbings. However, before proving the theorems stated in Section \ref{sec:intro}, we will prove a result that we find both interesting on its own and illustrative of the techniques presented thus far. This is the content of the following subsection.

\subsection{The Eye: a subflexible Weinstein 6--fold}\label{ssec:subflex}
In this subsection we present an example of a non--flexible subflexible Weinstein manifold using an $A_2$--bifibration diagram, this manifold was first considered by M.~Maydanskiy in \cite{May}, and later by R.~Harris \cite{Ha}, and revisited in the recent articles \cite{MuSi, Siegel}: we treat it here from the Legendrian perspective, which clarifies further the interesting symplectic topology of this Weinstein manifold.

Consider the Weinstein 6--fold $(E,\la,\p)$ described as the affine hypersurface
$$E=\{(x,y,z,w):x(xy-1)=z^2+w^2\}\sse\C^4.$$
This is a variation on the symplectic manifold considered in \cite[Example 1.6]{Se15}. We can use the Lefschetz fibration
$$\pi:E\lr\C,\quad\pi(x,y,z,w)=x+y,$$
which has two critical points and regular fibre $F_\pi\cong\{x(x^2-1)+z^2+w^2=\ln2\}\sse\C^3[x,z,w]$ symplectomorphic to the $A_2$--plumbing of two cotangent bundles $(T^*S^2,\la_\st,\p_\st)$, for it is the 4--dimensional Milnor fiber of the $\Z_3$--cyclic singularity. The fibration $\pi$ has two critical points, and we denote the corresponding vanishing cycles by $V_1$ and $V_2$: thus, in the notation above, $(E,\la,\p)=\lf(A_2,\{V_1,V_2\})$.
\begin{figure}[h!]
\includegraphics[scale=0.4]{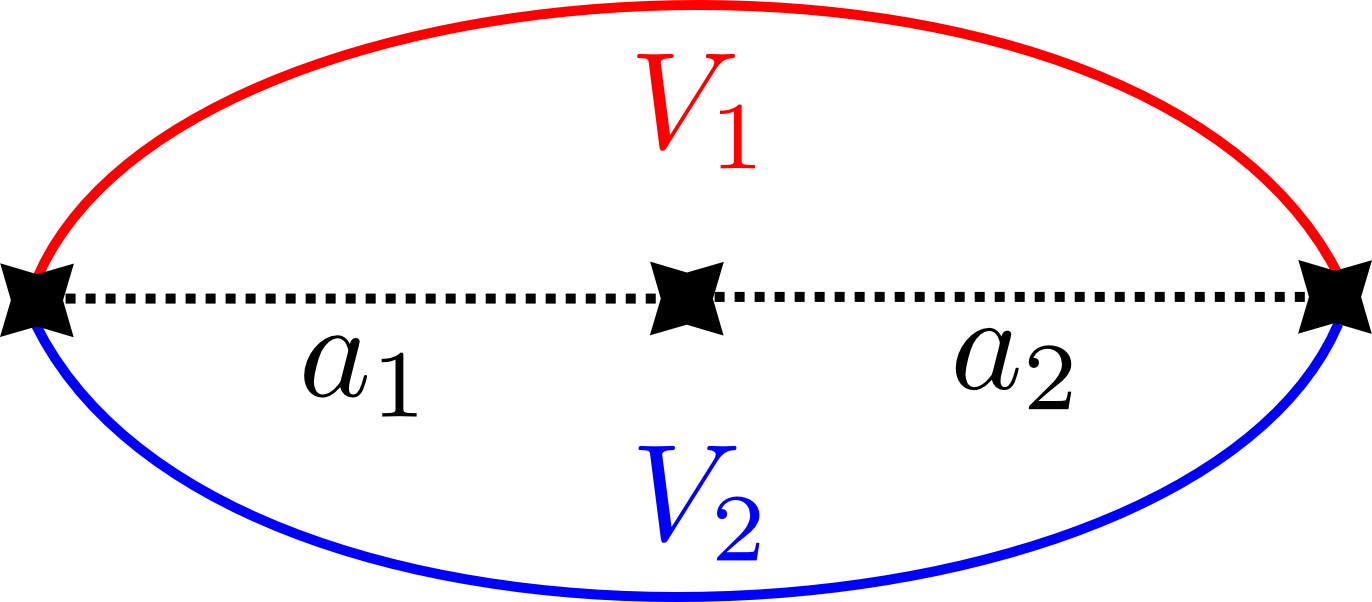}
\caption{Lefschetz bifibration diagram for $(E,\la,\p)$.}\label{fig:A2_harris}
\end{figure}
Following the results in Subsection \ref{ssec:AD_bifiber} we can choose the bifibration $\rho:F_\pi\lr\C$ given by the projection $\rho(x,z)=x$, which describes the Weinstein fiber $(F_\pi,\la)=(A_2,\la_\st)$ as the total space of a Lefschetz fibration with regular fiber the cotangent bundle $F_\rho=\{z^2+w^2=\ln 2\}\cong(T^*S^1,\la_\st,\p_\st)$ and three critical points. And indeed, the vanishing cycles of these three critical points are Lagrangian isotopic to the Lagrangian zero section $S^1\sse(T^*S^1,\la_\st,\p_\st)$. Following Recipe \ref{dictionary}, we need to depict the two vanishing cycles in $V_\pi=\{V_1,V_2\}$ in the complex plane $\C=\im\rho$ and express them as half--twists in terms of a basis for $(F_\pi,\la)$ given by matching paths.

The three critical points of the fibration $\rho$ lie in the real line and for pictorial purposes we assume them to be $(-1,0)$,$(0,0)$ and $(1,0)$: the matching path basis we choose is the linear $A_2$--basis given by the two segments $a_1=[-1,0]\times\{0\}$ and $a_2=[0,1]\times\{0\}$. Figure \ref{fig:A2_harris} depicts this and the two vanishing cycles of the initial fibration $\pi$, which algebraically are given by:
$$V_1=\tau_{a_2}a_1,\quad V_2=\tau_{a_1}a_2.$$
The bifibration diagram in Figure \ref{fig:A2_harris} is referred to as the {\it eye}, due to its shape, and it contains all the Weinstein information of $(E,\la,\p)$. The fourth step in Recipe \ref{dictionary} is to draw the Legendrian front associated to the bifibration diagram: using Proposition \ref{prop:stacking} we obtain the picture in Figure \ref{fig:A2_harrisFront}.

\begin{figure}[h!]
\includegraphics[scale=0.5]{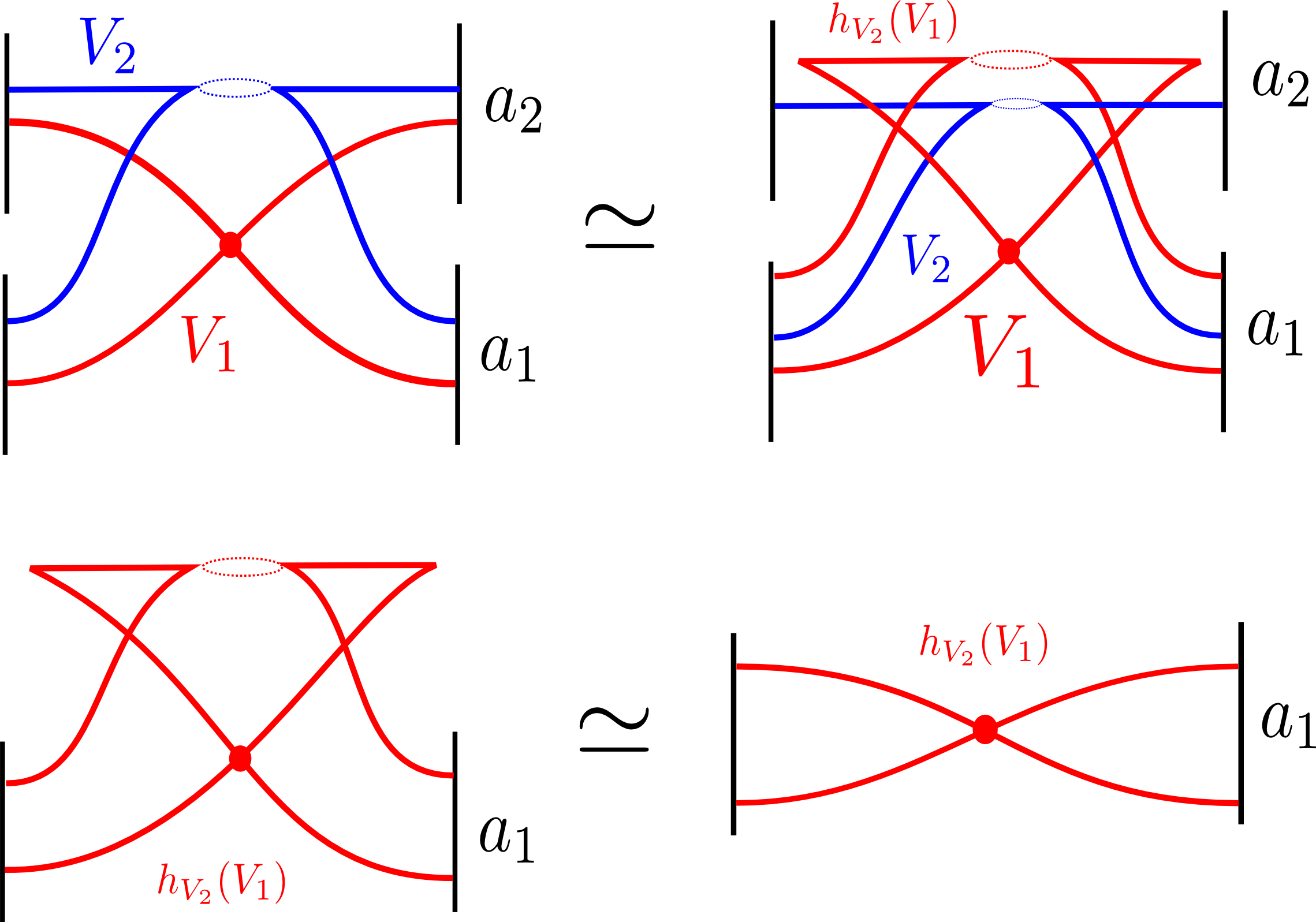}
\caption{The Legendrian Front for the Eye}\label{fig:A2_harrisFront}
\end{figure}

The reason the Legendrian front in Figure \ref{fig:A2_harrisFront} is relevant, and thus $(E,\la,\p)$ is an interesting Weinstein manifold, is the following result:

\begin{thm}[\cite{Ha,May,MuSi, Siegel}]\label{thm:subflex}
The Weinstein 6--fold $(E,\la,\p)$ is not flexible, but it embeds as a Weinstein sublevel set into the unique flexible Weinstein structure $(T^*S^3,\la_f,\p_f)$.
\end{thm}

\begin{proof}
The contribution in this article is a direct diagramatic proof of the second part of the statement. The first part has been proven in the articles \cite{Ha,MuSi, Siegel} with the use of pseudoholomorphic curves; let us however give the intuitive idea behind the fact that the Lefschetz bifibration diagram in Figure \ref{fig:A2_harris} yields a non--flexible Weinstein manifold. Consider the pair of vanishing cycles $(V_1,V_2')=(V_1,V_1)$ in Figure \ref{fig:A2_harrisnotflex}, the third critical value of the bifibration $\rho$ does not interact with the vanishing cycles and thus the resulting Weinstein 6--fold can be constructed as the Weinstein 6--fold associated to the right hand side of Figure \ref{fig:A2_harrisnotflex}, with two critical points and two vanishing cycles, and a subcritical 2--handle. This manifold is $(T^*S^3,\la_\st,\p_\st)$ with a subcritical handle attached to it, and thus a non--flexible Weinstein manifold since it contains the exact Lagrangian 3--sphere zero section.

\begin{figure}[h!]
\includegraphics[scale=0.6]{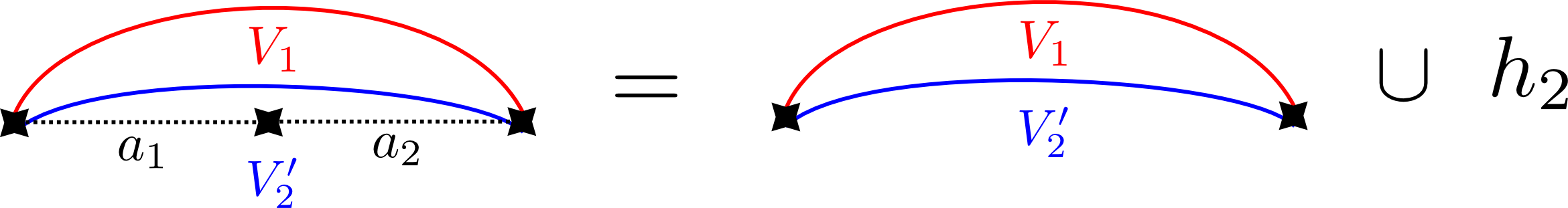}
\caption{Adding a critical value for the bifibration $\rho$ away from the vanishing cycles is tantamount to a Weinstein handle subcritical attachment.}\label{fig:A2_harrisnotflex}
\end{figure}

The vanishing cycle $V_2'$ is not isotopic to $V_2$, but it \emph{is} Hamiltonian isotopic to it after deforming the symplectic structure by a non-exact $2$--form. This shows that a non--exact deformation of $(E,d\la)$ is symplectomorphic to a non--exact deformation of $(T^*S^3 \cup h_2,d\la_\st)$. In the symplectic six--fold $(T^*S^3 \cup h_2,d\la_\st)$ the $2$--dimensional homology class can be made disjoint from the zero section $S^3\sse T^*S^3$ and thus we conclude that the deformed manifold still has an exact Lagrangian $3$-sphere. This implies non--flexibility of $(E,\la)$ for a flexible manifold should have no Lagrangian spheres even in its non--exact deformations \footnote{This is an intuitive geometric argument yet not rigorously established: see the article \cite{Siegel, MuSi} for an alternative argument which is rigorous and morally the same.}.

\begin{figure}[h!]
  \includegraphics[scale=0.4]{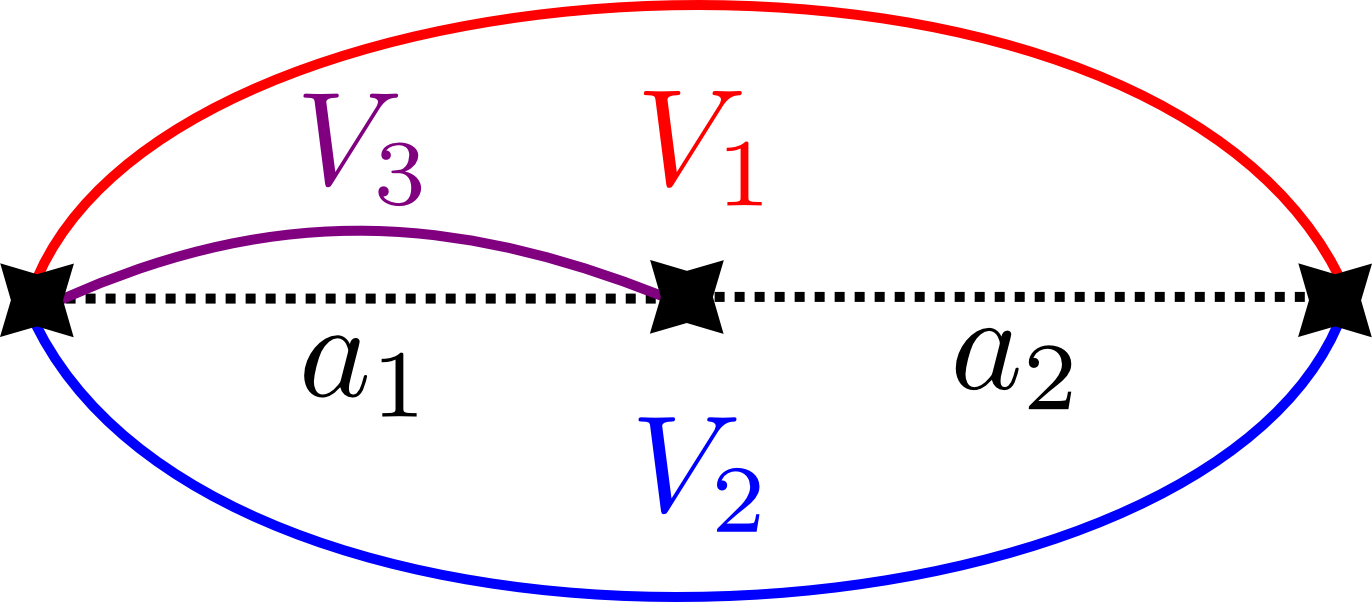}
\caption{Un Chien Andalou.}\label{fig:A2_DogsLF}
\end{figure}

Let us now prove that $(E,\la,\p)$ embeds into the flexible Weinstein manifold $(T^*S^3,\la_f,\p_f)$ using Figure \ref{fig:A2_harrisFront}. Consider the Lefschetz bifibration diagram in Figure \ref{fig:A2_DogsLF}, referred to as the Andalusian dog \cite{BD}, which is obtained from the bifibration diagram in Figure \ref{fig:A2_harris} by adding the third vanishing cycle $V_3=a_1$. The associated Weinstein 6--fold $(\wt E,\wt\la,\wt\p)$ is obtained from $(E,\la,\p)$ by adding a critical Weinstein 3--handle along the Legendrian lift of the Lagrangian vanishing cycle $V_3=a_1$, and we claim that the Andalusian dog is a flexible Weinstein 6--fold. Indeed, using Recipe \ref{dictionary} again we obtain the Legendrian attaching link associated to $(\wt E,\wt\la,\wt\p)$ as depicted in \ref{fig:A2_DogsFront}, which simplifies to the loose Legendrian unknot, and thus exhibiting an explicitly flexible Legendrian handlebody for $(\wt E,\wt\la,\wt\p)$.

\begin{figure}[h!]
\includegraphics[scale=0.55]{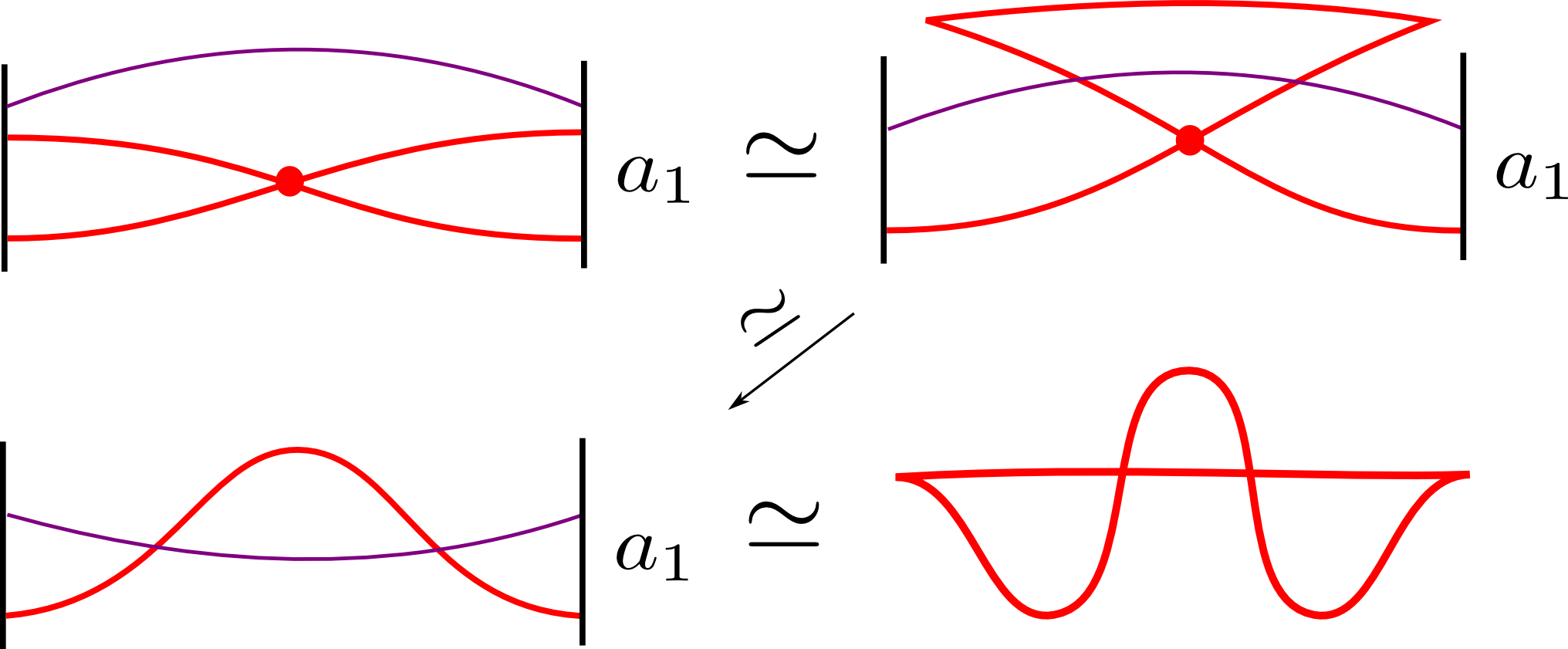}
\caption{Front for the Andalusian Dog}\label{fig:A2_DogsFront}
\end{figure}

The Weinstein 6--fold described by the Legendrian handlebody in Figure \ref{fig:A2_DogsFront} is the flexible Weinstein structure $(\wt E,\wt\la,\wt\p)\cong(T^*S^3,\la_f,\p_f)$, and since it is obtained from $(E,\la,\p)$ by a Weinstein handle attachment there exists a Weinstein embedding of our initial Weinstein manifold $(E,\la,\p)$ into the flexible Weinstein $(T^*S^3,\la_f,\p_f)$\footnote{In fact, $(T^*S^3, \la_f)$ admits a Liouville embedding into $(\C^3, \la_\std)$ \cite{EM}, so $(E, \la)$ has a Liouville embedding into $\C^3_\std$ as well; however one can see that this embedding cannot be as a Weinstein sublevel set for topological reasons.}, which proves the statement.
\end{proof}

The machinery constituting Recipe \ref{dictionary} being established and with the practice of applying it in this subsection, we now proceed to the proof of the more elaborate results stated in Section \ref{sec:intro}.

\section{Applications}\label{sec:app}

In this section we apply Recipe \ref{dictionary} to prove the results stated in Section \ref{sec:intro}, also the results will be proven in an increasing order to difficulty regarding the use of Recipe \ref{dictionary}.

First, in Subsection \ref{ssec:mirror1} we study the Weinstein manifold $(Y,\la,\p)$ which features in Theorem \ref{thm:mirror}. In this case, Recipe \ref{dictionary} will be applied in complete detail and the Legendrian handlebody for $(Y,\la,\p)$ shall appear effortlessly. Its relation with mirror symmetry, and thus the conclusion of Theorem \ref{thm:mirror}, is deferred until Subsection \ref{ssec:mirror2} in order to emphasize the construction of Legendrian handlebodies first.

Second, Theorem \ref{thm:Xab1} is proven in Subsection \ref{ssec:Xab} where Recipe \ref{dictionary} will be applied with a non--linear choice of basis of matching paths: although the Weinstein bifiber is symplectomorphic to an $A_k$--plumbing, the non--linear basis is more efficient and being able to use it also illustrates the flexibility allowed by Recipe \ref{dictionary}.

Third, the Koras--Russell cubic is proven to be Stein deformation equivalent to $(\C^3,\la_\st,\p_\st)$ in Subsection \ref{ssec:kr}, thus proving Theorem \ref{thm:kr1} as promised. In this case, the Weinstein bifiber is not a linear plumbing, but a $D_4$--plumbing, and there are words in Dehn twists which are not linear either. Fortunately, our work in Section \ref{sec:kirby} and \ref{sec:lef} will prove valuable and Recipe \ref{dictionary} shall provide the desired Legendrian handlebody.

Fourth, Subsection \ref{ssec:mirror2} studies the Weinstein manifolds in Theorem \ref{thm:torusIntro}, exhibiting the exact Lagrangian $S^1\times S^{n-1}$, and in particular the Weinstein 4--fold $(X,\la_\st,\p_\st)$ featuring Theorem \ref{thm:mirror}. The discussion then continues with the connection between the symplectic field theory of these Legendrian handlebodies and mirror symmetry, serving as the open--ended experimental evidence with which we finish this article.

\subsection{Theorem \ref{thm:mirror}: Part I}\label{ssec:mirror1} In this subsection we use Recipe \ref{dictionary} to study the symplectic topology of the affine surface
$$(Y,\la,\p) = \{(x,y,z):xyz + x + z +1 = 0\} \sse (\C^3,\la_\st,\p_\st).$$
First, we endow it with a Lefschetz fibration, in this case given by a generic linear section:
$$\pi:Y\lr\C,\quad\pi(x,y,z)=-3x-2z+y.$$
The regular function $\pi$ is a Lefschetz fibration with five critical values and $0$ as a regular value. Thus $(Y,\la,\p)$ is constructed from the subcritical skeleton $(F_\pi,\la)=(\pi^{-1}(0),\la_\st|_{\pi^{-1}(0)})$ by attaching five 2--handles along the contact boundary of $(F_\pi\times D^2,\la+\la_\st,\p+\p_\st)$.

The regular fiber $(F_\pi,\la)\sse\C^2[x,z]$ is a smooth algebraic curve embedded in the affine plane and, before performing the second step, let us notice that in such cases we can understand the topology of the curve by classical algebraic geometry methods. Indeed, the topology is determined by the number of boundary components $\pi_0(\dd F_\pi)$ and the genus of the algebraic curve $F_\pi$. The former can be obtained by computing the intersection of the projectivization $\ol{F_\pi}$ of the affine curve $F_\pi = \{xz(3x+2z)+x+z\} \sse \C^2$ to the projective plane $\P^2[x,z,w]$ with the line at infinity $\{w=0\}\cong\P^1_\infty\sse\P^2$. The latter is just the dimension $h^0(F_\pi,\Omega^1)$ of the vector space of holomorphic 1--forms and thus it can be computed as the number of integral lattice points in the interior of the Newton polytope of a defining polynomial for $F_\pi$. In our case we have that
$$|\dd F_\pi|=|\ol{F_\pi}\cap\P^1_\infty|=|\{xz(3x+2z)+xw^2+zw^2=w^3\}\cap\{w=0\}|=3,$$
$$g(F_\pi)=|\langle (2,1),(1,2),(1,0),(0,1),(0,0) \rangle\cap\Z^2|=1,$$
where $\langle P_1,P_2,P_3,P_4,P_5\rangle$ denotes the interior of the convex hull of the five points $P_i\in\R^2$, for $1\leq i\leq5$. In consequence, the algebraic curve $F_\pi$ is a thrice punctured torus, and the set $\L$ in the first step of Recipe \ref{dictionary} can be taken to be the set $\{\a,\beta,\gamma,\delta\}$ of four curves as depicted in Figure \ref{fig:Tpqr}. The curves in this figure were denoted by $T,P,Q,R$ in Subsection \ref{ssec:Tpqr}, and we rename them $\alpha=T$, $\beta=P$, $\gamma=Q$ and $\delta=R$.

Following Step 2 in Recipe \ref{dictionary}, we next endow the Weinstein fiber $(F_\pi,\la)$ with the Lefschetz fibration $\rho:F_\pi\lr\C$ given by $\rho(x,z)=3x+z$,
which expresses $(F_\pi,\la)$ as a regular $3$--fold branched cover of $\C$ with six branch points. The regular fiber $\Sigma_\rho=\rho^{-1}(0)$ of this bifibration consists of three points that we can express as $a\vee b$, where $a$ and $b$ are basepointed copies of the 0--sphere $S^0$. In this step we also need to express the set of Lagrangian curves $\L=\{\a,\beta,\gamma,\delta\}$ in terms of matching paths for the Lefschetz bifibration $\rho$. In order to achieve so we need the information of the branched cover at the critical values, i.e.~which of the three possible $0$--spheres $a,b$ or $\tau_b(a)$ are the vanishing cycles at each branch point.

We can depict the six critical values of $\rho$ as the sixth roots of unity $\{\zeta^k_6\}_{k=1,\ldots,6}$ and the corresponding vanishing cycles $\{\tau_b(a),a,b,\tau_b(a),a,b\}$ in their corresponding cyclic order. This information allows us to find four matching paths in $\C=\im(\rho)$ which realize the four curves in the set $\L\sse\Sigma_\rho$ in the $D_4$--intersection pattern. Figure \ref{fig:LFtrefoil} depicts these critical values with the information of the branching and a possible choice of matching paths.

\begin{figure}[h!]
\centering
  \includegraphics[scale=0.35]{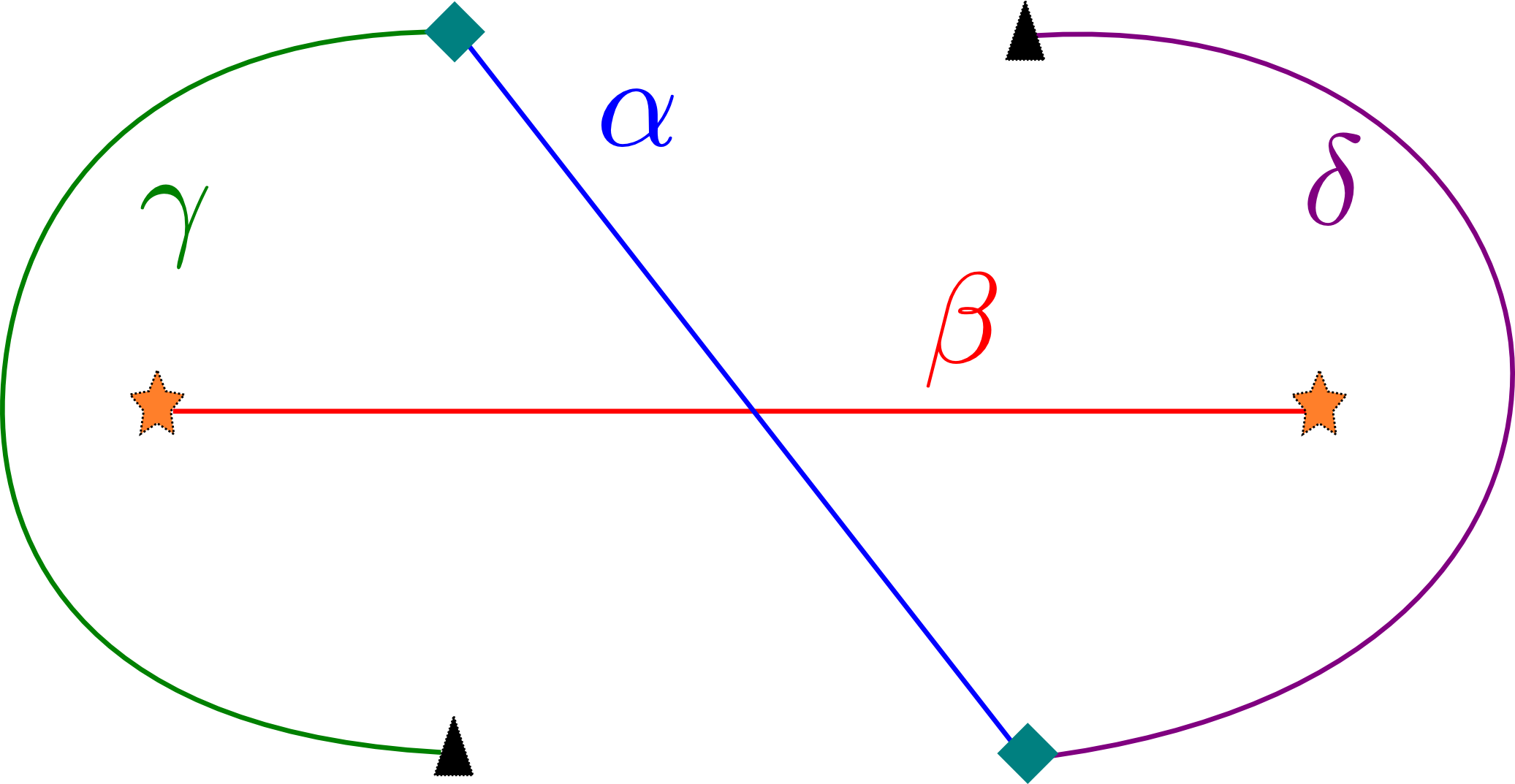}
  \caption{The image of $\rho$ with its critical values: the triangle indicates the vanishing cycle $a$, the square indicates the vanishing cycle $b$ and the five--pointed star indicates the vanishing cycle $\tau_b(a)$. In addition, we have drawn a set of matching paths realizing the $D_4$--intersection pattern present in the 1--skeleton of $F_\pi$.}
  \label{fig:LFtrefoil}
\end{figure}

\begin{remark}
It might appear unnecessary to think of $0$--spheres, but the strong advantage of being this systematic is that this process also gives a Legendrian handlebody for the higher--dimensional manifolds
$$\{(x,y,z,\underline{u}):xyz+x+z+\sum_{i=1}^{n-2}u_i^2=1\}\sse\C^{n+1},$$
for any $n\geq0$, just by rotation symmetry. In these higher dimensional cases however there will be no relations such as $\tau^2_a=\tau^2_b=\id$, and thus is it worth keeping track of the correct words in Dehn twists and Figures \ref{fig:LFtrefoil} and \ref{fig:VCtrefoil} are still valid.\hfill$\Box$
\end{remark}

The third step in Recipe \ref{dictionary} is to draw the image of the vanishing cycles of the fibration $\pi$ by the bifibration $\rho$. In practice, this is the study of the critical values of the bifibration $\rho(x,y)=3x+z$ on the varying domains $F_t=\{xz(t+3x+2z)+x+z=1\}$, for $t\in\C$. In case the value of $t\in\C$ is critical for the map $\pi$ the fiber $F_t$ is singular, which occurs in five occasions, and otherwise we have a symplectomorphism $(F_t,\la_\st)\cong (F_\pi,\la)$ provided by the parallel transport of a symplectic connection. Tracing these critical values of the fibration $\rho$ on the fiber $(F_t,\la_\st)$ as the complex variable $t$ goes from the origin to each of the critical values of $\pi$ produces the images under $\rho$ of the vanishing cycles associated to each of the the critical points of $\pi$. In our case, the resulting matching paths describing these vanishing cycles in $V_\pi$ are depicted in Figure \ref{fig:VCtrefoil}.

\begin{figure}[h!]
\centering
  \includegraphics[scale=0.4]{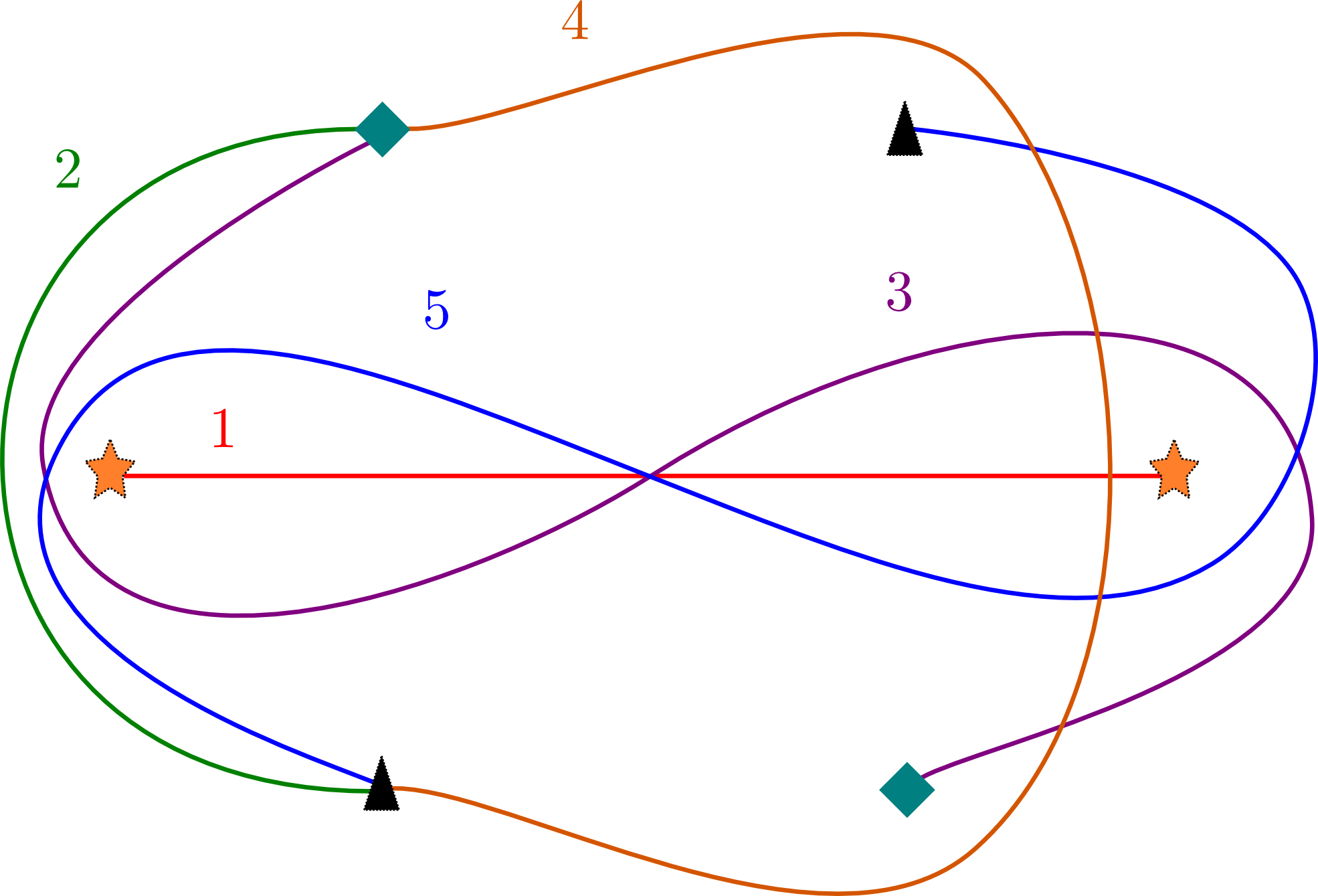}
  \caption{The matchings paths corresponding to the five vanishing cycles in $V_\pi\sse(F_\pi,\la,\p)$. This is the picture of Step 3 in Recipe \ref{dictionary}.}
  \label{fig:VCtrefoil}
\end{figure}

The fourth step is to express these matching paths as half--twist on the given $D_4$--basis $\L=\{\alpha,\beta,\gamma,\delta\}$. In this case, the resulting expressions are
$$V_1=\beta,\quad V_2=\gamma,\quad V_3=\tau_\beta(\alpha),\quad V_4=\tau_\alpha\tau_\beta\tau^{-1}_\delta\tau^{-1}_\gamma(\alpha),\quad V_5=\tau^{-1}_\delta\tau^{-1}_\gamma(\alpha).$$

The fifth step is to draw the Legendrian lifts of the five exact Lagrangians in $V_\pi$, according to Proposition \ref{prop:stacking}. Together with the sixth step, which just tells us to order the link by height, we obtain the diagram depicted in Figure \ref{fig:fronttrefoil}, which is the desired Weinstein handlebody for the affine manifold $(Y,\la,\p)$. Finally, the seventh step of Recipe \ref{dictionary} tells us to simply the diagram as much as possible: as seen in the sequence depiced it Figure \ref{fig:trefoil}, this simplification results in the maximal Thurston--Bennequin Legendrian right--handed trefoil.

\begin{figure}[h!]
\centering
  \includegraphics[scale=0.55]{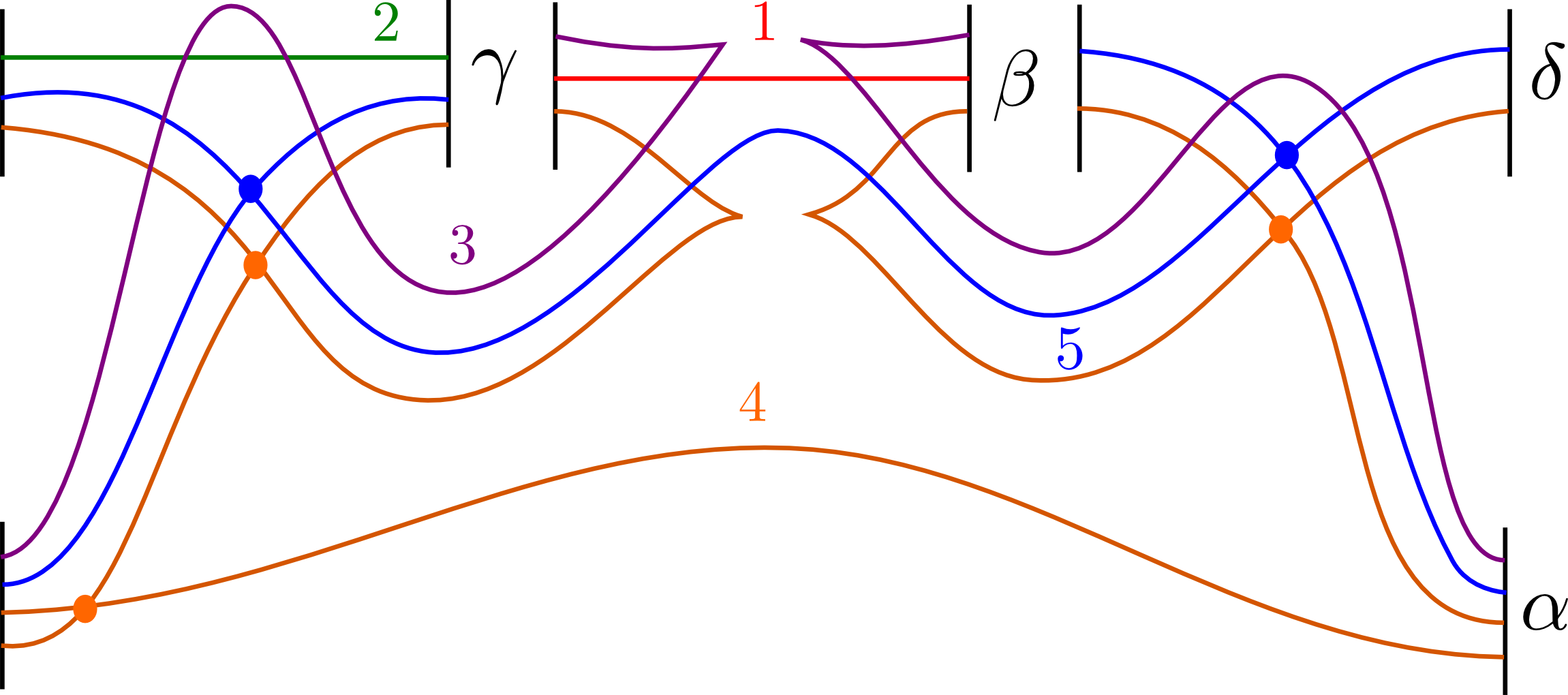}
  \caption{Front of the Legendrian lifts of the vanishing cycles in Figure \ref{fig:VCtrefoil}.}
  \label{fig:fronttrefoil}
\end{figure}

\begin{figure}[h!]
\centering
  \includegraphics[scale=0.75]{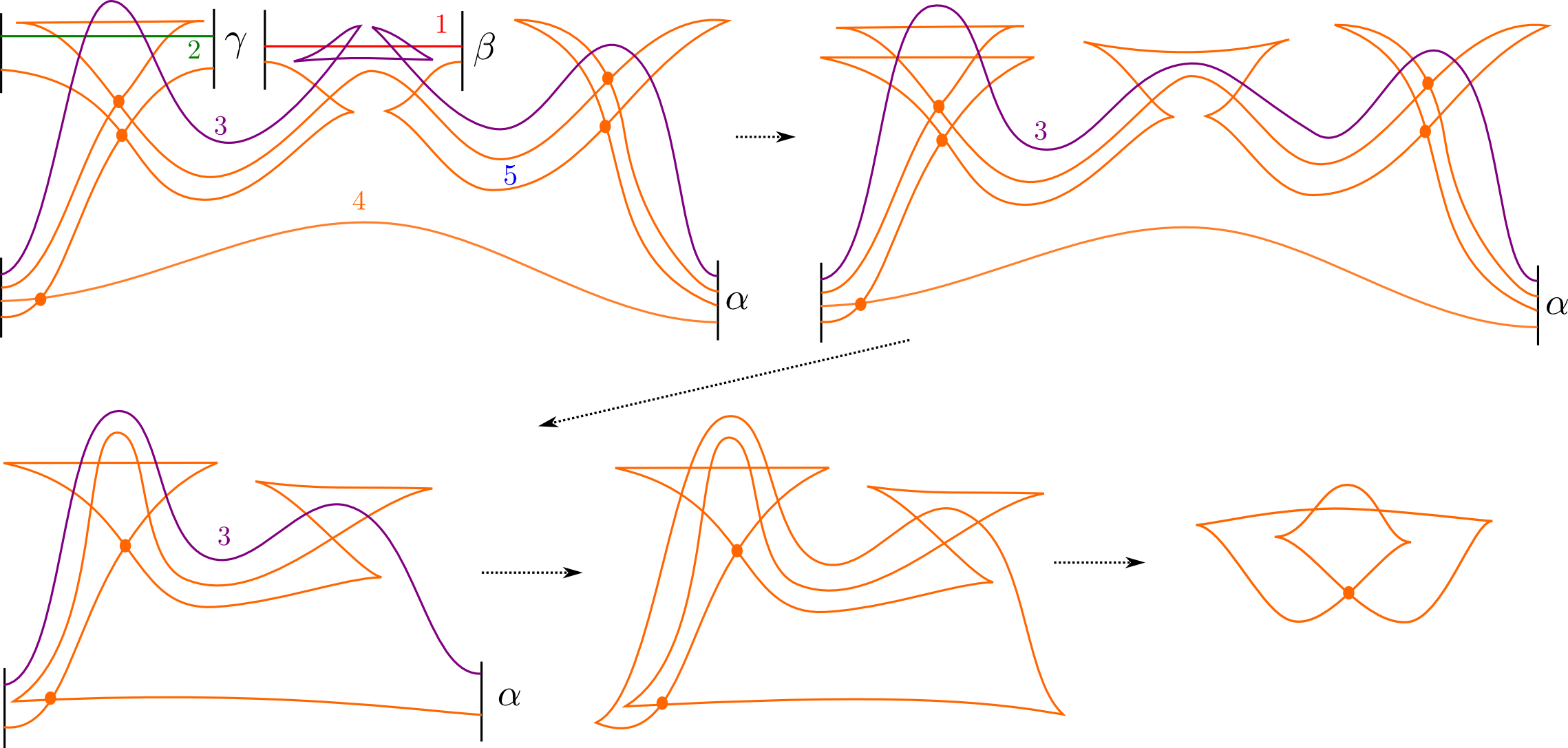}
  \caption{A sequence of Legendrian isotopies and Kirby moves, reducing the link in Figure \ref{fig:fronttrefoil} to the trefoil.}
  \label{fig:trefoil}
\end{figure}

In order to relate this result to mirror symmetry, it will suffice to compute the Legendrian differential graded algebra of the attaching Legendrian knot. Since this differential graded algebra has been computed in the literature \cite{EtNotes}, we just notice that it be given by the presentation 
$$\A=\{\langle a,b,x,y,z\rangle,|a|=|b|=2,|x|=|y|=|z|=1: da=1+x+z+xyz,\quad db=1+x+z+zyx\}.$$
The zeroth degree homology of the algebra $(\A,d)$ is the commutative polynomial algebra
$$R=\C[x,y,z]/(1+x+z+xyz),$$
since the commutators lie in the image of the differential:
$$d\left(a(1+yx) - (1+xy)b\right) = yx - xy, \quad d\left((1+zy)a - b(1+yz)\right) = zy - yz,$$
$$\text{and} \quad d(za-bz) = zx - xz + zxyz - zyxz,$$
and note that the spectrum of this commutative ring $R$ is the algebraic surface $Y\sse\C^3$ itself. Thus, the Legendrian trefoil $\La\sse(S^3,\xi_\st)$ is related to the equation
$$xyz + x + z + 1 = 0$$
in two distinct ways. The first is strictly symplectic topological: attaching a Weinstein handle to the Legendrian trefoil describes the symplectic manifold $\{xyz+x+z+1=0\} \sse \C^3$. The second is purely Floer-theoretic: the degree zeroth part of the homology of its Legendrian algebra is isomorphic to the algebra
$$\mbox{LCH}_0(\La) \cong \C[x,y,z]/(xyz+x+z+1).$$
This pattern should not be expected in general, and is due to the fact that the affine manifold $(Y,\la_\st,\p_\st)$ is self--mirror: we will continue the discussion on homological mirror symmetry in Subsection \ref{ssec:mirror2}. Thus far, we have obtained a Legendrian handlebody for the Weinstein manifold $(Y,\la_\st,\p_\st)$ whose symplectic topology, thanks to the fact that the Legendrian trefoil $\La\sse(S^3,\xi_\st)$ is well--studied, we now understand in depth.

\subsection{The Weinstein Manifolds $X_{a,b}$}\label{ssec:Xab}

In this subsection we study the interesting class of affine Stein manifolds introduced in Theorem \ref{thm:Xab1}. Given two coprime integers $a,b\in\N$, $a\leq b$, we aim to understand the symplectic topology of the Stein manifolds
$$X^n_{a,b}=\left\{(x,y,\underline{z}):x^ay^b+\sum_{i=1}^{n-1}z_i^2=1\right\}\sse\C^{n+1},$$
up to Stein deformation equivalence. This class of affine manifolds has been studied both in symplectic topology \cite{MS,Se15} and algebraic geometry \cite{DF,FM,Kr}. From the symplectic viewpoint, P.~Seidel computed their $q$--intersection numbers in \cite{Se15} and, joint with M.~Maydanskiy, they showed that the symplectic cohomology of the Weinstein manifolds $X_{1,b}$ vanishes. This will be readily implied by Corollary \ref{cor:Xab} below.

From the algebraic geometry perspective, the Danielewski surfaces $X^2_{1,b}$ were the first counter--examples to the Cancellation Conjecture for algebraic cylinders and variations on this class of affine varieties have been thoroughly studied \cite{Be,Du}; in particular, the affine complex 3--spheres $X^3_{1,b}$ have been shown to be algebraically distinct from the affine conic $X_{1,1} \cong T^*S^3$ \cite{DF}. This is also a consequence of our results which further imply that $X^3_{1,b}$ are all Stein deformation equivalent to each other for $b\geq2$.

The main result we need on the Weinstein manifolds $X^n_{a,b}$ is a description of a Legendrian handlebody, from which their symplectic topology can be better understood. This is the content of the following theorem:

\begin{thm}\label{thm:Xab}
Consider two coprimer integers $a,b\in\N$, $a\leq b$, and the Weinstein manifolds
$$X^n_{a,b}=\{(x,y,\underline{z}):x^ay^b+\sum_{i=1}^{n-1}z_i^2=1\}\sse\C^{n+1}.$$
Then the Weinstein manifold $(X^n_{a,b},\la_\st,\p_\st)$ is obtained by attaching a critical Weinstein handle to $(D^{2n},\la_\st,\p_\st)$ along a cusp $S^{n-2}$--spinning of the $(a,-b)$--Legendrian torus knot $\Lambda_{a,b}$ with $tb(\Lambda_{a,b})= -ab$ and rotation number $r(\Lambda_{a,b})=b-a$.
\end{thm}

\begin{remark}
In the statement of Theorem \ref{thm:Xab} we use a cusp $S^{n-2}$--spinning of a Legendrian front $\Lambda$ for a 1--dimensional Legendrian knot; this is the $n$--dimensional Legendrian submanifold described by the following front: choose an globally extremal cusp of $\Lambda$, rightmost or leftmost, and remove a small symmetric arc--neighborhood of the front containing the cusp singularity. Then consider a vertical axis through the two endpoints of the resulting arc--front and perform a $S^{n-2}$--spinning along that vertical axis. The cusp $S^{n-2}$--spinning of $\Lambda$ is the Legendrian submanifold represented by the resulting front; note that the cusp $S^0$--spinning is a connected sum of a Legendrian with a disjoint copy of itself, endowed with the reversed orientation.
\end{remark}

These explicit Legendrian handlebodies are genuinely useful to symplectic topology, for instance we can immediately conclude from Theorem \ref{thm:Xab} the following

\begin{cor}\label{cor:Xab}
The Weinstein manifolds $X^n_{1,b}$ are flexible for all $b\geq2$, $n\geq3$.\\
The Weinstein manifolds $X^2_{a,b}$ are not flexible for any $(a,b)\in\N\times\N$ with $a\geq2$.\\
The Weinstein manifolds $X^n_{2,b}$ contain exact Lagrangian Klein bottles $S^1\wt\times S^{n-1}$ for odd $b$.
\end{cor}

To our knowledge none of these results were known or expected; the following two observations might help appreciate Corollary \ref{cor:Xab}.

First, the Weinstein manifolds $X^n_{1,b}$ are the first examples of affine varieties which are flexible as Stein manifolds, and Weinstein flexibility of $X^n_{1,b}$ reproves the vanishing of their symplectic cohomology \cite{MS}.

Second, from the singularity theory standpoint the defining polynomials $xy^b-1$ give smooth deformations of the non--isolated singularities $\{(x,y)\in\C^2:xy^b=0\}$. This is interesting for the following reason: given a stabilized isolated plane singularity $$p(x,y)+z_1^2+\ldots+z_{n-1}^2=0,$$
the Stein structure on its Milnor fiber
$$p(x,y)+z_1^2+\ldots+z_{n-1}^2=1$$
contains exact Lagrangian spheres coming from the Morsification, and it is in particular not flexible. In contrast, Corollary \ref{cor:Xab} shows that by considering the simplest non--isolated plane singularity $\{(x,y)\in\C^2:xy^b=0\}$, the Stein structure on the Milnor fiber
$$\{(x,y,z_1,\ldots,z_{n-1}):xy^b+z_1^2+\ldots+z_{n-1}^2=1\}\sse\C^{n+1}$$
is flexible, and thus contains no exact Lagrangian spheres.

Third, Theorem \ref{thm:Xab} also implies that the Weinstein six--folds $X^3_{1,b}$ are all diffeomorphic to $T^*S^3$ since it only consists of a single handle attachment and $\R^3$--bundles over the 3--sphere $S^3$ are smoothly trivial. Further, Corollary \ref{cor:Xab} not only exhibits the standard symplectomorphism $X^3_{1,1}\cong (T^*S^3,\la_\st,\p_\st)$ but shows the unexpected fact that $X^3_{1,b}$ are all Stein deformation equivalent to $X^3_{1,2}$ for any value $b\geq2$.

\begin{remark}
In addition, note that there is no combinatorial realization of this Stein deformation by using Hurwitz moves in an $A_k$--Lefschetz fibration, and that $X^3_{1,b}$ are pairwise not algebraically isomorphic \cite{DF}.\hfill$\Box$
\end{remark}

Let us first prove Corollary \ref{cor:Xab}, and then we will proceed with Theorem \ref{thm:Xab}.

\begin{proof}[Proof of Corollary \ref{cor:Xab}]
To conclude the flexibility of the Weinstein manifolds $X_{1,b}$, for the values $b \geq 2$, note that the Legendrian knot $\Lambda_{1,b}\sse(S^3,\xi_\st)$ is a stabilized unknot for these $b\geq2$. Then, the results in Subsection \ref{ssec:loose} imply that any cusp spinning of the loose Legendrian $\Lambda_{1,b}$ will be a loose Legendrian submanifold in $(S^{2n-1},\xi_\st)=\dd(D^{2n},\la_\st)$, which implies that $X_{1,b}$ is a flexible Weinstein manifold.

The non--vanishing $\mbox{SH}^*(X^2_{a,b}) \neq 0$ of symplectic cohomology for any $a,b \geq 2$ is obtained by applying the Legendrian surgery exact sequence \cite{BEE} and using the fact that the Legendrian knot $\Lambda_{a,b}\sse(S^3,\xi_\st)$ has non--vanishing Legendrian contact homology. This latter fact holds because these Legendrian differential graded algebras have either augmentations or $2$--dimensional representations \cite{Siv}.

\begin{figure}[h!]
  \centering
  \includegraphics[scale=0.85]{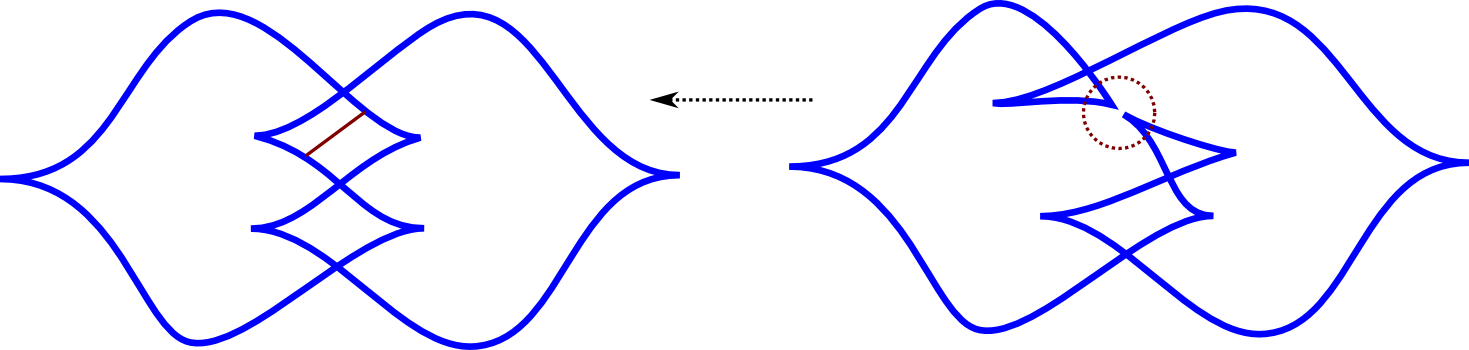}
  \caption{A M\"obius strip Lagrangian filling of $\Lambda_{2,b}$.}
\label{fig:mobfill}
\end{figure}

In order to proof the third statement of the corollary, let us now show that the Weinstein manifolds $X^n_{2,b}$ have each an exact Lagrangian $S^1\wt\x S^{n-1}$ for any odd $b\in\N$. For that, observe that each Legendrian knot $\Lambda_{2,b}\sse(S^3,\xi_\st)$ is the boundary of an exact Lagrangian M\"obius strip: indeed, we can construct these Legendrian knots from the Legendrian unknot by using the single non--oriented Legendrian ambient surgery indicated in Figure \ref{fig:mobfill}. In consequence, the cusp $S^{n-2}$--spinning of the Legendrian knot $\Lambda_{2,b}$ is the boundary of an exact Lagrangian $L\sse(D^{2n},\la_\st)$ whose smooth topology is $L \cong (M \x S^{n-2}) \cup h_{n-1}$, where $M$ is a M\"obius strip and $h_{n-1}$ is an $(n-1)$--handle attached along $\{\mbox{pt}\} \x S^{n-2} \sse \dd M \x S^{n-2}$. Thus, the exact Lagrangian filling $L$ of the Legendrian $\dd L\cong S^{n-1}$ is diffeomorphic to the punctured product $(S^1 \x S^{n-1}) \sm D^n\sse(D^{2n},\la_\st)$, which is then completed to the exact Lagrangian $S^1\wt\times S^n$ in $X^n_{2,b}$ when the Weinstein handle along $\dd L$ is attached.
\end{proof}

\begin{remark}
The assumption that the integers $a$ and $b$ are coprime is not essential, though the conclusions will be different for distinct pairs of integers. For example, the affine manifold
$$(X^n_{2,2},\la,\p) = \{(x,y,\underline{z}):x^2y^2 + \sum_{i=1}^{n-1} z_i^2 = 1\}\sse(\C^{n+1},\la_\st,\p_\st)$$
contains an exact Lagrangian $S^1 \x S^{n-1}$: we encourage the reader to prove this and explore other examples.\hfill$\Box$
\end{remark}

Let us now prove Theorem \ref{thm:Xab}.

\begin{proof}[Proof of Theorem \ref{thm:Xab}]
In order to apply Recipe \ref{dictionary} we first need to endow the Weinstein manifold $(X^n_{a,b},\la_\st,\p_\st)$ with a Weinstein Lefschetz fibration: we consider the regular map
$$\pi:X^n_{a,b}=\{(x,y,\underline{z}):x^ay^b+\sum_{i=1}^{n-1}z_i^2=1\}\lr\C,\quad \pi(x,y,z_1,\ldots,z_{n-1})=ax+by.$$
This Lefschetz fibration has $(a+b)$ critical points and Weinstein fiber
$$(F_\pi,\la,\p)\cong (A^{2n-2}_{a+b-1},\la_\st,\p_\st).$$
For the second step in Recipe \ref{dictionary} we shall use the auxiliary Lefschetz bifibration
$$\rho:(F_\pi,\la)\longrightarrow\C,\quad \pi(x,z_1,\ldots,z_{n-1})=x,$$
discussed in Subsection \ref{ssec:AD_bifiber}. The fibration $\rho$ has also $(a+b)$ critical points which can be assumed to be located at the $(a+b)$--roots of unity. The Weinstein bifiber is the affine conic $(F_\rho,\la,\p) \cong (T^*S^{n-2},\la_\st,\p_\st)$ and, in order to execute the third and fourth steps, we must describe the cyclically ordered set of vanishing cycles
$$V_\pi=\{V_1,\ldots,V_{a+b}\}$$
in terms of matching paths for $\rho:F_\pi\lr\C$, which is done as follows. Fix an $(a+b)$--root of unity $\zeta\in\C$ and denote by $R_i$ a straight path joining $\zeta$ with the $i$th $(a+b)$--root of unity, where we are ordering the roots counterclockwise and $1\leq i\leq a+b$. Then we have the following algebraic description:

$$ V_i =
  \begin{cases}
    R_a       & \quad \text{if } i=1,\\
    \tau_{R_{i+a-1}}(R_{i-1})       & \quad \text{if } 2\leq i \leq b,\\
    R_b       & \quad \text{if } i=b+1,\\
    \tau_{R_{i-1}}(R_{i-b-1}) & \quad \text{if } b+2 \leq i.\\
  \end{cases}
$$

\begin{figure}[h!]
  \centering
  \includegraphics[scale=0.45]{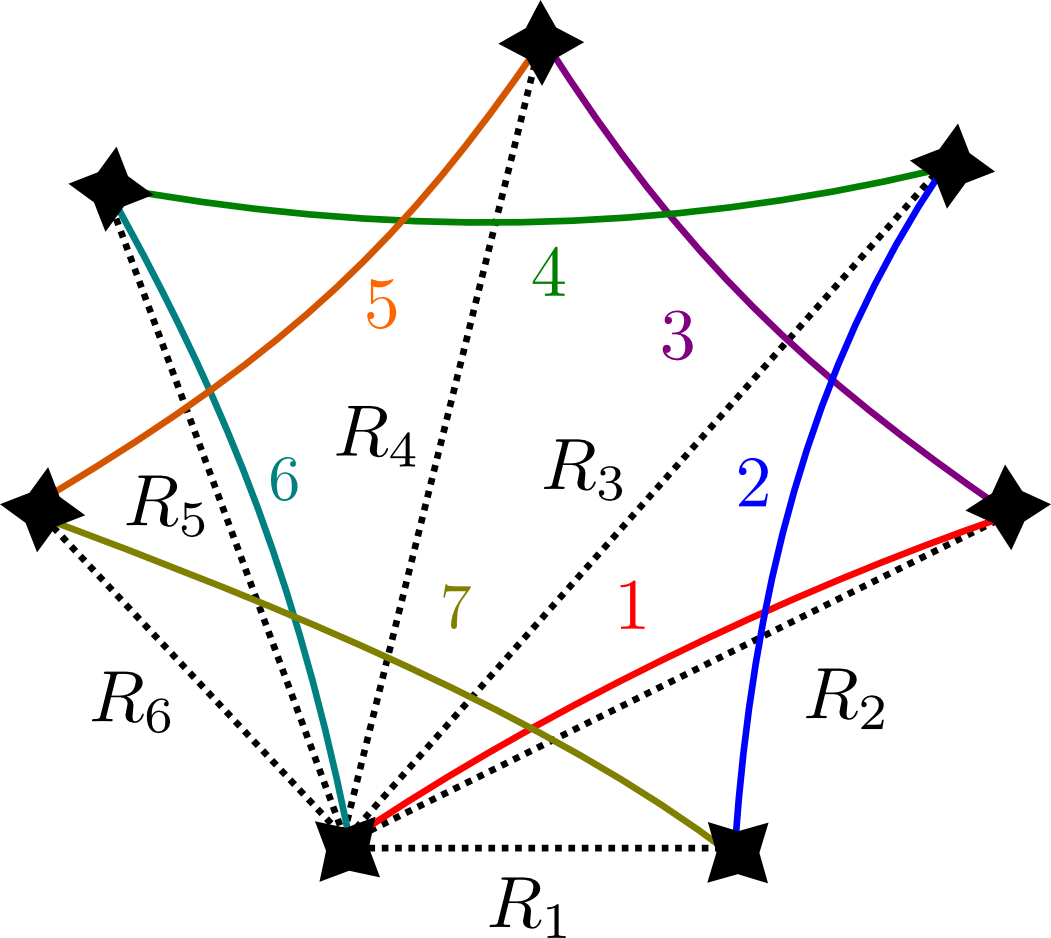}
  \caption{The Lefschetz bifibration diagram for the Weinstein manifold $X_{2,5}$, and the Lagrangian spheres $\{R_j\}$ used to describe the vanishing cycles.}
\label{fig:X25LF}
\end{figure}

\begin{figure}[h!]
  \centering
  \includegraphics[scale=0.45]{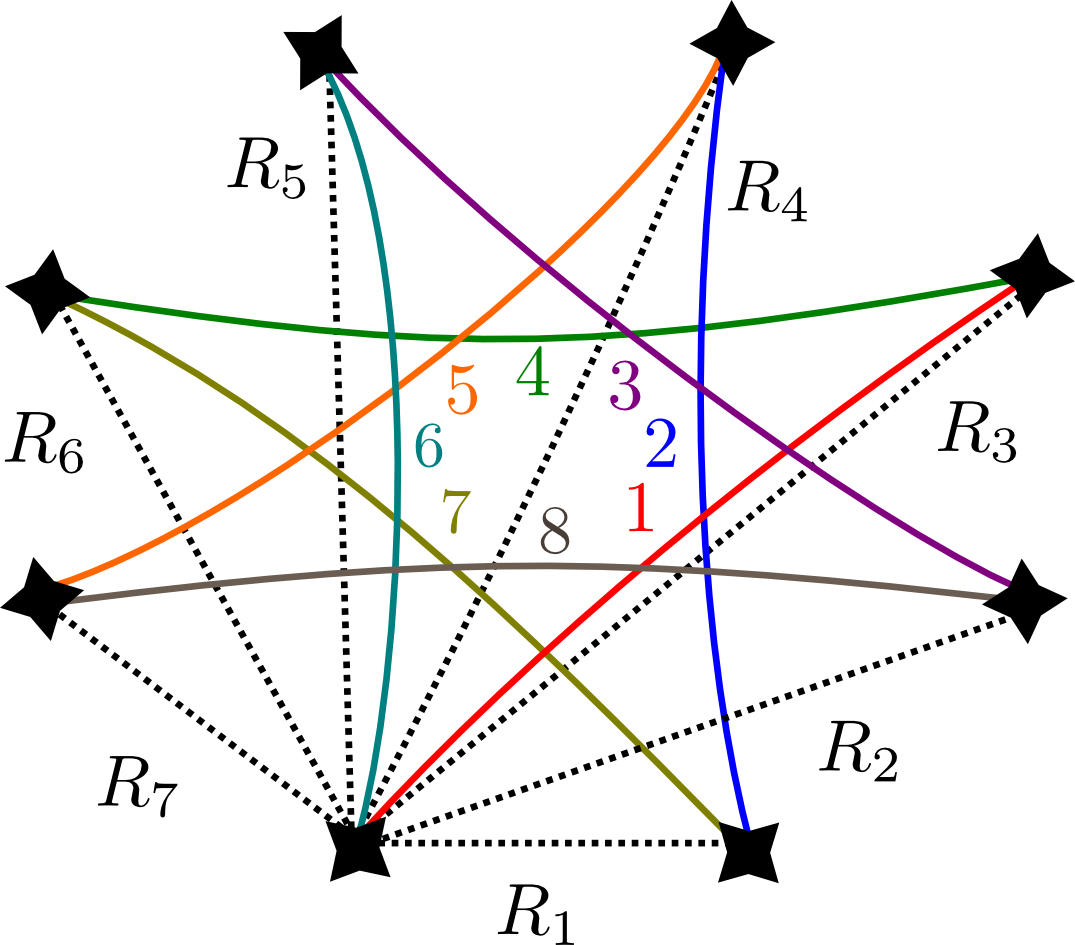}
  \caption{The equivalent to Figure \ref{fig:X25LF} for the Weinstein manifold $X_{3,5}$.}
\label{fig:X35LF}
\end{figure}

Both Figures \ref{fig:X25LF} and \ref{fig:X35LF} depict these matching paths for the vanishing cycles of the fibration $\pi$ in the cases $(a,b)=(2,5)$ and $(3,5)$; this part being solved, we can thus proceed with the fifth step. In order to draw a Legendrian handlebody of $X_{a,b}$, it suffices to gather the above information and apply Proposition \ref{prop:stacking} by using an $A_{a+b-1}$--linear basis as the set $\L$ in the first step of Recipe \ref{dictionary}: we could proceed in this manner, and we refer to Example \ref{ex:X23 linear} below for the results. Instead, we shall use a non--linear basis as the set $\L$ given by the set of exact Lagrangian spheres $\L=\{R_1,\ldots,R_{a+b-1}\}$; the reason being that the vanishing cycles $V_i$ have a simpler expression in terms of these Lagrangian spheres and thus can be managed more efficiently in order to conclude Theorem \ref{thm:Xab}.

We draw the subcritical handles associated to the Lagrangian spheres $\L$ starting at the first ray $R_1$ and going straight up to the last ray $R_{a+b-1}$: these correspond to the Legendrian skeleton depicted in Figure \ref{fig:Xab skel}. In detail, Figure \ref{fig:Xab skel} is the Legendrian lift of a Lagrangian skeleton of the Weinstein fiber $(F_\pi,\la,\p)\cong(A_{a+b-1},\la_\st,\p_\st)$, note though that this Legendrian skeleton is not a tree plumbing of spheres since these all intersect in one point; this is however a valid choice for a set $\L = \{R_i\}$ and we proceed with it.

Proposition \ref{prop:stacking} applies with this choice of Legendrian skeleton and we obtain the Legendrian front for the Legendrian attaching link depicted in Figure \ref{fig:X25front}. Note that with the choice of $\L$ and these vanishing cycles, the Legendrian front becomes $S^{n-2}$--rotationally symmetric along the central axis and thus it suffices to consider a 2--dimensional front slice. \\

\begin{figure}[h!]
  \centering
  \includegraphics[scale=0.55]{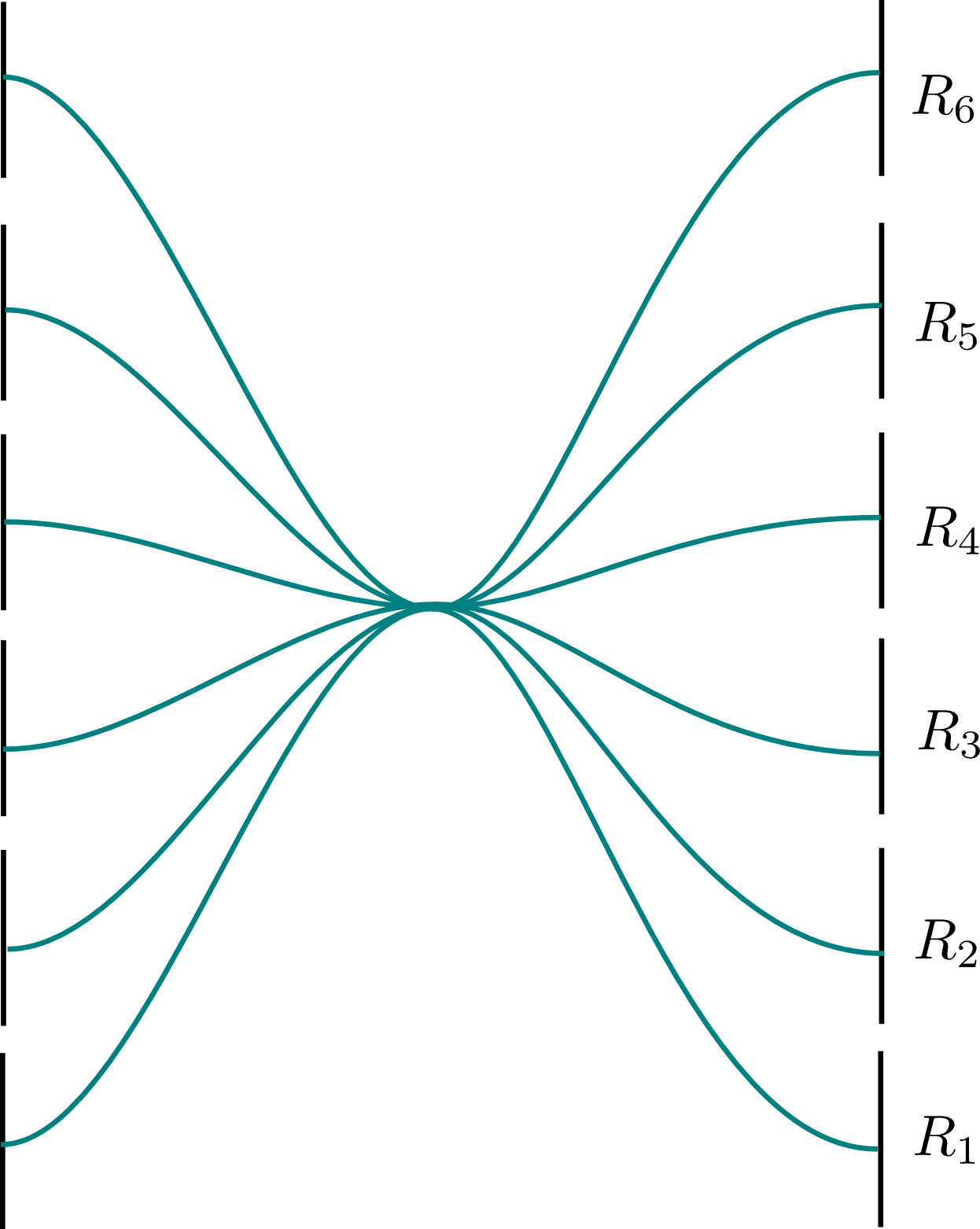}
  \caption{The Legendrian skeleton defined by $\L$.}
\label{fig:Xab skel}
\end{figure}

\begin{figure}[h!]
  \centering
  \includegraphics[scale=0.55]{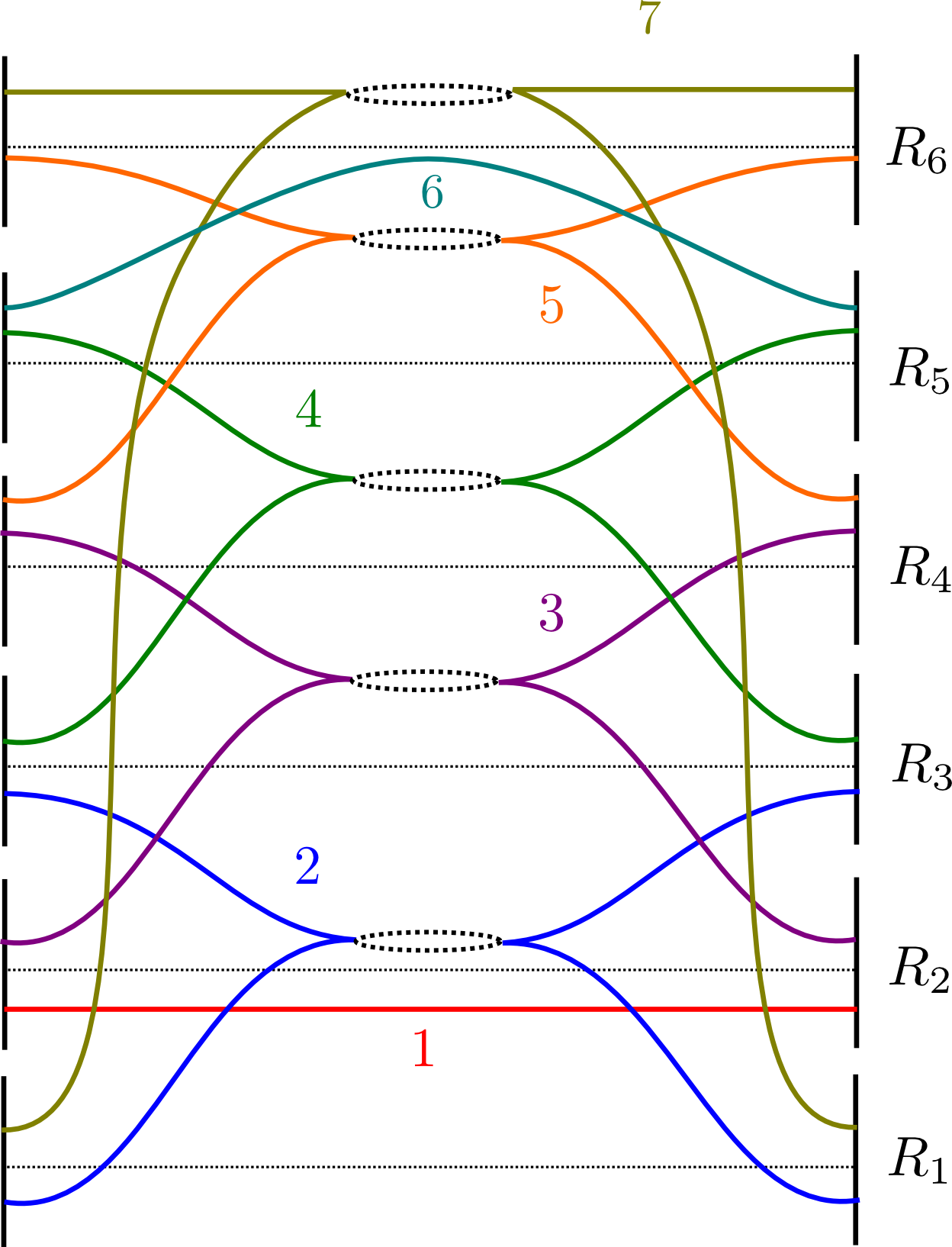}
  \caption{The Legendrian handlebody description of $X_{a,b}$ for $(a,b) = (2,5)$.}
\label{fig:X25front}
\end{figure}

\begin{figure}[h!]
  \centering
  \includegraphics[scale=0.45]{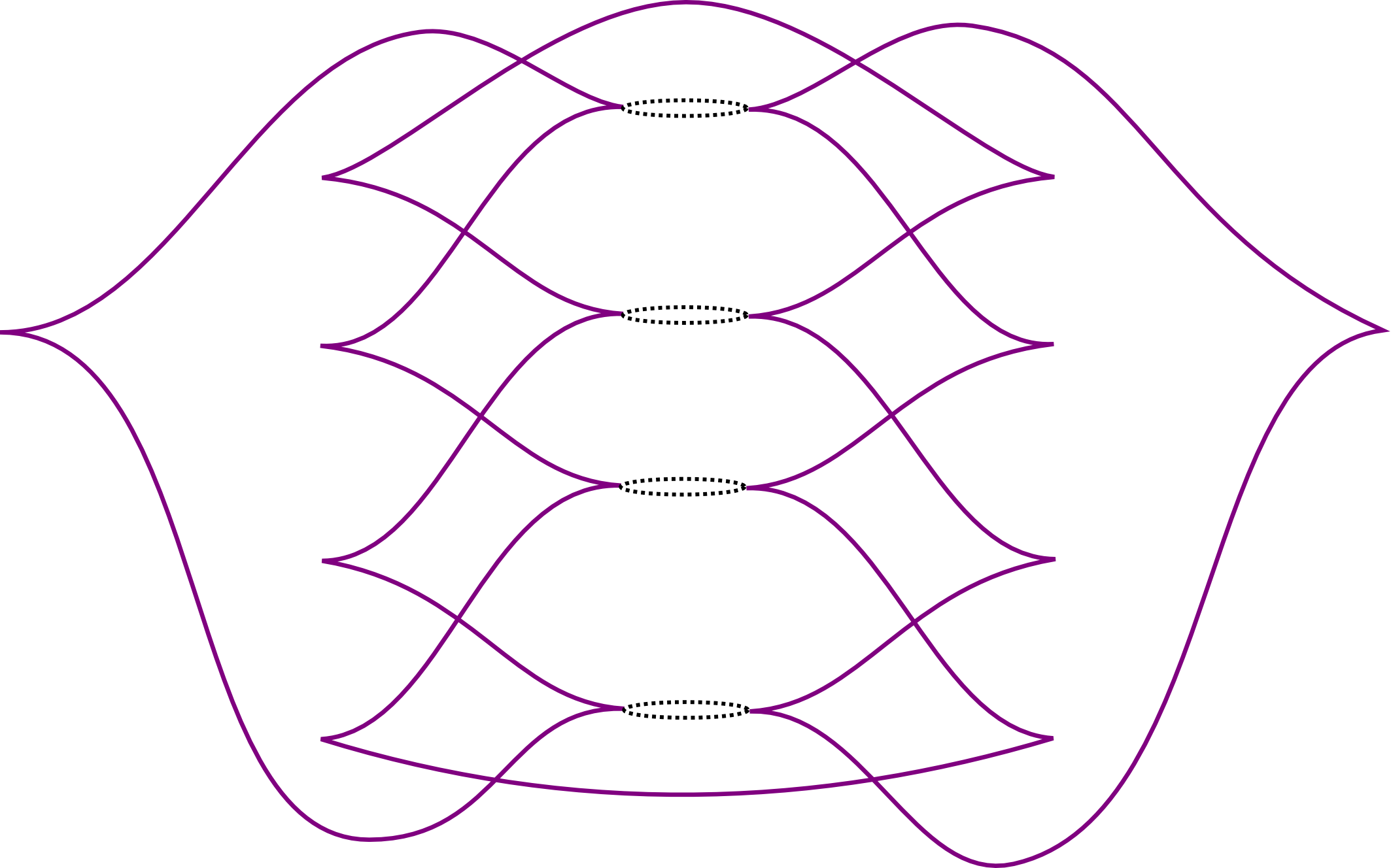}
  \caption{Figure \ref{fig:X25front} after applying handle cancellations.}
\label{fig:X25knot}
\end{figure}

Let us now simplify this Legendrian front, while reading the following lines it might be a healthy practice to keep Figures \ref{fig:X25front} and \ref{fig:X35front} in mind. First, notice that the coprimality of the integers $a$ and $b$ allows us to cancel each of the subcritical handles with some Legendrian handle. Since each subcritical handle intersects only two Legendrian strands, we can use the cancellation move in Proposition \ref{prop:handle cancel} and cancellation corresponds diagramatically to collapsing all subcritical handles to Legendrian cusps, see Figures \ref{fig:X25knot} and \ref{fig:X35front}. Second, the only two vanishing cycles in our collection $V=\{V_1,\ldots,V_{a+b}\}$ that will have fronts intersecting the central axis are $V_1$ and $V_{b+1}$ which, along with the global $S^{n-2}$--symmetry of this Legendrian front, shows that the resulting Legendrian attaching sphere is a cusp $S^{n-2}$--spinning of a Legendrian knot $\Lambda\sse(S^3,\xi_\st)$, which is defined to be the right hand side of Figure \ref{fig:X25knot}. In order to conclude Theorem \ref{thm:Xab} we must exhibit the Legendrian isotopy $\Lambda \cong \Lambda_{a,b}$.

\begin{figure}[h!]
  \centering
  \includegraphics[scale=0.65]{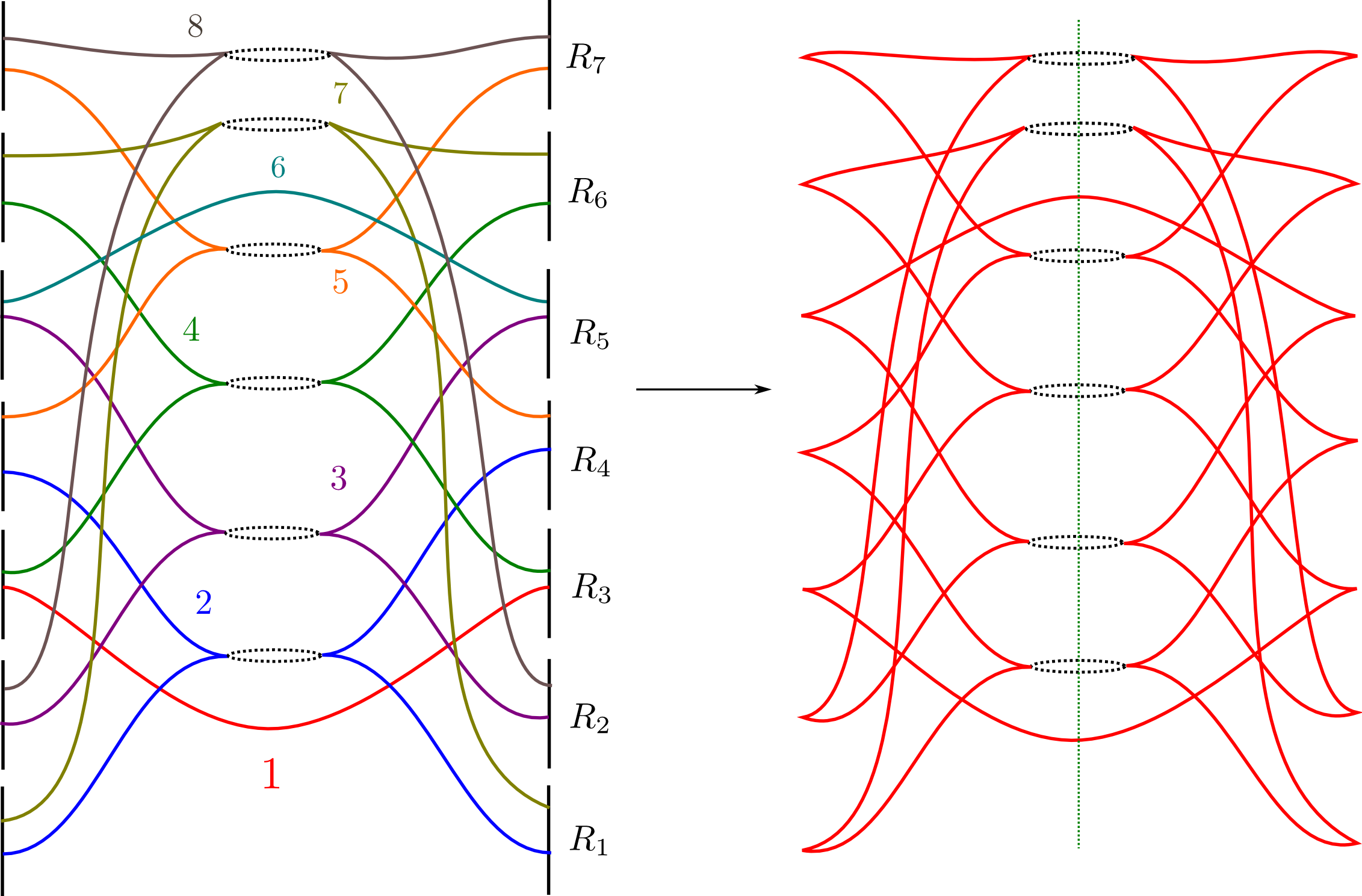}
  \caption{The Legendrian handlebody description of $X_{a,b}$ for $(a,b) = (3,5)$, before and after handle cancellations.}
\label{fig:X35front}
\end{figure}

\begin{figure}[h!]
  \centering
  \includegraphics[scale=0.65]{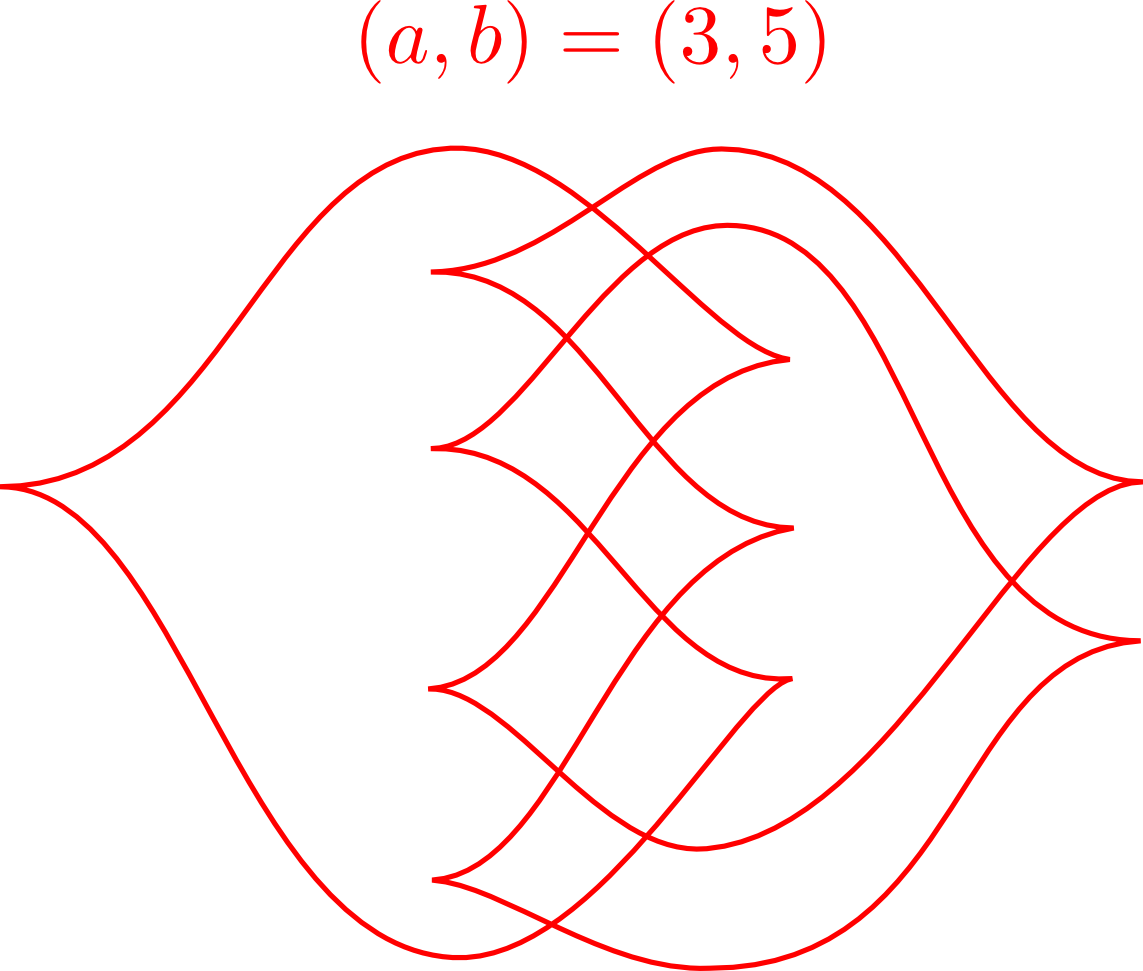}
  \caption{The left handed Legendrian torus knot $\Lambda_{a,b}$.}
\label{fig:X35knot}
\end{figure}

To construct this Legendrian isotopy, consider the Legendrian cusps coming from the rays $R_i$ with $i\leq a-1$, i.e.~ the cusps below the portion of $\Lambda$ coming from the Legendrian lift of $V_1$, and Legendrian isotope them to the right. The Legendrian strands on the upper branch of these cusps, that is, the steepest strands in Figures \ref{fig:X25front} and \ref{fig:X35front}, can also be Legendrian isotoped to the right, thus disjoining them from the main body of the Legendrian front. Then the left cusps on these strands can be cancelled with the cusps corresponding to the rays $R_i$ with $b+1\leq i$ by using Reidemeister I moves. The remaining Legendrian front is that of the standard Legendrian $(a,-b)$ braid together with a Legendrian ribbon on the right connecting $(a-1)$ strands on top to the bottom, and a single strand to the left of the braid corresponding to the rotational axis; the results are depicted in Figures \ref{fig:X25knot} and \ref{fig:X35knot} in the cases $(a,b)=(2,5)$ and $(3,5)$. This resulting Legendrian front is a Legendrian front for the Legendrian knot $\Lambda_{a,b}$, which proves Theorem \ref{thm:Xab}. 
\end{proof}

\begin{remark}
Note that the rotation class $r(\Lambda_{a,b}\#\overline{\Lambda_{a,b}})=0$ vanishes, as it necessarily should. Indeed, on the one hand the first Chern class $c_1(X^2_{a,b})$ is Poincar\'e dual to $r(\Lambda_{a,b}\#\overline{\Lambda_{a,b}})$ times the cocore of the handle, and on the other the tangent bundle of an affine complete intersection is (even holomorphically) trivial, and thus the first Chern class must satisfy $c_1(X^2_{a,b})=0$.

It is also relevant to remark that Legendrian torus knots in $(S^3,\xi_0)$ are simple Legendrian knots \cite[Theorem 4.3]{EH}, and the rotation number of a maximal Thurston--Bennequin Legendrian $(a,-b)$--torus knot belongs to the finite set
$$\left\{|b|-|a|-2ak:\quad k\in\left[0,\frac{|b|-|a|}{|a|}\right)\cap\N\right\}.$$
The Legendrian knot $\Lambda_{a,b}$ appearing in the statement of Theorem \ref{thm:Xab} corresponds to the unique maximal Thurston--Bennequin Legendrian $(a,-b)$--torus knot with rotation number equal to $|b|-|a|$.$\hfill\Box$
\end{remark}

In the proof of Theorem \ref{thm:Xab} we have used a non-generic basis $\L$ for the middle homology of the Weinstein fiber $(F_\pi,\la)$, the statement can also be proven using a standard $A_{a+b-1}$--linear basis as the following example shows; this choice however produces a Legendrian front that lacks the global $S^{n-2}$--symmetry.

\begin{figure}[h!]
  \centering
  \includegraphics[scale=0.5]{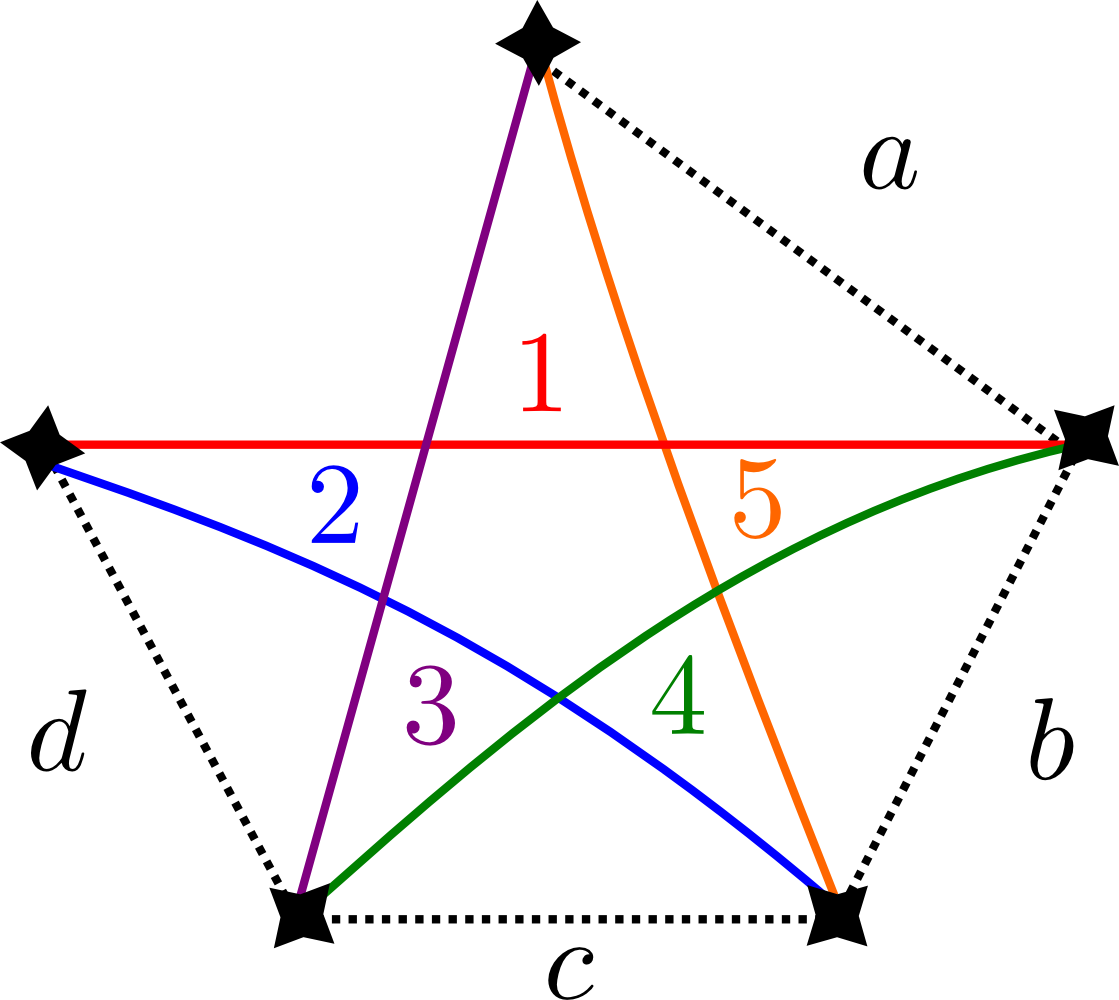}
  \caption{The $2+3=5$ vanishing cycles for $X_{2,3}$ and an $A_4$--basis $\{a,b,c,d\}$.}
\label{fig:X23LF}
\end{figure}

\begin{example} \label{ex:X23 linear} Let us obtain the Legendrian handlebody for the Weinstein manifold
$$X^2_{2,3}=\{(x,y,z):	x^2y^3+z^2=1\}\sse\C^3$$
using the linear basis $A_4$--basis $\L=\{a,b,c,d\}$ indicated in Figure \ref{fig:X23LF}, which represents the Lefschetz bifibration $\rho:F_\pi\lr\C$ given by $\rho(x,z)=x$. In this case it is a two--fold branch covering of the plane, given by the two choices of square root for the $z$--coordinate, with five branch points, whose critical values are depicted in the four--pointed stars. Thus the vanishing cycles of $\rho$ coincide with the zero section $Z \sse T^*S^{n-2} \cong F_\rho$ and Figure \ref{fig:X23LF} contains the necessary information to proceed with Recipe \ref{dictionary}.

\begin{figure}[h!]
  \centering
  \includegraphics[scale=0.55]{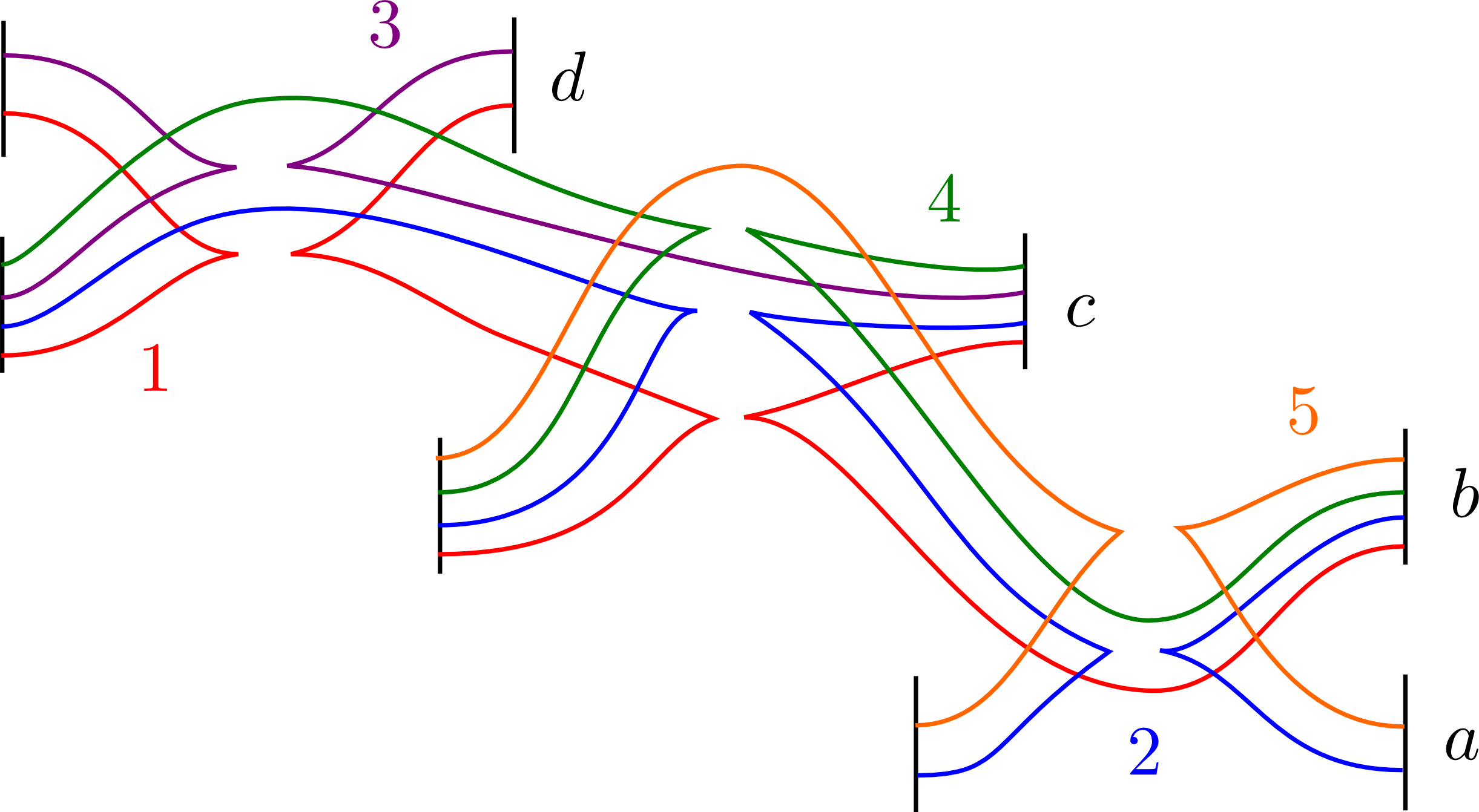}
  \caption{Legendrian handlebody for $X_{2,3}$.}
\label{fig:X23Front1}
\end{figure}

The five vanishing cycles expressed in this $A_4$--basis are
\begin{align*}
V_1=\tau_d\tau_c(b),\quad V_2=\tau_c\tau_b(a),\quad V_3=\tau_d(c),\\
V_4=\tau_c(b),\quad V_5=\tau_b(a).\\
\end{align*}
Then we proceed with the fifth step of Recipe \ref{dictionary} by using Proposition \ref{prop:stacking} in order to produce the front projection of the Legendrian attaching link: the Legendrian lifts of the five vanishing cycles are depicted in Figure \ref{fig:X23Front1}, which can be readily simplified to Figure \ref{fig:X23Front2}. The vertical dotted segment in Figure \ref{fig:X23Front2} indicates the connected sum decomposition of the Legendrian knot and it can be traced from the original front in Figure \ref{fig:X23Front1} at the vertical arc connecting the basis elements $b$ and $c$. From this simplified Legendrian front we also obtain the statement of Theorem \ref{thm:Xab} in this case.\hfill$\Box$

\begin{figure}[h!]
  \centering
  \includegraphics[scale=0.45]{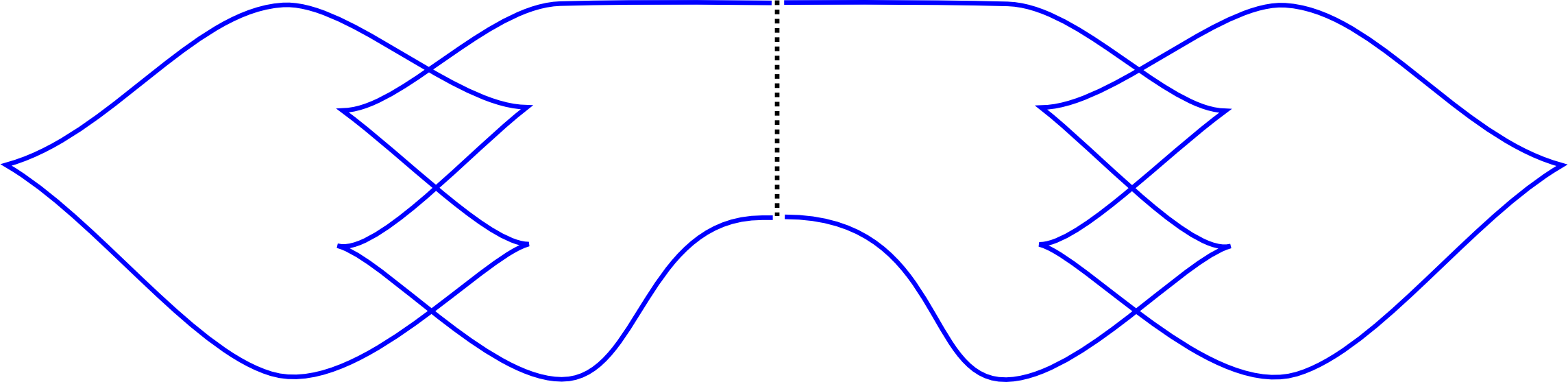}
  \caption{Simplification of the Legendrian Front in Figure \ref{fig:X23Front1}.}
\label{fig:X23Front2}
\end{figure}
\end{example}

This concludes the proof of Theorem \ref{thm:Xab}, Corollary \ref{cor:Xab} and our study of the Weinstein manifolds $(X^n_{a,b},\la,\p)$. Now proceed to the study of the symplectic topology of an exotic affine six--manifold: interestingly, the application of Recipe \ref{dictionary} differs in two aspects with respect to the calculations in Subsection \ref{ssec:mirror1} and the proof of Theorem \ref{thm:Xab}, but it still succeeds in producing the desired Legendrian handlebody.

\subsection{The Koras--Russell Cubic}\label{ssec:kr}

The characterization of the algebraic isomorphism type of the affine space $\C^n$ is a core problem in affine algebraic geometry. In the seminal article \cite{Ram}, C.P.~Ramanujam proved that a smooth contractible algebraic surface simply--connected at infinity is algebraically isomorphic to $\C^2$; this is however no longer true in higher dimensions. A beloved counter--example is the Koras--Russell cubic, a smooth affine 3--fold diffeomorphic to $\C^3$ but not algebraically isomorphic to it. It is defined by the following equation:

$$\SC:=\{(x,y,z,w):x+x^2y+w^3+z^2=0\}\sse\C^4.$$

The Koras--Russell cubic belongs to a family of smooth contractible 3--folds introduced in \cite{KR}, which were proven to be algebraically exotic $\C^3$ in \cite{KML} by finding non--constant regular functions annihilated by every locally nilpotent derivation. To our knowledge, it was not known whether $(\SC,\la,\p)$ is symplectomorphic to the standard affine 3--space $(\C^3,\la_\st,\p_\st)$, in this section we prove this.

\begin{thm}\label{thm:kr}
The Koras--Russell cubic $(\SC,\la,\p)$ is Weinstein equivalent to $(\C^3,\la_\st,\p_\st)$.
\end{thm}

\begin{proof} First, we consider the Lefschetz fibration
$$\pi:\SC\lr\C,\quad\pi(x,y,w,z)=-2x+y+2w,$$
which has four critical points, with different critical values; let us denote their four vanishing cycles by $V=\{V_1,V_2,V_3,V_4\}$. Since the defining equation of the Koras--Russell $\SC$ has the $z^2$ term and the Lefschetz fibration $\pi:\SC\lr\C$ does not interact with the $z$--coordinate, the Legendrian handlebody for $\SC$ can be obtained by exhibiting a Legendrian handlebody for the Weinstein 4--fold
$$\SC'=\{(x,y,w):x+x^2y+w^3=0\}\sse\C^3$$
and adding the corresponding local symmetries discussed in Subsection \ref{ssec:highD}. In consequence, we continue our work with the Weinstein manifold $\SC'$ and the Lefschetz fibration $\pi:\SC'\lr\C$ introduced above. The regular fiber
$$(F_\pi,\la)=\{x+x^2y+(x-y/2)^3=0\}\sse\C^2$$
is a smooth affine cubic curve which has the Weinstein type of the $D_4$ Milnor fiber.\\
Second, we use the linear projection
$$\rho:(F_\pi,\la)\lr\C,\quad\rho(x,y)=x+3.4\cdot y,$$
in order to proceed with the third step of Recipe \ref{dictionary}, i.e.~ describing the vanishing cycles of the fibration $\pi$ as matching paths for the Lefschetz bifibration $\rho$.

The regular bifiber $(F_\rho,\la,\p)$ is the standard $A_2$ Milnor fiber and we can use Subsection \ref{ssec:AD_bifiber} to conclude that the linear projection $\rho$ has six non--degenerate critical points with different critical values. The first exact Lagrangian circle $S_1 \sse F_\rho$ is the vanishing cycle for three of these critical values, whereas the exact Lagrangian, $S_2 \sse F_\rho$, is the vanishing cycle for the remaining three critical values. See Figure \ref{fig:KR_LFSquare}(A), which depicts the three critical points whose vanishing cycles are $S_1$ with ochre pentagons, and the critical points whose vanishing cycles are $S_2$ with black triangles.

\begin{figure}[h!]
\centering
\begin{subfigure}{.5\textwidth}
  \centering
  \includegraphics[scale=0.3]{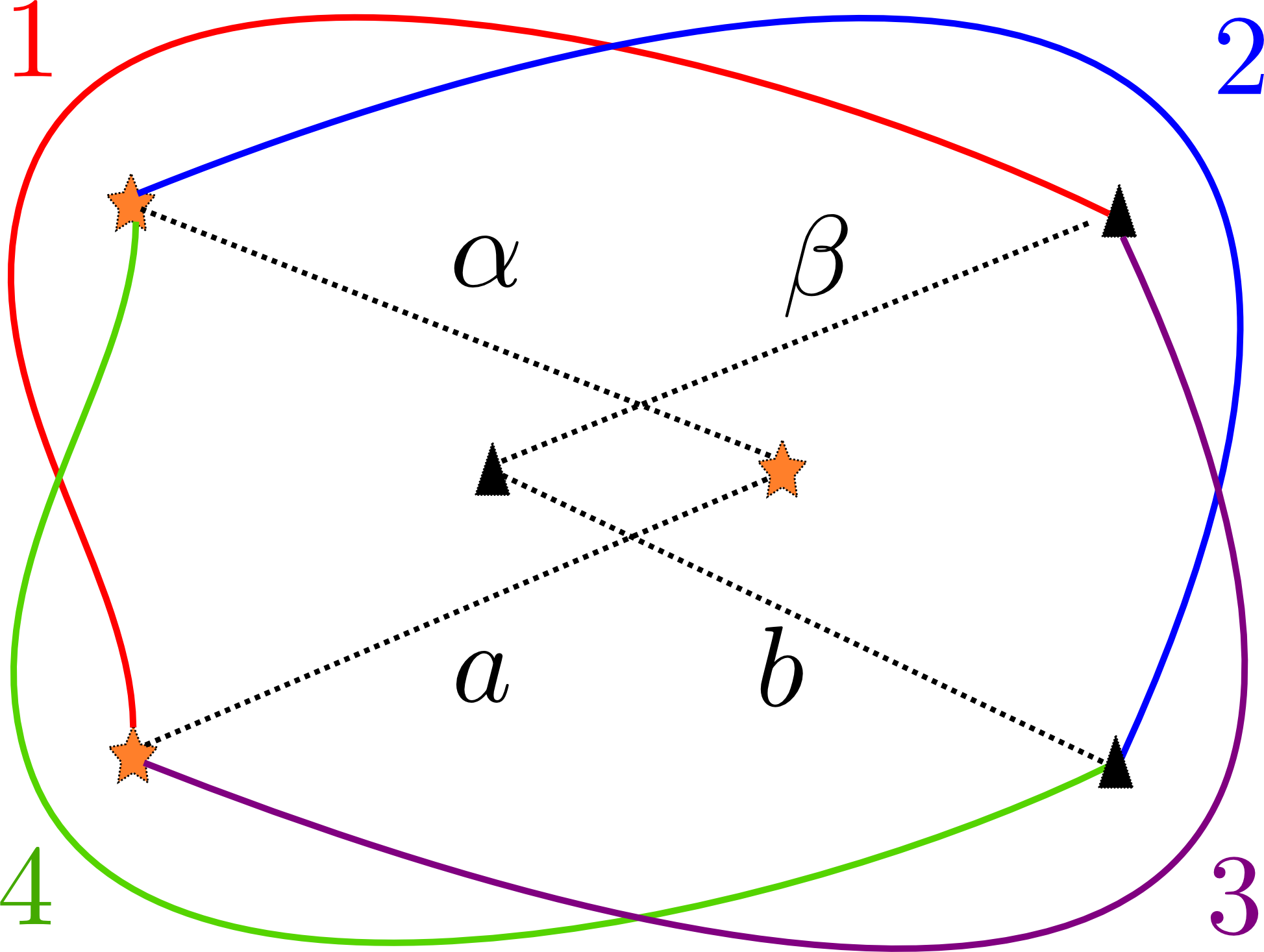}
  \caption{Matching paths for the pair $(\pi,\rho)$.}
\end{subfigure}%
\begin{subfigure}{.5\textwidth}
  \centering
  \includegraphics[scale=0.3]{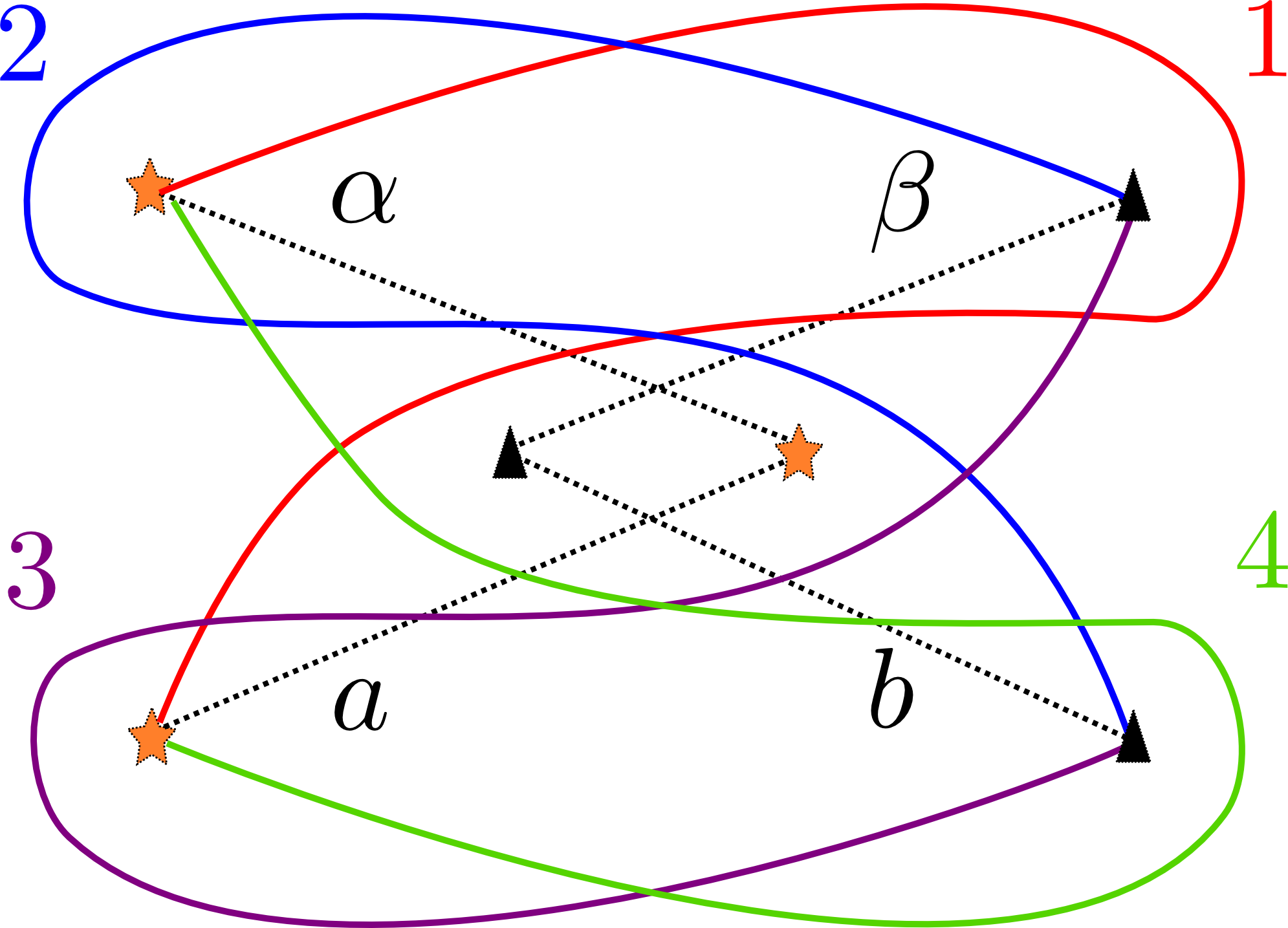}
  \caption{Configuration after four V--moves.}
\end{subfigure}
\caption{Vanishing cycles for the Koras--Russell cubic.}
\label{fig:KR_LFSquare}
\end{figure}

Figure \ref{fig:KR_LFSquare}(A) also exhibits the four vanishing cycles for the initial Lefschetz fibration $\pi$ as the four coloured matching paths for this Lefschetz bifibration $\rho$. It is our next task to describe these four matching paths as words in Dehn twists for a Lagrangian $D_4$--basis of the Weinstein fiber $(F_\pi,\la)$. Instead of considering a $D_4$--basis from the start, we first consider the four spheres $\{a,b,\alpha,\beta\}$ described by the homonymous matching paths in Figure \ref{fig:KR_LFSquare}(A) and perform one V--move for each vanishing cycle, as described in Subsection \ref{ssec:lef}. This results in the configuration illustrated in Figure \ref{fig:KR_LFSquare}(B).

\begin{figure}[h!]
\includegraphics[scale=0.3]{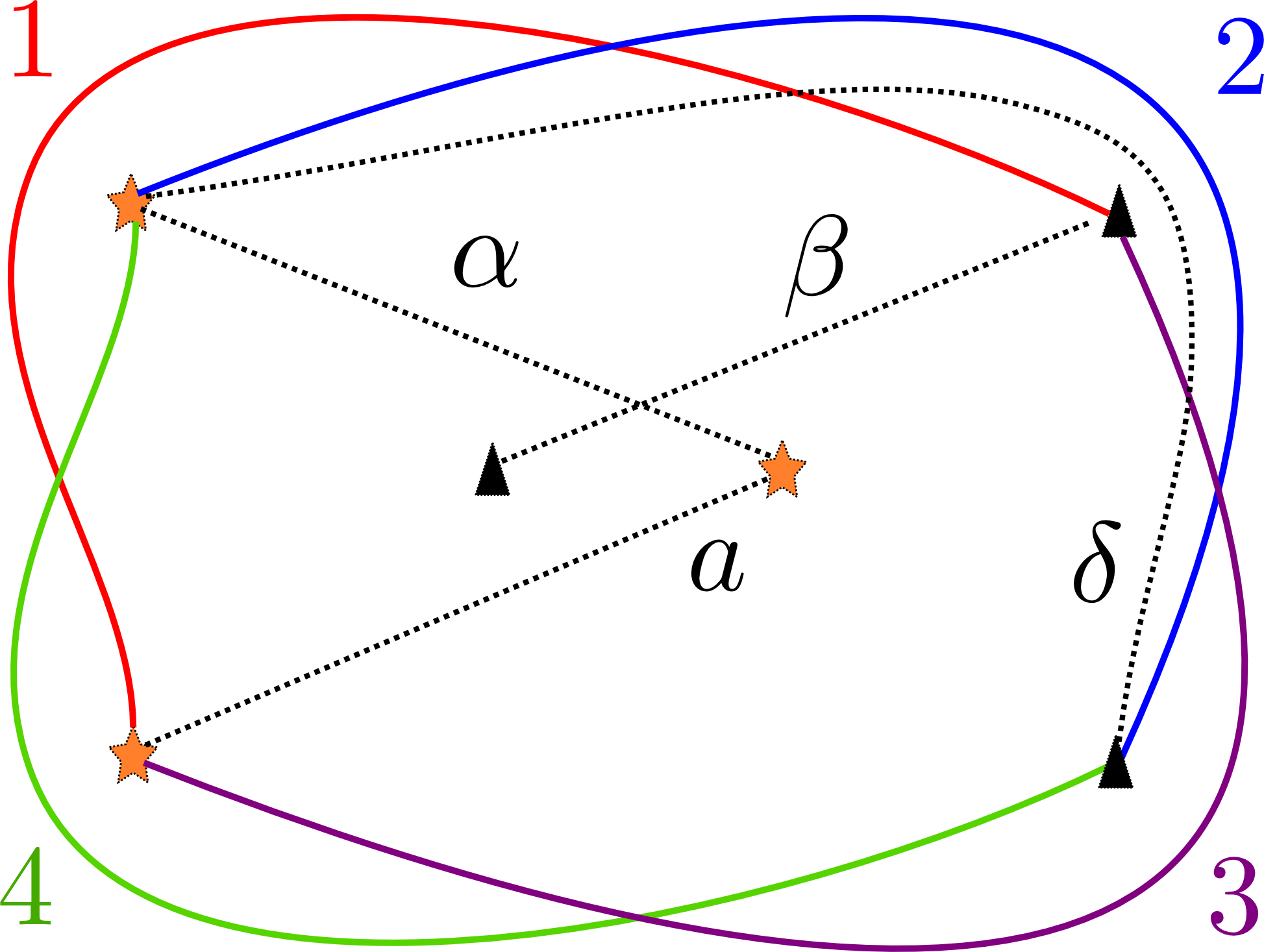}
\caption{$D_4$--basis for $F_\pi$.}\label{fig:KR_LFD4}
\end{figure}

We can describe each matching path in terms of the four spheres $\{a,\alpha;b,\beta\}$ as follows:
$$V_1=\tau_\beta^{-1}\tau_a(\alpha),\quad V_2=\tau_\alpha\tau_\beta(b)$$
$$V_3=\tau_a^{-1}\tau_\beta(b),\quad V_4=\tau_b\tau_a(\alpha)$$
This allows us to readily express the four matching paths in the $D_4$--basis. Indeed, the four Lagrangian spheres $\{a,\alpha,\delta;\beta\}$ depicted in Figure \ref{fig:KR_LFD4} form a $D_4$ intersection pattern, and they relate to the previous Lagrangian spheres via the equality $b=\tau_\beta^{-1}\tau_\alpha^{-1}(\delta)$. Therefore the vanishing cycles can be expressed as
$$V_1=\tau_\beta^{-1}\tau_a(\alpha),\quad V_2=\tau_\alpha\tau_\beta(b)=\delta$$
$$V_3=\tau_a^{-1}\tau_\beta(b)=\tau_a^{-1}\tau_\delta(\alpha),\quad V_4=\tau_b\tau_a(\alpha)=\tau_\alpha^{-1}\tau_a^{-1}\tau_\beta^{-1}\tau_\delta(\alpha)$$
Proposition \ref{prop:stacking} can now be applied to obtain a Legendrian front presentation for the Koras--Russell cubic $\SC\sse\C^4[x,y,z,w]$ which is depicted in Figure \ref{fig:KR_FrontInitial}.

\begin{figure}[h!]
\includegraphics[scale=0.5]{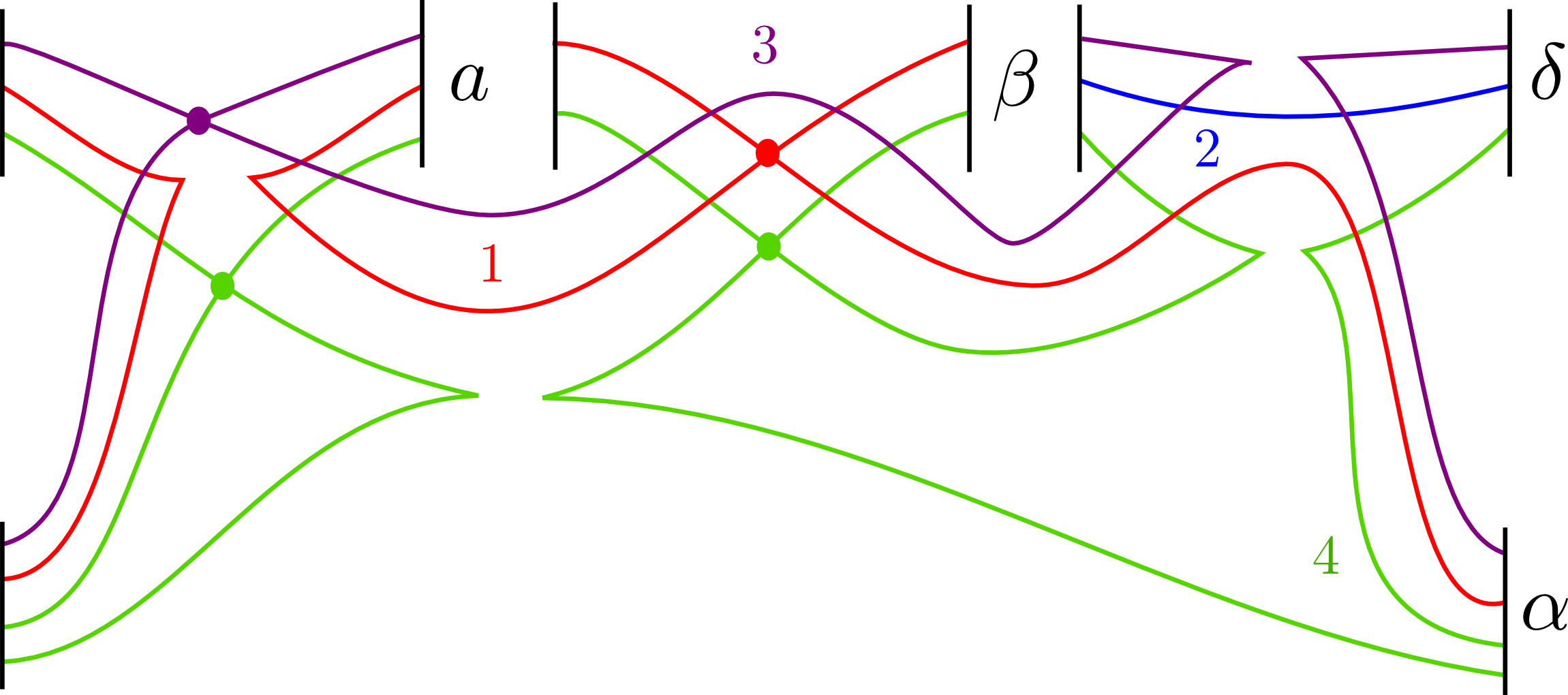}
\caption{Legendrian front for the Koras--Russell cubic}\label{fig:KR_FrontInitial}
\end{figure}

\begin{figure}[h!]
\includegraphics[scale=0.75]{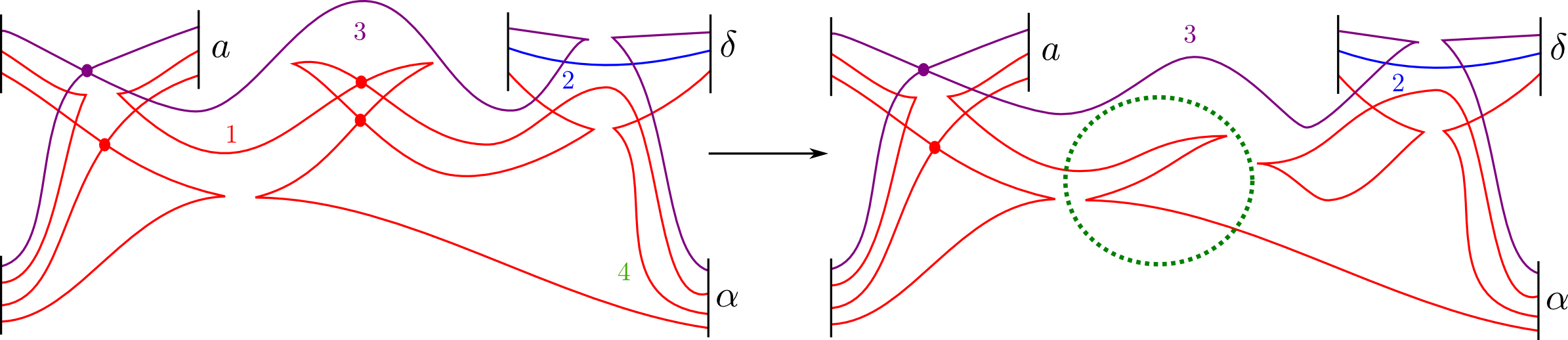}
\caption{A simplification of the Koras--Russell cubic, showing it is flexible.}\label{fig:KR_FrontSimplify}
\end{figure}

In order to conclude Theorem \ref{thm:kr}, it suffices to show that the Legendrian handlebody in Figure \ref{fig:KR_FrontInitial} is flexible: that is, we need to exhibit a loose chart for each component of the Legendrian link in the complement of the remaining components. This is achieved by simplifying the Legendrian front with the moves introduced in Subsection \ref{ssec:highD} and then applying the results in Subsection \ref{ssec:loose}. First, we cancel the subcritical handle $\beta$ with either of the Legendrian lifts of the vanishing cycle $V_1$, and denote by $V_{1\cup 4}$ the resulting new component. Then apply the Reidemeister move that removes the two cone singularities, depicted in Figure \ref{fig:highDReid}, to this Legendrian component $V_{1\cup 4}$; this simplification is shown in Figure \ref{fig:KR_FrontSimplify}. We can then see that the Legendrian component $V_{1\cup 4}$ is a loose Legendrian since a zig--zag visibly appears in the diagram and we can apply Proposition \ref{prop:loose slice}. The remaining two components, corresponding to the Legendrian lifts of the vanishing cycles $V_2$ and $V_3$, are in cancelling position with the subcritical handles $\delta$ and $a$ respectively. After performing these cancellations we are left with only the subcritical handle $\alpha$, and a Legendrian with one component, which is still loose. Since this Legendrian is loose and it cancels with the subcritical handle $\alpha$ up to smooth isotopy, Theorem \ref{thm: c0 loose} implies that it also cancels symplectically. 
\end{proof}

\begin{remark}
Consider the polynomial $p(x,y,z,w)=x+x^2y+z^2+w^3$, then the Koras--Russell cubic is the fiber $\SC=p^{-1}(0)$. The regular fibers $p^{-1}(c)$, $c\neq0$, are all algebraically isomorphic to the affine variety $p^{-1}(1)$ by rescaling, and thus Stein deformation equivalent to the smooth hypersurface $\{x^2y+z^2+w^3=1\}\sse\C^4$, which is the example we presented at the beginning of the article. Note that these fibers are no longer diffeomorphic to $\C^3$ since their Euler characteristic is three.

\begin{figure}[h!]
\centering
\begin{subfigure}{.5\textwidth}
  \centering
  \includegraphics[scale=0.3]{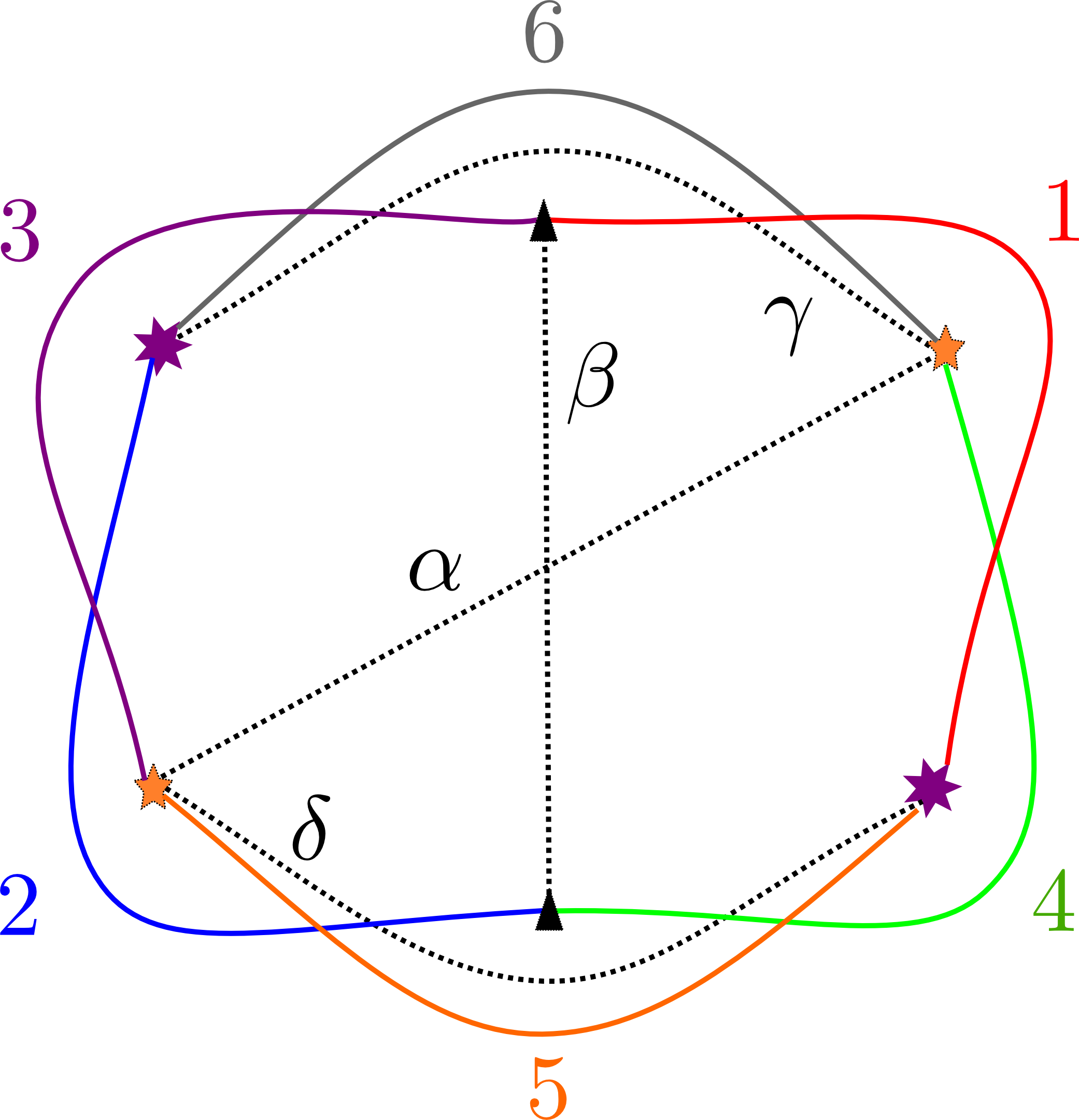}
  \caption{Vanishing cycles for a $D_4$--bifibration.}
\end{subfigure}%
\begin{subfigure}{.5\textwidth}
  \centering
  \includegraphics[scale=0.4]{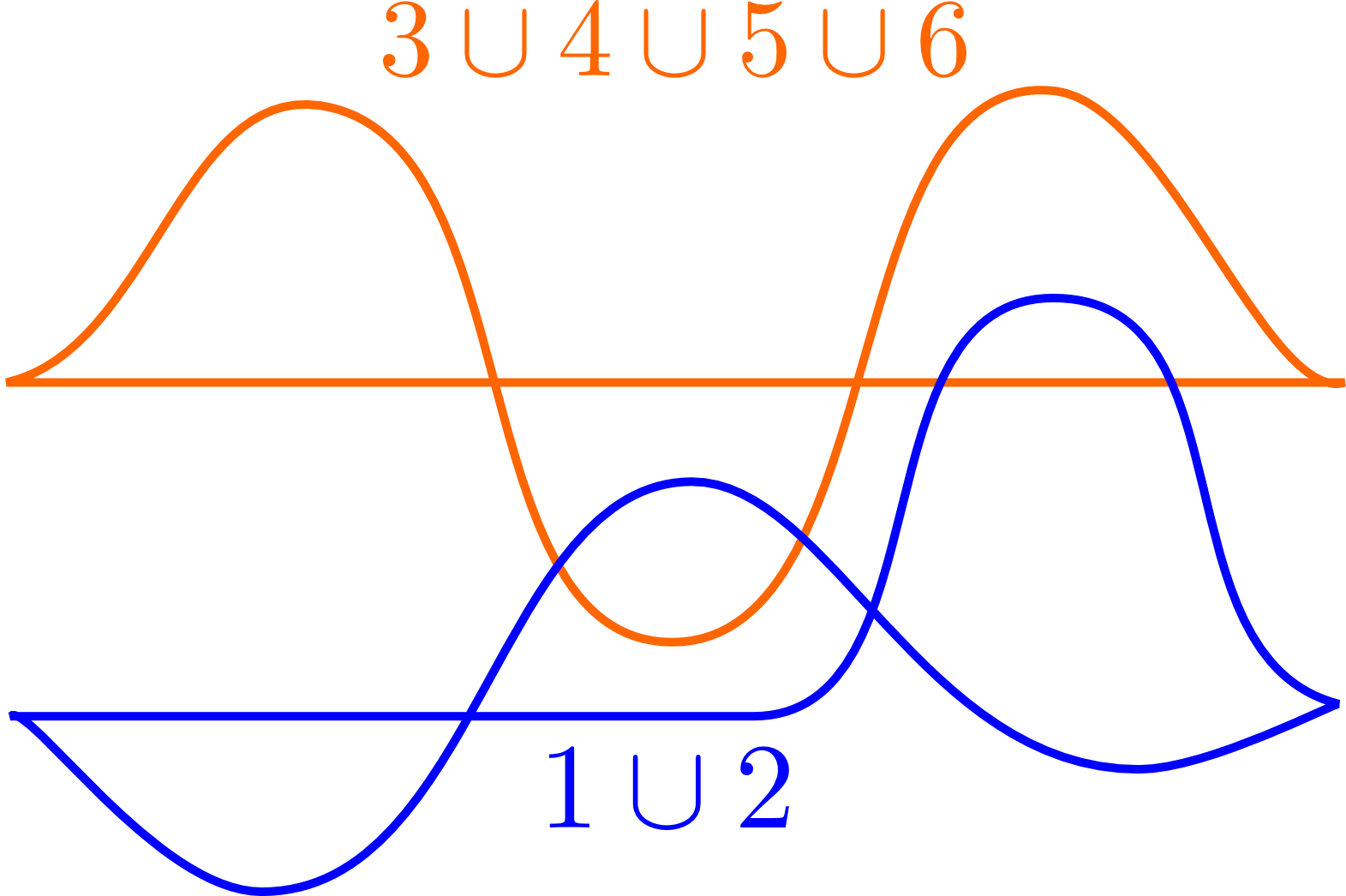}
  \caption{Resulting Legendrian link after cancellation: the initial six 3--handles coming from the vanishing cycles cancel the initial four 2--handles and result in this two--component link.}
\end{subfigure}
\caption{The Stein manifold $\{xy^2+z^2+w^3=1\}$.}
\label{fig:genericKR}
\end{figure}

Regardless, the proof for Theorem \ref{thm:kr} can be modified in order to show that these fibers are flexible as well: at this point the reader will hopefully be able to readily verify this statement starting from the data in Figure \ref{fig:genericKR}(A), which uses the $D_4$--basis $(\delta,\alpha,\gamma;\beta)$ and where the six vanishing cycles $\{V_1,V_2,V_3,V_4,V_5,V_6\}$ read:
$$1=\tau_\beta\tau_\alpha(\delta),\quad 3=\tau_\alpha(\gamma),\quad 5=\delta,$$
$$2=\tau_\beta\tau_\alpha(\gamma),\quad 4=\tau_\alpha(\delta),\quad 6=\gamma.$$
The six critical values in this Weinstein $(D^4_4,A^2_2)$--bifibration described in Figure \ref{fig:genericKR}(A) have vanishing cycles $S_1$, if they are depicted with the ochre star, $S_2$, if depicted with the black triangle, and $\tau_{S_2}(S_1)$, if depicted with the purple star. The resulting Stein flexible $6$--manifold is described by the Legendrian surface front in Figure \ref{fig:genericKR}(B), which is indeed a loose link.$\hfill\Box$
\end{remark}

This concludes our proof that the Koras--Russell cubic is Stein deformation equivalent to $(\C^3,\la_\st,\p_\st)$. In the following subsection we continue our proof of Theorem \ref{thm:mirror} by studying the Weinstein manifolds $M^n_b$ featuring in Theorem \ref{thm:torusIntro} and discussing the connection between Legendrian handlebodies and homological mirror symmetry.

\subsection{Theorem \ref{thm:mirror}: Part II}\label{ssec:mirror2} In this subsection we prove Theorem \ref{thm:torusIntro} and Theorem \ref{thm:mirror}. The former will be an application of Recipe \ref{dictionary} and the result in Subsection \ref{ssec:lagcob}, whereas the latter requires a discussion on symplectic field theory and mirror symmetry. 

Let us first proceed with this discussion and briefly give a simplified view of mirror symmetry, which aims at explaining why symplectic field theoretic invariants of Legendrian submanifolds can be useful for homological mirror symmetry.

\begin{remark}
The content in this subsection related to Legendrian handlebodies is rigorously based in the results presented thus far, in contrast we remark that the ideas we propose in relation to mirror symmetry are only interpretive and therefore the computations below should be seen as insinuative experiments rather than rigorous verifications.\hfill$\Box$
\end{remark}

The mirror symmetry functor maps
$$D^b(\wfuk(X))\lr D^b(\coh(\check X)),$$
which we can be describe in a naive form as follows. Both categories have natural operations, including taking direct sums, cones, and grading shifts, which allow us to restrict ourselves to sets which generate the category under these operations.

Let us focus on the symplectic side, where we can use the fact that the Liouville manifolds we consider are Weinstein manifolds, and furthermore Weinstein manifolds constructed with a unique critical Weinstein handle. In the study of the Weinstein case, S.~Ganatra and M.~Maydanskiy show in the Appendix of the article \cite{BEE}, that the category $D^b(\wfuk(X))$ is generated by the Lagrangian disks which are the cocores of the critical index handles. Based on the work in the main body of the articles \cite{BEE,BEE2}, it is expected that the chain level isomorphism
$$\mbox{LCC}_*(\Lambda) \cong \mbox{WC}_*(L)$$
holds in this case, where $\Lambda$ denotes the Legendrian attaching sphere of a critical Weinstein handle, $L$ denotes the Lagrangian cocore of this handle, $\mbox{LCC}_*$ denotes the Legendrian contact homology differential graded algebra, $\mbox{WC}_*$ denotes the wrapped Floer homology $A_\infty$--algebra, and the isomorphism is an $A_\infty$ quasi--isomorphism; here the product on the Legendrian chains $\mbox{LCC}_*$ is formal concatenation and all higher products vanish. This exact statement is not quite stated in the articles \cite{BEE,BEE2}, but it can be proven using similar ingredients than those featuring in these articles and it is much in line with the numerous isomorphisms appearing there, particularly \cite[Theorem 5.8]{BEE}.

In consequence, if we restrict to the case where only one Weinstein critical handle is attached, the mirror symmetry functor should take chain complexes of modules over the algebra $\mbox{LCC}_*(\Lambda)$ to chain complexes of modules over a certain ring $R$. The best case scenario here would be that the Legendrian homology $\mbox{LCH}_*(\Lambda)$ is a commutative module, supported in grading zero, and actually isomorphic to the ring $R$, rather than just Morita equivalent. In case we are in this situation, we can simply write $\check X = \op{Spec}(\mbox{LCH}_0(\Lambda))$. This has been the case for the Legendrian handlebody obtained in Subsection \ref{ssec:mirror1} and it will also hold in the example presented below.

\begin{remark}
These assumptions do not necessarily hold in general but the discussion above shows that Recipe \ref{dictionary}, combined with computations of the Legendrian homology $\mbox{LCH}_*$ of the resultin Legendrian handlebody and the Legendrian surgery isomorphisms from \cite{BEE}, gives a powerful method to compute the wrapped Fukaya category $\wfuk(X)$ for Stein manifolds $(X,\la_\st,\p_\st)$ presented as affine varieties.

To establish more generic mirror symmetry results we should calculate $D(\wfuk(X))$ from this information and construct an algebra $R$ which is Morita equivalent to it.  In short, these discussion highlights the fact that Recipe \ref{dictionary} can be used as a tool to calculate wrapped Fukaya categories, and the reason we focus on examples related to mirror symmetry is that these manifolds are well--studied in the literature and this might help readers to both appreciate and start using the Legendrian viewpoint.\hfill$\Box$
\end{remark}

This concludes the discussion on the interaction between Legendrian invariants and homological mirror symmetry, and we now proceed to prove Theorem \ref{thm:torusIntro} and Theorem \ref{thm:mirror}.

Let us apply Recipe \ref{dictionary} to the Weinstein $4$--fold
$$(X,\la,\p)=\{(x,y,z):x(xy^2-1)+z^2=0\}\sse(\C^3,\la_\st,\p_\st)$$
featuring in the statement of Theorem \ref{thm:mirror}. In comparison to the previous computation in Subsection \ref{ssec:mirror1} above, the resulting Legendrian link in the contact boundary $\#^3(S^1\times S^2,\xi_\st)$ does not cancel with all the subcritical handles, and consequently the computation of the Legendrian contact homology is more elaborate. Still, we perform such computation and the degree zero part shall provide the algebra of functions of the algebraic mirror of the symplectic manifold $(X,\la,\p)$ as dictated by our previous discussion.

Let us in fact consider the more general class of Weinstein manifolds
$$(M^n_{b},\la,\p)=\left\{(x,y,\underline{z}):x(xy^b-1)+\sum_{i=1}^{n-1}z_i^2=0\right\}\sse\C^{n+1},$$
and prove the following theorem:

\begin{thm}\label{thm:Xabprime}
For any $b\geq1$, the Weinstein manifold
$$(M^n_{b},\la,\p)=\left\{(x,y,\underline{z}):x(xy^b-1)+\sum_{i=1}^{n-1}z_i^2=0\right\}\sse\C^{n+1},$$
has a Legendrian handlebody presentation as given in Figure \ref{fig:Mb front}.
\end{thm}

These Weinstein manifolds $(M^n_{b},\la,\p)$ are a variation on the Weinstein manifolds $(X^n_{1,b},\la,\p)$, and first appear in the work of P.~Seidel \cite{Se15}, to whom we are grateful for useful discussions on the symplectic topology of these manifolds and their relation to mirror symmetry. The variation consists in modifying the defining polynomial by changing the degree--$0$ constant $1$ to a linear term on $x$. The purpose of such an exchange is to introduce an additional critical value for the Lefschetz bifibration on the Weinstein bifiber in such a manner that all the vanishing cycles can be described by matching paths disjoint from that critical value. The eager reader can refer to \cite[Figure 5]{AKO} for a pictorial inception of our intentions.

\begin{proof}[Proof of Theorem \ref{thm:Xabprime}]
Consider the linear Lefschetz fibration
$$\pi:M^n_b\lr\C,\quad \pi(x,y,z)=x+by$$
with Weinstein fiber $(F_\pi,\la,\p)\cong (A^{2n-2}_b,\la_\st,\p_\st)$; projecting this fiber onto the complex plane via the auxiliary bifibration $\rho(x,z)=x$, we obtain that the critical points of the Lefschetz fibration $\rho$ are the $(b+1)$--roots of unity and the origin. In terms of a radial basis $\{R_0, R_1,\ldots, R_b\}$ depicted in Figure \ref{fig:Mb Lef}, the $(b+1)$ vanishing cycles $V$ of $\pi$ are described by the following collection of matching paths:
$$V_j=\tau_{R_{j+1}}(R_j),\mbox{ where }0\leq j\leq b,$$
where the indices are written modulo $b+1$. Figure \ref{fig:Mb Lef} shows this configuration of matchings paths in the cases $b=3$ and $4$.\\

\begin{figure}[h!]
\includegraphics[scale=0.6]{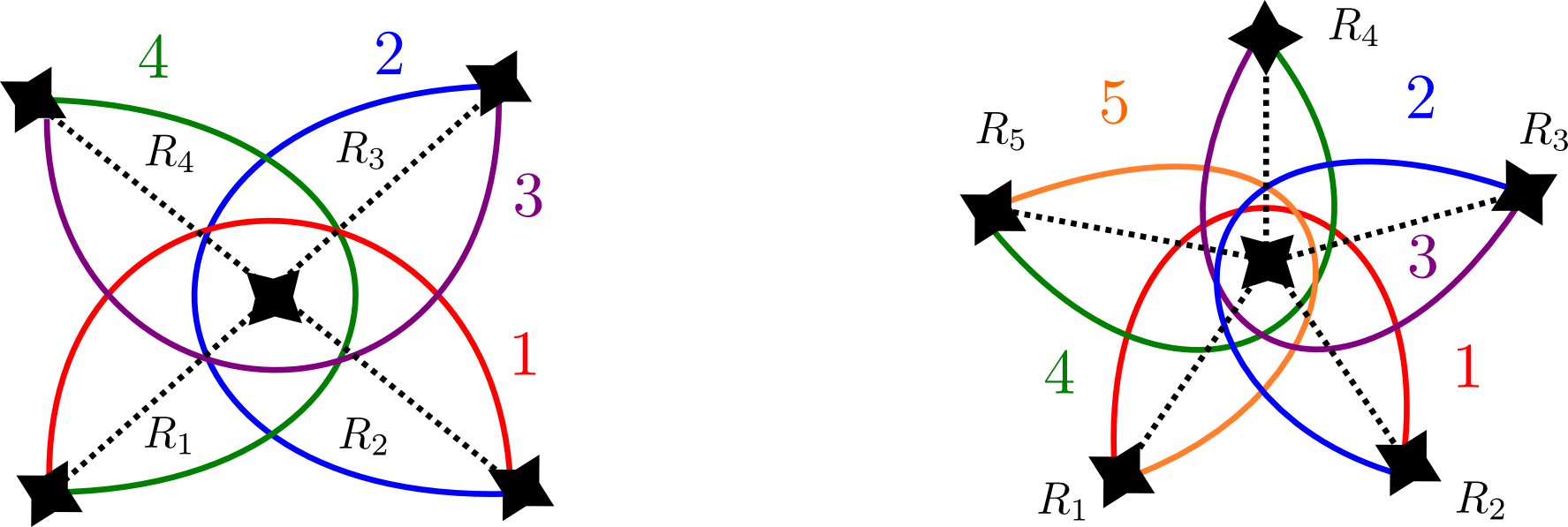}
\caption{A Lefschetz bifibration for $M^n_b$ for $b=3$ and $b=4$.}\label{fig:Mb Lef}
\end{figure}

\begin{figure}[h!]
\includegraphics[scale=0.6]{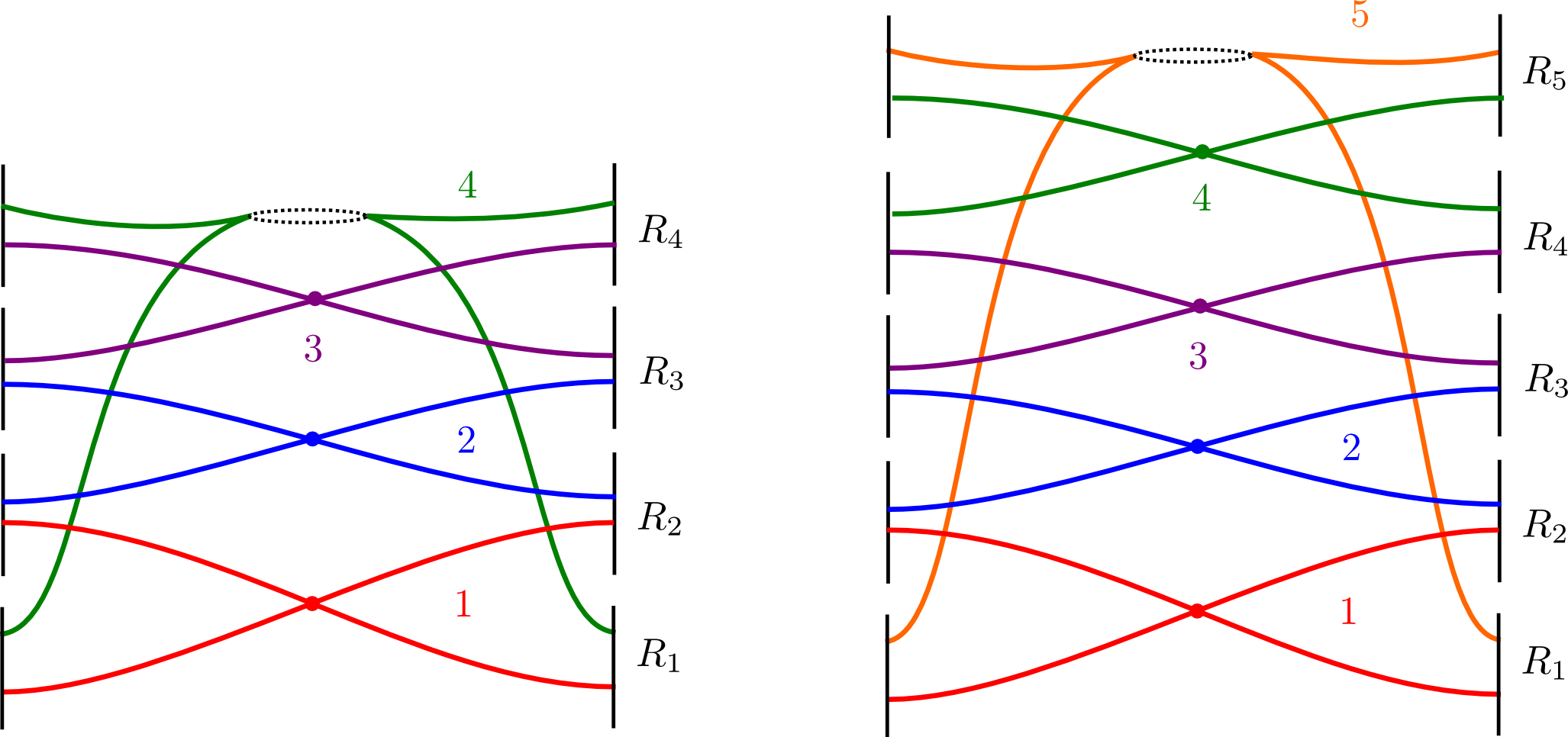}
\caption{A Weinstein handle diagram for $M^n_b$, for $b=3$ and $b=4$.}\label{fig:Mb front1}
\end{figure}

\begin{figure}[h!]
\includegraphics[scale=0.9]{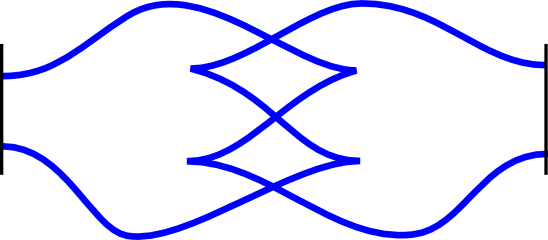}
\caption{A simplified Weinstein handle diagram for $M^n_b$, for $b=4$.}\label{fig:Mb front}
\end{figure}

We can then apply Proposition \ref{prop:stacking} and obtain the handle decomposition depicted in Figure \ref{fig:Mb front1} which, after cancelling all the subcritical handles except for the lowest, yields the Legendrian handlebody presented in Figure \ref{fig:Mb front}.
\end{proof}

Theorem \ref{thm:Xabprime} readily implies Theorem \ref{thm:torusIntro}:

\begin{proof}[Proof of Theorem \ref{thm:torusIntro}]
The Legendrian front depicted Figure \ref{fig:Mb front} can be build from a Legendrian unknot by performing two ambient Legendrian surgeries of indices $1$ and $(n-1)$, and thus Subsection \ref{ssec:lagcob} constructs the required exact Lagrangian $S^1\times S^n\sse (M^n_b,\la,\p)$.
\end{proof}

Let us now focus on the Weinstein $4$--manifold $(M^2_2,\la,\p)$ and conclude Theorem \ref{thm:mirror}. For that, we apply Recipe \ref{dictionary} as in the proof of Theorem \ref{thm:Xabprime} and obtain the Legendrian knot $\Lambda\sse(S^1\times S^2,\xi_\st)$ depicted in Figure \ref{fig:C2ConicFront}.

\begin{figure}[h!]
\includegraphics[scale=0.4]{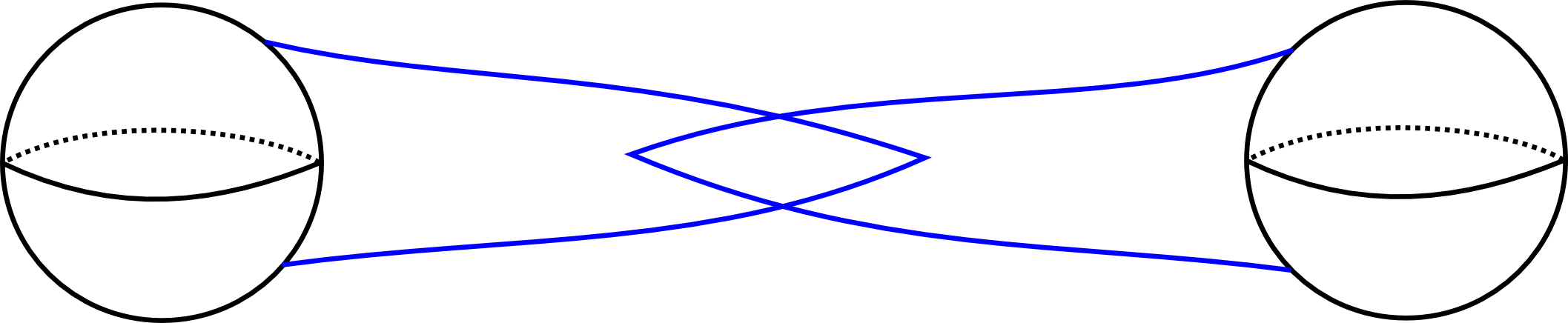}
\caption{Front diagram of $M_2^2=\{x(xy^2-1)=z^2\}\sse\C^2$.}\label{fig:C2ConicFront}
\end{figure}

This is a $\tb(\Lambda)=1$ Legendrian knot describing the Weinstein 4--fold $(X,\la,\p)=(M^2_2,\la,\p)$ from which we now compute the mirror affine surface $\check{X}$.\\

\begin{figure}[h!]
\includegraphics[scale=0.5]{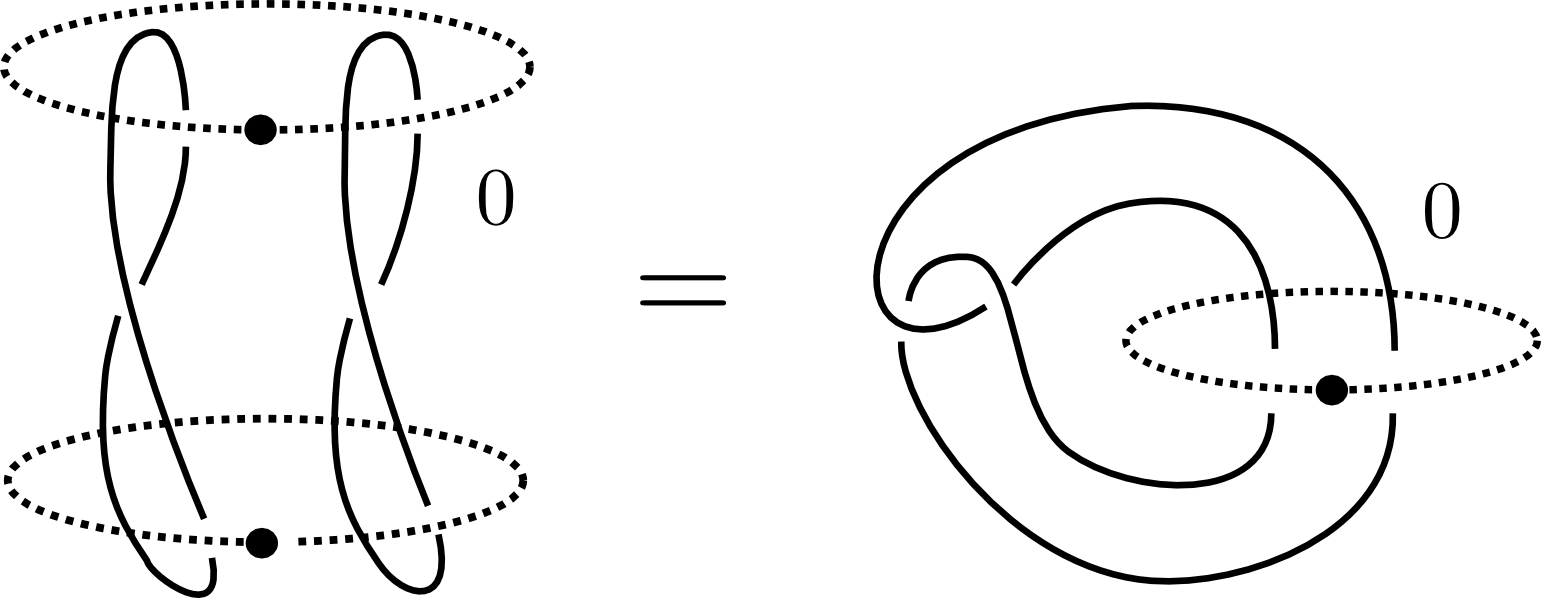}
\caption{Smooth handlebody for the complement of a smooth conic in $\R^4$: it is formed by two 1--handles and two 2--handles, respectively dual to the two 0--handles and 1--handle of the (affine conic) ribbon surface bounding the Hopf link $\{z_1^2+z_2^2=1\}\cap S^3\sse\C^2[z_1,z_2]$.}\label{fig:C2ConicKirby}
\end{figure}

\begin{remark}
Notice that the Weinstein 4--fold $(M^2_2,\la,\p)$ is the Stein complement of an affine smooth conic in $\C^2$ and thus, equivalently, the self--plumbing of $(T^*S^2,\la_\st,\p_\st)$. This can be readily seen in the smooth category from the Kirby diagrams in Figure \ref{fig:C2ConicKirby}.\hfill$\Box$
\end{remark}

It is known via the SYZ--duality \cite{SYZ} that the algebraic mirror of the symplectic complement of a smooth conic in $\C^2$ is the complement in $\P^2$ of the normal crossing divisor $\mathcal{O}(3)$ conformed by a projective conic and a projective line, i.e.~ the binodal cubic curve \cite{Au,Pa}; we will now recover this result from the Legendrian viewpoint. By the discussion above, this requires the computation of Legendrian invariants in order to construct the algebraic mirror variety $\check{X} = \op{Spec}(\mbox{LCH}_0(\Lambda))$. Thus to complete the proof of Theorem \ref{thm:mirror} it suffices the prove the following lemma.

\begin{lemma}\label{lem:LCH conic}
Let $\Lambda \sse (S^2 \x S^2, \xi_\st)$ be the Legendrian knot presented in Figure \ref{fig:C2ConicFront}. Then the Legendrian contact homology is zero in all negative degrees, and its degree $0$ part is the commutative algebra
$$LCH_0(\Lambda) \cong \C[x_1,x_2,(x_1x_2+1)^{-1}].$$
\end{lemma}

\begin{proof}
We compute the Legendrian contact homology differential graded algebra $(\A, d)$ using the results in \cite{EN}, which apply since $\Lambda$ is a Legendrian knot in the boundary of the subcritical Stein manifold $(D^3\times S^1,\la_\st,\p_\st)$. Normalize the Legendrian front in Figure \ref{fig:C2ConicFront} to the Lagrangian projection depicted in Figure \ref{fig:C2ConicLag}, where we have assigned a Maslov potential to each strand, $m(1)=2$ and $m(2)=1$, and drawn the external algebra generators $x_1,x_2,a_1$ and $a_2$ in dotted blue. Though it is not necessary for our purposes, we also mark a purple point carrying the homology class, which can be used for computing with homology coefficients. The internal algebra generators created by the periodic geodesic flow in the subcritical 1--handle are denoted by $\{c^p_{ij}\}$ as in \cite[Section 2.3]{EN}.

\begin{figure}[h!]
\includegraphics[scale=0.3]{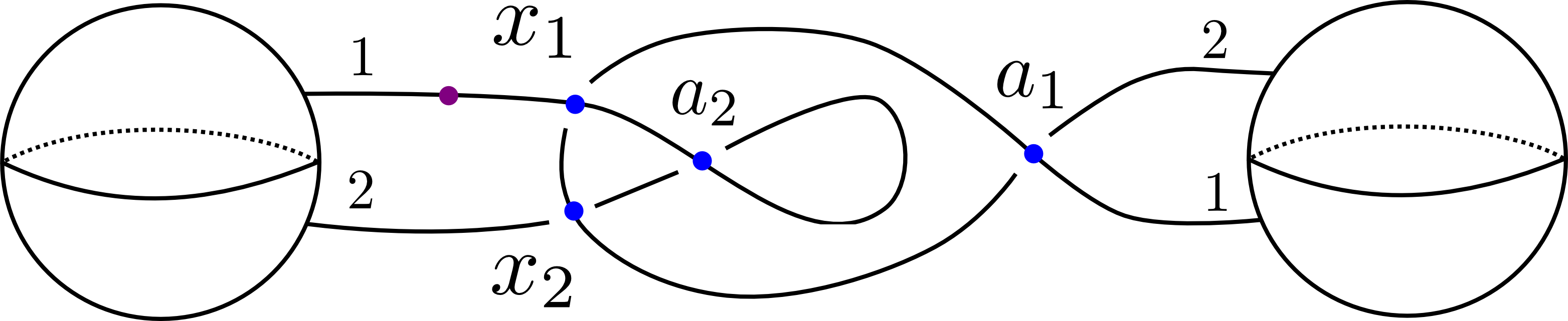}
\caption{Diagram of $\{x(xy^2-1)=z^2\}$}\label{fig:C2ConicLag}
\end{figure}

The internal algebra of $(\A, d)$ is the tensor algebra over the coefficient ring
$$\Z[H_1(\Lambda)]=\Z[t,t^{-1}]$$
generated by the variables $\{c ^0_{12},c^p_{12},c^p_{21},c^p_{11},c^p_{22}\}_{p\geq1}$. Since the rotation class $r(\Lambda)$ vanishes, we have a $\Z$--grading and the homological variables have degree $|t|=|t^{-1}|=0$. In consequence, generators of the algebra $\A$ in degree $0$ and $1$ are
$$|x_1|=|x_2|=|c^0_{12}|=|c^1_{21}|=|t|=|t^{-1}|=0,\qquad|a_1|=|a_2|=|c^1_{11}|=|c^1_{22}|=1.$$
The degrees of the remaining internal variables are
$$|c^p_{21}|=2p-2,\quad |c^p_{11}|=|c^p_{22}|=2p-1,\quad |c^p_{12}|=2p,$$
and we thus have two generators in the $k$th degree:
$$
\begin{cases}
    c^{(k+1)/2}_{11},\quad c^{(k+1)/2}_{22} & \quad \text{if $k$ is odd}\\
    c^{k/2}_{12},\quad c^{k/2+1}_{21}       & \quad \text{if $k$ is even}.\\
\end{cases}
$$
This describes the graded algebra. Since $\A$ is trivial in negative gradings, the elements in grading zero are closed, which implies that the subalgebra of exact grading zero elements forms an ideal inside of the grading zero part of $\A$. Thus to compute the exact elements it suffices to compute the differential of the four generators with grading one.
\begin{align*}
d(a_1) & = 1+c^0_{12}+x_2 x_1 \\
d(a_2)&=1+t^{-1}c^0_{12}+x_1 x_2\\
d(c^1_{11})&=1-c^0_{12}c^1_{21}\\
d(c^1_{22})&=1-c^1_{21}c^0_{12}.
\end{align*}
The higher graded terms of the differential for the internal variables are described in [Section 2.3]\cite{EN}, but since the internal algebra is standard beyond degree 2, we focus on the homology of $(\A_\La,d_\La)$ concentrated on degree $0$.
$$LCH_0(\Lambda;H_1\Lambda)=\frac{\langle x_1,x_2,c^1_{21},c^0_{12}\rangle}{\langle 1+c^0_{12}+x_2 x_1,1+t^{-1}c^0_{12}+x_1x_2,1-c^0_{12}c^1_{21},1-c^1_{21}c^0_{12}\rangle}.$$ 
Specializing at $t=1$ gives the algebra
$$LCH_0(\Lambda)=\frac{\langle x_1,x_2,c^1_{21},c^0_{12}\rangle}{\langle 1+c^0_{12}+x_2x_1,1+c^0_{12}+x_1x_2,1-c^0_{12}c^1_{21},1-c^1_{21}c^0_{12}\rangle}.$$
Subtracting the first two relations shows that $x_1$ and $x_2$ commute, the second two relations show that $c^1_{21} = (c^0_{12})^{-1}$. It follows that the algebra is commutative since $c^0_{12} = -1 - x_1x_2$, and furthermore this is equivalent to $-c^1_{21} = (x_1x_2 + 1)^{-1}$.
\end{proof}

\begin{remark}
In mirror symmetry, homology coefficients in the algebra --which are morally equivalent to nonexact deformations of the symplectic structure-- are related to non-commutative deformations of the mirror. The noncommutative algebra $LCH_0(\Lambda;H_1\Lambda)$ might have such an interpretation in these terms, but this remains to be explored.\hfill$\Box$
\end{remark}

\end{document}